\documentclass[11pt]{article}
\usepackage[utf8]{inputenc}
\usepackage{graphicx}
\usepackage[bookmarks]{hyperref}
\usepackage{amsmath, amsthm}
\usepackage{amsfonts, dsfont}
\usepackage{amssymb, color, xcolor, tcolorbox}
\usepackage{graphicx, verbatim}
\usepackage[caption=false]{subfig}
\usepackage{afterpage}
\usepackage[affil-it]{authblk}
\usepackage[left=2.5cm, right=3cm]{geometry}

\usepackage[comma]{natbib}
\newtheorem*{remark}{Remark}
\newtheorem{lemma}{Lemma}
\newtheorem{prop}{Proposition}
\newtheorem{theorem}{Theorem}

\tcbset{width=0.9\textwidth,boxrule=0pt,colback=red!30,arc=0pt,auto outer arc,left=0pt,right=0pt,boxsep=5pt}

\author{Anna Melnykova %
		\thanks{\texttt{anna.melnykova@u-cergy.fr}, Universit\'e de Cergy-Pontoise, AGM UMR-CNRS 8088 and Universit\'e de Grenoble Alpes, LJK UMR-CNRS 5224 } }

\title{Parametric inference for hypoelliptic diffusions with full observations}
\date{}

\begin{document}
\maketitle

\begin{abstract}
Multidimensional hypoelliptic diffusions arise naturally in different fields, for example to model neuronal activity. Estimation in those models is complex because of the degenerate structure of the diffusion coefficient. In this paper we consider hypoelliptic diffusions, given as a solution of two-dimensional stochastic differential equations (SDEs), with the discrete time observations of both coordinates being available on an interval $T = n\Delta_n$, with $\Delta_n$ the time step between the observations. The estimation is studied in the asymptotic setting, with  $T\to\infty$ as $\Delta_n\to 0$. We build a consistent estimator of the drift and variance parameters with the help of a discretized log-likelihood of the continuous process.  We discuss the difficulties generated by the hypoellipticity and provide a proof of the consistency and the asymptotic normality of the estimator. We test our approach numerically on the hypoelliptic FitzHugh-Nagumo model, which describes the firing mechanism of a neuron.
\end{abstract}

\textit{Keywords: } parametric inference, hypoelliptic diffusions, FitzHugh-Nagumo model, contrast estimator

\section{Introduction}\label{section:intro}

Hypoelliptic diffusions naturally occur in various applications, most notably in neuroscience, molecular physics and mathematical finance. In particular, neuronal activity of one single neuron \citep{Hoepfner2016, leon2018samson}, or a large population of neurons \citep{Ditlevsen2017eva, Ableidinger2017}, or exotic models of option pricing \citep{Malliavin2006} are described by hypoelliptic diffusions.

The main difference between classical (or \textit{elliptic}) and hypoelliptic systems of stochastic differential equations (SDE) is that in the latter case the rank of the diffusion matrix is lower than the dimension of the system itself. More formally, hypoellipticity can be explained in the following way: though the covariance matrix is singular, smooth transition densities with respect to the Lebesgue measure still exist. That is the case when the noise is propagated to all the coordinates through the drift term. 

Hypoelliptic SDEs present a number of extra challenges in comparison to elliptic systems. The most important problem is the degenerate diffusion coefficient. As the explicit form of the transition densities of a SDE is often unknown, parametric inference is usually based on  discrete approximation with a piece-wise Gaussian processes (see, for example \cite{Kessler1997}). But in the hypoelliptic case this approach cannot be applied directly because the covariance matrix of the approximated transition density is not invertible, since its rank is smaller than the dimension of the system. The second problem is that each coordinate has a variance of different order. It needs to be taken into account when constructing the discretization scheme for approximating the density. 

Now let us be more specific. Consider a two-dimensional system of stochastic differential equations of the form:
\begin{equation}\label{eq:system}
\begin{cases} 
dX_t = a_1(X_t, Y_t; \theta^{(1)}) dt \\
dY_t = a_2(X_t, Y_t; \theta^{(2)}) dt + b(X_t, Y_t; \sigma) dW_t, 
\end{cases}
\end{equation}
where $(X_t, Y_t)^T \in \mathds{R}\times\mathds{R}$, $(a_1(X_t, Y_t; \theta^{(1)}), \: a_2(X_t, Y_t; \theta^{(2)}))^T$ is the drift term, $ (0, b(X_t, Y_t; \sigma))^T$ is the diffusion coefficient, $W$ is a standard Brownian motion defined on some probability space $(\Omega, \mathcal{F}_t, \mathds{P})$, where $\mathcal{F}_t$ contains the information about all states of the process until time $t$. $(\theta^{(1)}, \theta^{(2)}, \sigma)$ is the vector of the unknown parameters, taken from some compact set $\Theta_1\times\Theta_2\times\Xi$, and $(x_0, y_0)$ is a bounded $\mathcal{F}_0$-measurable random variable, thus independent on $(X_t, Y_t)$. 

The goal of this paper is to estimate the parameters of \eqref{eq:system} from discrete observations of both coordinates $X$ and $Y$. It is achieved in two steps: first, we consider a discretization scheme in order to approximate the transition density of the continuous process. Then we propose an estimation technique which maximizes the likelihood function of the discrete approximate model  {in the asymptotic setting $T = n\Delta_n\to\infty$ and $\Delta_n\to 0$ as $n\to\infty$}. Let us discuss the solutions proposed by other authors for hypoelliptic systems. 

Several works treat the parametric inference problem for a particular case of system \eqref{eq:system}, the class of stochastic Damping Hamiltonian systems, also known as Langevin equations \citep{GardinetCollet1985}. These hypoelliptic models arise as the stochastic expansion of 2-dimensional deterministic dynamical systems --- for example, the Van der Pol oscillator \citep{VanderPol1920} perturbed by noise. They are defined as the solution of the following SDE:
\begin{equation}\label{eq:hamiltonian_system}
\begin{cases} 
dX_t = Y_t dt \\
dY_t = a_2(X_t, Y_t; \theta) dt + b(X_t, Y_t; \sigma) dW_t.  
\end{cases}
\end{equation}
The particular case of Hamiltonian systems with $b(X_t, Y_t; \sigma) \equiv \sigma $ and $ a_2(X_t, Y_t; \theta) = g_1(X_t; \theta)X_t + g_2(X_t; \theta)Y_t $  is considered in \cite{Ozaki1989}, where the link between the continuous-time solution of \eqref{eq:hamiltonian_system} and the corresponding discrete model is obtained with the so-called local linearization scheme. The idea of this scheme is the following: for a SDE with a non-constant drift and a constant variance, its solution can be interval-wise approximated by the solution of a system with a linear drift (see \cite{Biscay1996, Ozaki2012, Jimenez2015}). This scheme allows to construct a quasi Maximum Likelihood Estimator. Consistency of the estimator based on the local linearization scheme for Hamiltonian SDEs is proven in \cite{Leon2018}. 
\cite{Pokern2007} attempt to solve the problem of the non-invertibility of the covariance matrix for the particular case of \eqref{eq:hamiltonian_system} with a constant variance with the help of It\^{o}-Taylor expansion of the transition density. The parameters are estimated with the Gibbs sampler based on the discretized model with the noise propagated to the first coordinate with order $\Delta_n^\frac{3}{2}$. This approach allows to estimate the variance coefficient, but it is not  {suitable} for estimating the parameters of the drift term. 
In \cite{Samson2012} it is shown that a consistent estimator for fully and partially observed data can be constructed using only the discrete approximation of the second equation of system \eqref{eq:hamiltonian_system}. This method can be used in practice even for more general models, on condition that the system \eqref{eq:system} can be converted to the simpler form \eqref{eq:hamiltonian_system}. However, this transformation of the observations sampled from the continuous model \eqref{eq:system} often requires the prior knowledge of the parameters involved in the first equation which is often unrealistic.
The particular case of \eqref{eq:system}, when $b(X_t, Y_t; \sigma) \equiv \sigma$ and the drift term is linear and thus the transition density is known explicitly, is treated in \cite{LeBreton1985}. A consistent maximum likelihood estimator is then constructed in two steps --- first, a covariance matrix of the process is estimated from the available continuous-time observations, and then it is used for computing the parameters of the drift term.  The resulting estimator is strongly consistent as $T\to\infty$. Few works are devoted to the non-parametric estimation of the drift and the variance terms \citep{Cattiaux2014a,Cattiaux2016}.

To the best of our knowledge for systems \eqref{eq:system} the only reference is \cite{Ditlevsen2017}. They construct a consistent estimator using a discretization scheme based on a It\^{o}-Taylor expansion. To take into account different variance orders in each variable they construct two separate estimators for the rough and the smooth variables. However, this approach has several limitations. The first problem consists in minimizing two different criteria simultaneously, which is not very natural from a numerical point of view. The second problem is that in order to prove the convergence of the estimator for each variable, the parameters in the other variable need to be fixed to their true values.

In this paper, we want to avoid these limitations by proposing a single estimation criteria, able to estimate simultaneously all the parameters. This allows to prove the theoretical convergence of the vector of estimators, without any assumption on the knowledge of a set of parameters. Moreover, we illustrate that from a numerical point of view, the estimation of the first coordinate parameters are less biased than those obtained with approach \cite{Ditlevsen2017}. 
More precisely, we develop a new estimation method, adjusting the local linearization scheme described in \cite{Ozaki1989} developed for the models of type \eqref{eq:hamiltonian_system} to the more general class of SDEs \eqref{eq:system}. Under the hypoellipticity assumption this scheme propagates the noise to both coordinates of the system and allows to obtain an invertible covariance matrix. 
We start with describing the discretization scheme, proving the rate of convergence even when only one part of the parameters is fixed at the true value. 
Then we propose a contrast estimator based on the discretized log-likelihood, estimating the parameters included in the drift and diffusion coefficient simultaneously. 
Then we study the convergence of the scheme and prove the consistency and the asymptotic normality of the proposed estimator based on the 2-dimensional contrast.
To the best of our knowledge, the proof of this consistency is new in the literature. We finish with numerical experiments, testing the proposed approach on the hypoelliptic FitzHugh-Nagumo model and compare it to the other estimators. 

This paper is organized as follows: Section \ref{Section_model} presents the model and assumptions. Discrete model is introduced in Section \ref{Section_discrete_model}. The estimators are studied in Section \ref{subsection:estimation} and illustrated numerically in Section \ref{section_simulation}. We close with Section \ref{section:conclusions}, devoted to conclusions and discussions. Formal proofs are gathered in Appendix. 

\section{Models and assumptions}\label{Section_model}

 We assume that both variables of \eqref{eq:system} are discretely observed at equally spaced periods of time $\Delta_n$ on the time interval $[0, T]$. The vector of observations at time $i\Delta_n$ is denoted by $Z_i = (X_i, Y_i)^T $, where $Z_i$ is the value of the process at time $i\Delta_n, \: i\in \{0,\dots, n\}, \text{ where } n = \frac{T}{\Delta_n}$.  We further assume that it is possible to draw a sufficiently large and accurate sample of data, i.e that $T=n\Delta_n\to\infty$, with the partition size $\Delta_n\to 0$ as $n\to\infty$.  
Let us also introduce the vector notations:
\begin{equation}\label{eq:system_vector}
dZ_t = A(Z_t; \theta)dt + B(Z_t; \sigma) dW_t, \quad Z_0 = z_0, \quad t\in [0, T]
\end{equation}
where $Z_t = (X_t, Y_t)^T$, $W_t$ is a one-dimensional Brownian motion defined on the filtered probability space, $z_0 = (x_0, y_0)$, and $\theta=(\theta^{(1)}, \theta^{(2)})$ is the vector of drift parameters. Matrices $A$ and $B$ represent, respectively, the drift and the diffusion coefficient, that is $A(Z_t; \theta) = (a_1(X_t, Y_t; \theta^{(1)}), \: a_2(X_t, Y_t; \theta^{(2)}))^T$ and
\begin{equation}\label{diff_matrix}
B(Z_t; \sigma) = 
\left( \begin{matrix}
0 & 0 \\
0 & b(Z_t; \sigma)
\end{matrix} \right).
\end{equation}
Throughout the paper we use the following abbreviations for the partial derivatives (unless the arguments need to be specified): $\partial_{x_i} f \equiv \frac{\partial f}{\partial x_i}(x_1, \dots, x_p)$, $\partial^2_{x_i, x_j} \equiv \frac{\partial^2 f}{\partial x_i \partial x_j}(x_1, \dots, x_p) \: \forall i,j\in \{1,\dots, p\}$. We suppress the dependency on the parameters, when their values are clear from the context, otherwise additional indexes are introduced. True values of the parameters are denoted by $\theta^{(1)}_0, \theta^{(2)}_0, \sigma_0$ and $P_0$ is the probability $P_{\theta^{(1)}_0, \theta^{(2)}_0, \sigma_0}$. We also refer to the variable $Y_t$ which is directly driven by the Gaussian noise as "rough", and to $X_t$ as "smooth". 

We are working under the following set of assumptions:
\begin{itemize}
\item[\textbf{A1}] The functions $a_1(Z_t; \theta^{(1)})$, $a_2(Z_t; \theta^{(2)})$ and $b(Z_t; \sigma)$ have bounded partial derivatives of every order, uniformly in $\theta$ and $\sigma$ respectively. Furthermore $ \partial_y a_1 \neq 0\quad \forall (x, y) \in \mathds{R}^2$.
\item[\textbf{A2}] \textit{Global Lipschitz and linear growth conditions.} $\forall t, s \in [0, \infty) \: \exists K_\theta\: \text{s.t.}$: 
\begin{gather*}
\|A(Z_t; \theta) - A(Z_s; \theta) \| + \| B(Z_t; \sigma) - B(Z_s; \sigma) \| \leq  K_\theta \|Z_t - Z_s \| \\
\|A(Z_t; \theta) \|^2 + \| B(Z_t; \sigma) \|^2 \leq  K^2_\theta (1 + \|Z_t \|^2),
\end{gather*}
 where $\| \cdot \|$ is the standard Euclidean norm.
\item[\textbf{A3}] The process $(Z_t)_{t\geq 0}$ is ergodic and there exists a unique invariant probability measure $\nu_0$ with finite moments of any order. 
\item[\textbf{A4}] The functions $a_1(Z_t; \theta^{(1)})$, $a_2(Z_t; \theta^{(2)})$ and $b(Z_t; \sigma)$ are identifiable. By the identifiability we mean that $u(Z_t; \theta) \equiv u(Z_t; \theta_0) \Leftrightarrow \theta = \theta_0$. The diffusion coefficient is assumed to be strictly positive with a non-zero derivative with respect to $\sigma$, that is $b(Z_t; \sigma) > 0, \: \partial_\sigma b(Z_t; \sigma) \neq 0 \quad  \forall t$. 
\end{itemize}

Further, we introduce a rather restrictive assumption, which is required for the study of the consistency and asymptotic normality of the estimator for the parameters of the rough variable. 
\begin{itemize}
\item[\textbf{A5}] The fuction $a_1(Z_t; \theta^{(1)})$ can be represented in the following form:
\begin{equation}\label{eq:assumption_h1} 
a_1(z; \theta^{(1)})  = f(z) + (\theta^{(1)})^Tg(x),
\end{equation}
where $g(x)$ is a vector-valued function of the same dimension as vector $\theta^{(1)}$, $f(z)$ is a continuous function. Functions $f(z)$ and $g(x)$ are such that the assumptions (A1)-(A4) hold everywhere.  
\end{itemize}
The full force of the last assumption will be explained in Section \ref{subsection:contrast_estimator}, when the estimator is introduced. Representation \eqref{eq:assumption_h1} implies that the derivative $\partial_y a_1$ does not depend on the parameter $\theta^{(1)}$. It will be shown in Section \ref{Section_discrete_model} that the marginal variance of variable $X$ depends on $\partial_y a_1$. However, condition (A5) ensures that $\theta^{(1)}$ appears only in the drift, which simplifies the analysis. We note however that in practice the estimator shows good results even when $(A5)$ does not hold, as it will be shown in the simulation study.

Assumption (A1) ensures that the system is hypoelliptic in the sense of the stochastic calculus of variations \citep{Nualart2006a,Malliavin2006}. In order to prove it we first write the coefficients of the system \eqref{eq:system_vector} as two vector fields, converting \eqref{eq:system_vector} from the It{\^o} to the Stratonovich form:
\[
A_0(x,y) = \left(\begin{matrix} a_1(x, y; \theta^{(1)}) \\ a_2(x, y; \theta^{(2)}) - \frac{1}{2} b(x, y; \sigma) \partial_y b(x, y; \sigma) \end{matrix}\right) \qquad 
A_1(x, y) = \left(\begin{array}{c} 0 \\ b(x, y; \sigma) \end{array}  \right).
\]
Then their Lie bracket is equal to
\[
[A_0, A_1] = \left( \begin{matrix}  \partial_y a_1 (x, y; \theta^{(1)}) \\ \partial_x a_2(x, y; \theta^{(2)}) - \frac{1}{2}\partial_x b(x, y; \sigma) \partial^2_{xy} b(x, y; \sigma) \end{matrix} \right).
\]
By (A1) the first element of this vector is not equal to $0$, thus we conclude that $A_1$ and $[A_0, A_1]$ generate $\mathds{R}^2$. That means that the weak H\"ormander condition is satisfied and as a result the transition density for the system \eqref{eq:system_vector} exists, though not necessarily has an explicit form. 
 (A2) is a sufficient condition to ensure the existence and uniqueness in law of the strong solution of system \eqref{eq:system_vector}, moreover this solution is Feller \citep{revuz2013}. However, Feller property and the existence of the strong solution often hold under milder assumptions, thus (A2) can be relaxed.  (A4) is a standard condition which is needed to prove the consistency of the estimator. 
(A3) ensures that we can apply the weak ergodic theorem. That is, for any continuous function $f$ of polynomial growth at infinity:
\[
\frac{1}{T}\int_0^T f(Z_s) ds \underset{T\to\infty}{\longrightarrow} \nu_0(f) \quad \text{a.s.},
\]
 where $\nu_0(\cdot)$ is the stationary density of model \eqref{eq:system_vector}. By choosing this notation we highlight that $\nu_0(\cdot) := \nu_{\theta^{(1)}_0, \theta^{(2)}_0, \sigma_0}(\cdot)$. 

We do not investigate the conditions under which the process $(Z_t)_{t\geq 0}$ is ergodic as it is not the main focus of this work. Ergodicity of the stochastic damping Hamiltonian system \eqref{eq:hamiltonian_system} is studied in \cite{Wu2001}. Conditions for a wider class of hypoelliptic SDEs can be found in  \cite{Roynette1975a, Mattingly2002, Arnold1987}.
It is also important to know that if the process $(Z_t)_{t\geq 0}$ is ergodic then its sampling $\{Z_{i} \}, \: i\in\{0,\dots,n\} $ is also ergodic \citep{Genon-Catalot2000a}. 

\section{Discrete model}\label{Section_discrete_model}

In this section we introduce a Local Linearization scheme, which approximates the solution $Z_t$ of \eqref{eq:system_vector} by the solution of a piece-wise linear autonomous equation. This solution has a piece-wise Gaussian density. We use the approximated solution to construct a discretization scheme and study its properties. 

\subsection{Approximation with the Local Linearization scheme}
Local Linearization refers to the family of approximation schemes studied by different authors \citep{Biscay1996, Ozaki2012, Jimenez2015}. The idea consists in approximating the solution of a general SDE by the solution of an autonomous linear SDE, which can be solved explicitly. Before we proceed to the derivation of the scheme, let us introduce additional notations. The Jacobian of the drift vector $A(z; \theta)$ is given by 
\begin{equation}\label{eq:jacobian}
 \left(\begin{matrix} \partial_xa_1(x,y; \theta^{(1)}) & \partial_ya_1(x,y; \theta^{(1)}) \\
\partial_xa_2(x,y; \theta^{(2)}) & \partial_ya_2(x,y; \theta^{(2)}) \end{matrix}\right) =: J(z; \theta).
\end{equation} 
We also define the Hessian matrix of the $j-$th coordinate ($j=1,2$) in the drift vector $A(Z_t; \theta)$ as:
\begin{equation}\label{eq:hessian}
 \left(\begin{matrix} \partial^2_{xx}a_j(x,y; \theta^{(j)}) & \partial^2_{xy}a_j(x,y; \theta^{(j)})\\
\partial^2_{yx}a_j(x,y; \theta^{(j)}) & \partial^2_{yy}a_j(x,y; \theta^{(j)}) \end{matrix}\right) =: H_{a_j}(z; \theta^{(j)}).
\end{equation} 
For further use we also compute the following operator, which corresponds to the cross-term between the diffusion and drift in It\^o-Taylor-expansion for each coordinate: 
\[
Tr\left[B^T(z; \sigma) H_{a_j}(z; \theta^{(j)}) B(z; \sigma)\right] = b^2(z;\sigma)\partial^2_{yy}a_j(z; \theta^{(j)}).
\]
We now consider the It\^o-Taylor expansion of the drift term on the interval $t\in [i\Delta_n , (i+1)\Delta_n ]$: 
\[
A(z_t; \theta) \approx A(z_i; \theta) + J(z_i; \theta)(z_t - z_i)+\frac{(t-i\Delta_n)}{2} b^2(z_i;\sigma) \partial^2_{yy} A(z_i; \theta),
\]
where $\partial^2_{yy} A(z; \theta) := (\partial^2_{yy}a_1(z; \theta^{(1)}), \partial^2_{yy}a_2(z; \theta^{(2)}))^T $.

This transformation allows us to find an approximate solution of \eqref{eq:system_vector}. We introduce a new process $(\tilde{Z}_t)_{t\in [i\Delta_n , (i+1)\Delta_n ] }$ which is the solution of the following linear equation (see Section 5.6 in \cite{karatzas}):
\begin{multline*}
d\tilde{Z}_t = \left(A(\tilde{Z}_i; \theta) + J(\tilde{Z}_i; \theta)(\tilde{Z}_t - \tilde{Z}_i)+\frac{1}{2}b^2(\tilde{Z}_i;\sigma)\partial^2_{yy}A(\tilde{Z}_i; \theta)(t-i\Delta_n)  \right)dt + B(\tilde{Z}_i; \theta) dW_t .
\end{multline*}
The solution for the above equation is given for $t\in[i\Delta_n ,(i+1) \Delta_n ]$ by
\begin{multline}\label{eq:LAE_solution}
\tilde{Z}_t = \tilde{Z}_i + \int_{i\Delta_n}^{t}e^{J(\tilde{Z}_i; \theta)(t-s)} \left(A(\tilde{Z}_i; \theta)-J(\tilde{Z}_i; \theta) \tilde{Z}_i + \frac{1}{2}b^2(\tilde{Z}_i;\sigma)\partial^2_{yy}A(\tilde{Z}_i; \theta)(s-i\Delta_n)\right) ds + \\ \int_{i\Delta_n}^{t}e^{J(\tilde{Z}_i; \theta)(s-i\Delta_n)} B(\tilde{Z}_i; \theta)dW_s.
\end{multline}
Note that conditionally on $\tilde{Z}_i$, $\tilde{Z}_{i+1}$ is a normal variable, whose expectation and variance are given, respectively, by: 
\begin{align}
&\begin{aligned}
\mathds{E}\left[\tilde{Z}_{i+1}|\tilde{Z}_{i} \right] & = \tilde{Z}_i + \int_{i\Delta_n}^{(i+1)\Delta_n}e^{J(\tilde{Z}_i; \theta)((i+1)\Delta_n-s)} \left(A(\tilde{Z}_i; \theta)-J(\tilde{Z}_i; \theta) \tilde{Z}_i + \frac{1}{2}b^2(\tilde{Z}_i;\sigma)\partial^2_{yy}A(\tilde{Z}_i; \theta)(s-i\Delta_n)\right) ds
\end{aligned}, \label{eq:drift_approx_cont}\\ 
&\begin{aligned}
\Sigma\left[\tilde{Z}_{i+1}|\tilde{Z}_i\right] &= \mathds{E}\left[\left( \int_{i\Delta_n}^{(i+1)\Delta_n} e^{J(\tilde{Z}_i; \theta)((i+1)\Delta_n-s)}B(\tilde{Z}_i; \sigma) dW_s \right)\left( \int_{i\Delta_n}^{(i+1)\Delta_n} e^{J(\tilde{Z}_i; \theta)((i+1)\Delta_n-s)}B(\tilde{Z}_i;\sigma) dW_s \right)^T \right] . \label{eq:cov_mat_cont}
\end{aligned}
\end{align}
The approximation of the solution of \eqref{eq:system_vector} $(\tilde{Z}_i)_{i\geq 0}$ is then defined recursively as a sum of random variables with the mean and variance given by \eqref{eq:drift_approx_cont} and \eqref{eq:cov_mat_cont}. However, these expressions are not convenient for the numerical implementation because of the integrals and the matrix exponents. One possible solution is to rely on numerical integration algorithms when implementing the scheme. But we propose to simplify \eqref{eq:drift_approx_cont}-\eqref{eq:cov_mat_cont} in order to obtain the final scheme which is easier to implement and analyze from the theoretical point of view. We use the following propositions, whose proofs are postponed to appendix: 
\begin{prop}\label{prop:drift}
The component-wise drift approximation for \eqref{eq:drift_approx_cont} is given by:
\begin{equation*}
\mathds{E}\left[\tilde{Z}_{i+1}|\tilde{Z}_{i} \right] = \left(\begin{matrix} \bar{A}_1 \\ \bar{A}_2
\end{matrix}\right) + O(\Delta_n^3),
\end{equation*}
where $ \bar{A}_1$ and $ \bar{A}_2$ are given as follows:
\begin{equation}\label{eq:drift_approximation}
\begin{aligned}
 \bar{A}_1(\tilde Z_i; \theta^{(1)}, \theta^{(2)}, \sigma)  := \tilde X_i + \Delta_n a_1(\tilde{Z}_{i}; \theta^{(1)}) + & \\  \frac{\Delta_n^2}{2}\left(\frac{\partial a_1(\tilde{Z}_{i}; \theta^{(1)})}{\partial x} \right. & \left.  a_1(\tilde{Z}_{i}; \theta^{(1)}) + \frac{\partial a_1(\tilde{Z}_{i}; \theta^{(1)})}{\partial y}a_2( \tilde{Z}_{i}; \theta^{(2)})  \right) + \\ 
 \frac{\Delta_n^2}{4}b^2(\tilde{Z}_i;\sigma)&\partial^2_{yy}a_1(\tilde{Z}; \theta^{(1)}) \\
 \bar{A}_2(\tilde Z_i; \theta^{(1)}, \theta^{(2)}, \sigma) := \tilde Y_i + \Delta_n a_2(\tilde{Z}_{i}; \theta^{(2)}) + & \\ \frac{\Delta_n^2}{2}\left(\frac{\partial a_2(\tilde{Z}_{i}; \theta^{(2)})}{\partial x} \right. & \left. a_1(\tilde{Z}_{i}; \theta^{(1)})  + \frac{\partial a_2(\tilde{Z}_{i}; \theta^{(2)})}{\partial y}a_2(\tilde{Z}_{i}; \theta^{(2)})  \right) + \\ 
 \frac{\Delta_n^2}{4}b^2(\tilde{Z}_i;\sigma)&\partial^2_{yy}a_2(\tilde{Z}; \theta^{(2)}).
 \end{aligned}
\end{equation}
\end{prop}
\begin{prop}\label{prop:sigma}
The matrix $\Sigma[\tilde{Z}_{i+1}|\tilde{Z}_i]$ defined in \eqref{eq:cov_mat_cont} is approximated by: 
\begin{equation}\label{SIGMA_Taylor}
 b^2(\tilde{Z}_{i}; \sigma) \left(\begin{matrix} \left(\partial_y a_1  \right)^2 \frac{\Delta_n^3}{3} & \quad  (\partial_y a_1)  \frac{\Delta_n^2}{2} + (\partial_y a_1)(\partial_y a_2) \frac{\Delta_n^3}{3}  \\ (\partial_y a_1)   \frac{\Delta_n^2}{2} +(\partial_y a_1)(\partial_y a_2) \frac{\Delta_n^3}{3}  &  \quad\Delta_n + (\partial_y a_2)\frac{\Delta_n^2}{2} + (\partial_y a_2)^2 \frac{\Delta_n^3}{3}  \end{matrix} \right) +  {O}(\Delta_n^4),
\end{equation}
where the derivatives are computed at time $i\Delta_n$. 
\end{prop}
Both from the theoretical and computational points of view, it is enough to use only the lower-order terms of \eqref{SIGMA_Taylor}. Thus, we define:
\begin{equation}\label{eq:Sigma_Delta}
\Sigma_{\Delta_n}(\tilde{Z}_{i+1}; \theta, \sigma^2):=  b^2(\tilde{Z}_{i}; \sigma) \left(\begin{matrix} \left(\partial_y a_1 (\tilde{Z}_i; \theta^{(1)}) \right)^2 \frac{\Delta_n^3}{3} & \quad \partial_y a_1(\tilde{Z}_i; \theta^{(1)})  \frac{\Delta_n^2}{2}  \\ \partial_y a_1(\tilde{Z}_i; \theta^{(1)})   \frac{\Delta_n^2}{2}  & \quad \Delta_n  \end{matrix} \right).
\end{equation}
The inverse of \eqref{eq:Sigma_Delta} is defined by:
\begin{equation}\label{eq:Sigma_Delta_inverse}
\Sigma_{\Delta_n}^{-1}(\tilde{Z}_{i+1}; \theta, \sigma^2) = \frac{1}{b^2(\tilde{Z}_{i}; \sigma)} \left(\begin{matrix} \frac{12}{\left(\partial_y a_1(\tilde{Z}_i; \theta^{(1)})  \right)^2 \Delta_n^3} & \quad -\frac{6}{\partial_y a_1(\tilde{Z}_i; \theta^{(1)})\Delta_n^2} \\ -\frac{6}{\partial_y a_1(\tilde{Z}_i; \theta^{(1)}) \Delta_n^2} &  \quad \frac{4}{\Delta_n} \end{matrix}\right).
\end{equation}
Finally, the element-wise approximation of $\tilde{Z}_{i+1}$ conditionally on $\tilde{Z}_i$ is written as:
\begin{align}\label{eq:process_approximation}
\begin{split}
\tilde X_{i+1} &= \bar{A}_1(\tilde Z_i; \theta^{(1)}, \theta^{(2)}, \sigma) +  \xi_{1,i} \\
\tilde Y_{i+1} &=  \bar{A}_2(\tilde Z_i; \theta^{(1)}, \theta^{(2)}, \sigma) + \xi_{2,i},
\end{split}
\end{align}
where $(\xi_{1,i})$ and $(\xi_{2,i})$ are normal random sequences with zero means, independent in $i$, such that the covariance matrix of vector $(\xi_{1,i},\xi_{2,i})$ is given by \eqref{eq:Sigma_Delta}. Numerically they can be simulated by decomposing the matrix \eqref{eq:Sigma_Delta} with the help of the LU or Cholesky decomposition, i.e. any matrix $\bar{B}(Z_{i}; \theta, \sigma^2)$ such that $\bar{B}\bar{B}^T = \Sigma(Z_{i}; \theta, \sigma^2)$, and multiply it by a 2-dimensional vector whose entries are independent standard normal variables. The chosen method of the decomposition does not affect the theoretical properties of the scheme. 
Note that the approximated diffusion term now depends on the parameters of the drift term.  
It is proven that the approximated solution $\tilde{Z}$ converges weakly to the true solution $Z$ with order 2 (see Theorem 2 in \cite{Jimenez2015}). 

Now we want to study component-wise the moments of the obtained discretization, build on the observations of the process $(Z_t)_{t\geq 0}$.  We will rely on the result of the following Proposition (recall that the true value of the vector of parameters is denoted by $\theta_0$): 
\begin{prop}[Moments of the discretized process]\label{prop:bounds}
The following holds: 
\begin{align*}
\mathds{E}\left[X_{i+1} - \bar{A}_1(Z_i; \theta^{(1)}_0, \theta^{(2)}_0, \sigma_0) | Z_i \right] &=  {O}(\Delta_n^{3}) \\
\mathds{E}\left[Y_{i+1} - \bar{A}_2(Z_i; \theta^{(1)}_0, \theta^{(2)}_0, \sigma_0) | Z_i \right] &=  {O}(\Delta_n^{3}) \\
\mathds{E}\left[\left(X_{i+1} - \bar{A}_1(Z_i; \theta^{(1)}_0, \theta^{(2)}_0, \sigma_0)\right)^2 | Z_i \right] &= \left(\partial_y a_1\right)^2_{\theta^{(1)}_0} \frac{\Delta_n^3}{3} b^2(Z_i; \sigma_0) +  {O}(\Delta_n^4) \\ 
\mathds{E}\left[\left(Y_{i+1} - \bar{A}_2(Z_i; \theta^{(1)}_0, \theta^{(2)}_0, \sigma_0)\right)^2 | Z_i \right] &= \Delta_n b^2(Z_i; \sigma_0) +  {O}(\Delta_n^2) \\
\mathds{E}\left[\left(X_{i+1} - \bar{A}_1(Z_i; \theta^{(1)}_0, \theta^{(2)}_0, \sigma_0) \right)\left(Y_{i+1} - \bar{A}_2(Z_i; \theta^{(1)}_0, \theta^{(2)}_0, \sigma_0) \right) | Z_i \right] &= \left(\partial_y a_1\right)_{\theta^{(1)}_0} \frac{\Delta_n^2}{2} b^2(Z_i; \sigma_0) +  {O}(\Delta_n^3),
\end{align*}
where $\mathds{E}$ is taken under $\mathds{P}_{0}$ and the derivatives $\partial_y a_1$ are computed at time $i\Delta_n$.
\end{prop}
\begin{proof}
The moments of the Feller process can be approximated by its generator \citep{Kloeden2003}. That is, for a sufficiently smooth and integrable function $f: \mathds{R}\times\mathds{R} \to \mathds{R}$:
\begin{equation}\label{momentgen} 
\mathds{E}(f(Z_{t+\Delta_n})|Z_t = z) = \sum_{i=0}^j \frac{\Delta_n^i}{i!} L^if(z) +  {O}(\Delta_n^{j+1}),
\end{equation}
where $L^i f(z)$ is the $i$ times iterated generator of model \eqref{eq:system_vector} given by
\begin{equation*}
Lf(z) = (\partial_zf(z))A(z) + \frac{1}{2}\bigtriangledown_B^2f(z),
\end{equation*}
where $\bigtriangledown_B^2(\cdot) = b^2(z;\sigma)\frac{\partial^2}{\partial y^2}(\cdot)$ is a weighted Laplace type operator. 
Since the process is approximated by \eqref{eq:drift_approximation}, it coincides with \eqref{momentgen} up to the terms of order $\Delta_n^2$. 
\end{proof}
Further, we need an extension of Proposition \ref{prop:bounds} which gives the order of moments of the increments of the discrete process when parameters are fixed to their true values only partly. By doing that, we loose one order of accuracy in first two bounds, but the results for the variance remain unchanged. Note however that we cannot obtain the last three terms unless $\theta^{(1)} = \theta^{(1)}_0$. This is the main technical challenge to overcome when constructing an estimator. 
\begin{prop}\label{prop:bounds_part}
The following holds: 
\begin{align*}
(i)\quad & \mathds{E}\left[X_{i+1} - \bar{A}_1(Z_i; \theta^{(1)}_0, \theta^{(2)}, \sigma) | Z_i \right] =  {O}(\Delta_n^{2}) \\
(ii)\quad &\mathds{E}\left[Y_{i+1} - \bar{A}_2(Z_i; \theta^{(1)}, \theta^{(2)}_0, \sigma) | Z_i \right] =  {O}(\Delta_n^{2}) \\
(iii)\quad &\mathds{E}\left[\left(X_{i+1} - \bar{A}_1(Z_i; \theta^{(1)}_0, \theta^{(2)}, \sigma)\right)^2 | Z_i \right] = \left(\partial_y a_1\right)^2_{\theta^{(1)}_0} \frac{\Delta_n^3}{3} b^2(Z_i; \sigma) +  {O}(\Delta_n^4) \\ 
(iv)\quad &\mathds{E}\left[\left(Y_{i+1} - Y_i\right)^2 | Z_i \right] = \Delta_n b^2(Z_i; \sigma) +  {O}(\Delta_n^2) \\
(v)\quad &\mathds{E}\left[\left(X_{i+1} - \bar{A}_1(Z_i; \theta^{(1)}_0, \theta^{(2)}, \sigma) \right)\left(Y_{i+1} - Y_i \right) | Z_i \right] = \left(\partial_y a_1\right)_{\theta^{(1)}_0} \frac{\Delta_n^2}{2} b^2(Z_i; \sigma) +  {O}(\Delta_n^3),
\end{align*}
where $\mathds{E}$ is taken under $\mathds{P}_{0}$ and the derivatives $\partial_y a_1$ are computed at time $i\Delta_n$.
\end{prop}
\begin{proof}
We show the result for (i) and (iii). Start with (i):
\begin{multline*}
X_{i+1} - \bar{A}_1(Z_i; \theta^{(1)}_0, \theta^{(2)}, \sigma)= X_{i+1} - \bar{A}_1(Z_i; \theta^{(1)}_0, \theta^{(2)}_0, \sigma_0)  + \\ \bar{A}_1(Z_i; \theta^{(1)}_0, \theta^{(2)}_0, \sigma_0)  - \bar{A}_1(Z_i; \theta^{(1)}_0, \theta^{(2)}, \sigma) .
\end{multline*} 
The difference $ \mathds{E}\left[ X_{i+1} - \bar{A}_1(Z_i; \theta^{(1)}_0, \theta^{(2)}_0, \sigma_0) |Z_i \right] = O(\Delta^3_n)$ by Proposition \ref{prop:bounds} and the assumption (A2). It remains to consider the second part:
\begin{multline*}
\bar{A}_1(Z_i; \theta^{(1)}_0, \theta^{(2)}_0, \sigma_0)  - \bar{A}_1(Z_i; \theta^{(1)}_0, \theta^{(2)}, \sigma)  = \\
  \frac{\Delta_n^2}{2}\left( \frac{\partial a_1({Z}_{i}; \theta^{(1)}_0)}{\partial y}\left(a_2({Z}_{i}; \theta^{(2)}_0) - a_2( {Z}_{i}; \theta^{(2)})\right) + \frac{\partial^2_{yy}a_1({Z}; \theta^{(1)}_0)}{2} \left( b^2({Z}_i;\sigma_0)- b^2({Z}_i;\sigma) \right) \right).
\end{multline*}
Thus, $ \mathds{E}\left[ X_{i+1} - \bar{A}_1(Z_i; \theta^{(1)}_0, \theta^{(2)}_0, \sigma_0) |Z_i \right] = O(\Delta^2_n)$. 
Let us now consider (iii):
\begin{multline*}
\left(X_{i+1} - \bar{A}_1(Z_i; \theta^{(1)}_0, \theta^{(2)}, \sigma)\right)^2= \left(X_{i+1} - \bar{A}_1(Z_i; \theta^{(1)}_0, \theta^{(2)}_0, \sigma_0)\right)^2  + \\ \left(\bar{A}_1(Z_i; \theta^{(1)}_0, \theta^{(2)}_0, \sigma_0)  - \bar{A}_1(Z_i; \theta^{(1)}_0, \theta^{(2)}, \sigma) \right)^2 +\\ 2  \left(X_{i+1} - \bar{A}_1(Z_i; \theta^{(1)}_0, \theta^{(2)}_0, \sigma_0)\right)  \left(\bar{A}_1(Z_i; \theta^{(1)}_0, \theta^{(2)}_0, \sigma_0)  - \bar{A}_1(Z_i; \theta^{(1)}_0, \theta^{(2)}, \sigma) \right). 
\end{multline*}
Again, by Proposition \ref{prop:bounds} and previous computations we have the following:
\begin{gather*}
\mathds{E}\left[ \left(X_{i+1} - \bar{A}_1(Z_i; \theta^{(1)}_0, \theta^{(2)}_0, \sigma_0)\right)^2|Z_i \right] = \left(\partial_y a_1\right)^2_{\theta_0} \frac{\Delta_n^3}{3} b^2(Z_i; \sigma) + O(\Delta^4_n) \\
\mathds{E}\left[ \left(\bar{A}_1(Z_i; \theta^{(1)}_0, \theta^{(2)}_0, \sigma_0)  - \bar{A}_1(Z_i; \theta^{(1)}_0, \theta^{(2)}, \sigma)\right)^2 |Z_i \right] = O(\Delta^4_n)  \\ 
\mathds{E}\left[ 2  \left(X_{i+1} - \bar{A}_1(Z_i; \theta^{(1)}_0, \theta^{(2)}_0, \sigma_0)\right)  \left(\bar{A}_1(Z_i; \theta^{(1)}_0, \theta^{(2)}_0, \sigma_0)  - \bar{A}_1(Z_i; \theta^{(1)}_0, \theta^{(2)}, \sigma) \right)  |Z_i \right] = O(\Delta^5_n) .
\end{gather*}
The rest of proofs follows the same pattern. 
\end{proof}

\section{Parameter estimation}\label{subsection:estimation}

In this section we propose a contrast estimator based on the pseudo-likelihood function and prove its consistency and asymptotic normality. Then we discuss other already known results for the linear homogeneous SDEs (least squares estimator in particular) and show how it works in the general case. 

\subsection{Contrast estimator}\label{subsection:contrast_estimator}
Let us introduce the so-called contrast function for the system \eqref{eq:system_vector}. This function is defined as $-2$ times the log-likelihood of the discretized model \citep{Florens-Zmirou1989,Kessler1997}: 
\begin{multline}\label{loglikelihood}
\mathcal{L}_{n,\Delta_n}(\theta, \sigma^2; Z_{0:n}) = \frac{1}{2} \sum_{i=0}^{n - 1}(Z_{i+1} - \bar{A}(Z_i; \theta))^T\Sigma_{\Delta_n}^{-1}(Z_i; \theta, \sigma^2) (Z_{i+1} - \bar{A}(Z_i; \theta)) \\ + \sum_{i=0}^{n - 1}\log\det(\Sigma_{\Delta_n}(Z_i; \theta, \sigma^2)),
\end{multline} 
where the inverse matrix $\Sigma_{\Delta_n}^{-1}$ is given by \eqref{eq:Sigma_Delta_inverse}.
Then the local linearization (LL) estimator is defined as:
\begin{equation}\label{eq:estimator_original}
(\hat{\theta}_{n, \Delta_n}, \hat{\sigma}^2_{n,\Delta_n}) = \underset{\theta, \sigma^2}{\arg \min} \: \mathcal{L}_{n,\Delta_n}(\theta, \sigma^2;  Z_{0:n}),
\end{equation}
where $\hat{\theta}_{n, \Delta_n} = \left(\hat{\theta}^{(1)}_{n, \Delta_n}, \hat{\theta}^{(2)}_{n, \Delta_n}\right)$. 

Before proceeding to the proofs, let us explain how the contrast estimator works in the classical elliptic setting and give a roadmap for the proofs of consistency and asymptotic normality of the estimator \eqref{eq:estimator_original}, following \cite{Kessler1997}. The first notable difference between the estimator \eqref{eq:estimator_original} and the elliptic case is that in the elliptic case the estimation of the drift and the variance parameters can be separated. For example, the contrast estimator for a 1-dimensional SDE discretized with the Euler-Maruyama scheme is defined as follows:
\begin{equation}\label{eq:contrast_elliptic}
\left(\hat\theta, \hat\sigma^2\right) = \underset{\theta, \sigma^2}{\arg\min}\sum_{i=1}^n  \left(\frac{1}{2}\frac{\left(X_{i+1} - X_i -\Delta_n a(X_i; \theta)  \right)^2}{\Delta_n b^2(X_i; \sigma)} + \log b^2(X_i; \sigma)\right),
\end{equation}
where $a(x;\theta)$ is a drift term. Here the estimation of the parameter $\theta$ is independent of the value of $\sigma$, because the minimization of the criteria boils down to minimizing the expression $\left(X_{i+1} - X_i -\Delta_n a(X_i; \theta)  \right)^2$ and $\hat\theta$ converges to $\theta_0$ with a rate $\sqrt{n\Delta_n}$. For the variance term, the estimator of $\sigma$ converges independently of the value of $\theta$, because $\left(X_{i+1} - X_i -\Delta_n a(X_i; \theta)  \right)^2$ is of order $\Delta_n$ for any $\theta$ and it is enough to ensure the convergence of the variance parameter. The convergence rate for the variance is $\sqrt{n}$. This property is also shared by the estimator for the Hamiltonian SDE  proposed by \cite{Samson2012}.

In a general hypoelliptic setting the parametric inference is more complicated. First, the drift parameter $\theta$ is contained in the covariance matrix $\Sigma_{\Delta_n}$. Second, the variance of the first variable is of order $\Delta_n^3$, while for an arbitrary chosen vector of parameters $\theta^{(1)}$ the expression $(X_{i+1} - \bar{A}_1(Z_i; \theta^{(1)}, \theta^{(2)}, \sigma))^2$  is of order $\Delta_n^2$. It is not enough to show the convergence of the diffusion parameter in a standard way. From the practical point of view it means that if we launch the minimization algorithm on \eqref{eq:estimator_original} only with respect to $\theta^{(2)}$ and $\sigma^2$, it will not converge to the true value. The inverse, however, is possible: using the Proposition \ref{prop:bounds_part} the consistency result for $\hat\theta^{(1)}$ can be obtained without fixing $\theta^{(2)}$ and $ \sigma$.  

\cite{Ditlevsen2017} propose to overcome the problem of dependency between the estimators by separating the estimation of the rough and smooth variables. They introduce two separate contrasts, based on the approximate marginal distribution on each variable:
\begin{multline}\label{eq:DitSam_cont1}
\hat\theta^{(1)} = \underset{\theta^{(1)}}{\arg\min}\sum_{i=1}^{n-1} \left(\frac{3}{2}\frac{\left(X_{i+1} - \bar{A}_1(Z_i; \theta^{(1)}, \theta^{(2)}_0, \sigma_0)  \right)^2}{\Delta_n^3 \left(\partial_y a_1(Z_i; \theta^{(1)})\right)^2b^2(Z_i; \sigma_0)}+ \log \left(\partial_y a_1(Z_i; \theta^{(1)})b(Z_i; \sigma_0)\right)^2\right),
\end{multline}
\begin{equation}\label{eq:DitSam_cont2}
\left(\hat\theta^{(2)}, \hat\sigma^2\right) = \underset{\theta^{(2)}, \sigma^2}{\arg\min}\sum_{i=1}^{n-1} \frac{1}{2} \left(\frac{\left(Y_{i+1} - \bar{A}_2(Z_i; \theta^{(1)}_0, \theta^{(2)}, \sigma)  \right)^2}{\Delta_n b^2(Z_i; \sigma)} + \log b^2(Z_i; \sigma)\right).
\end{equation}
\\
The estimation is then conducted as follows: first, the parameters of the first equation are estimated from \eqref{eq:DitSam_cont1}. Estimator \eqref{eq:DitSam_cont1} is shown to converge with rate $\sqrt{\frac{n}{\Delta_n}}$. Since the parameters of the second equation are contained only in higher order terms, they are shown to have no impact on the convergence of the estimator. We are able to show that thanks to the Proposition \ref{prop:bounds_part}. Then, the obtained value $\hat\theta^{(1)}$ is plugged in \eqref{eq:DitSam_cont2}. The contrast is minimized with respect to $\sigma$ and $\theta^{(2)}$. It is proven that the estimator $\hat\theta^{(2)}$ converges with the rate $\sqrt{n\Delta_n}$ and $\hat\sigma^2$ --- with a rate $\sqrt{n}$. The rates are identical to the rates of convergence obtained in \cite{Kessler1997} for elliptic systems. The weak point of the scheme is that in order to prove the convergence of the estimator \eqref{eq:DitSam_cont2} the value of $\theta^{(1)}$ needs to be fixed to $\theta^{(1)}_0$.  

We choose a different approach and focus on the 2-dimensional contrast without splitting the numerical procedure in two parts. We still need to take into account the different rates of convergence and the eventual dependencies between the parameters. Thus, the proof of the consistency and asymptotic normality is splitted in two principal steps. The first step is a proof of the consistency and the convergence rate for $\hat\theta^{(1)}_{n,\Delta_n}$ (Theorem \ref{thm:consistency_and_asymptotics_theta1}). Except for the unusual convergence rate, the proof repeats the standard techniques from \cite{Kessler1997} and \cite{Ditlevsen2017}, adapted to the unknown value of $\theta^{(2)}$. The second step, however, is more intricate. As in \cite{Ditlevsen2017}, the estimators for $\theta^{(2)}$ and $\sigma^2$ do not converge for an arbitrary $\theta^{(1)}$. However, we prove that the consistency and the asymptotic normality still hold for $\hat\theta^{(2)}_{n,\Delta_n}$ and $\hat\sigma^2_{n,\Delta_n}$, because the sequence of estimators $\hat\theta^{(1)}_{n, \Delta_n}$ is tight and converges with rates proven in Theorem \ref{thm:consistency_and_asymptotics_theta1}. It is proven at the cost of an additional assumption (A5) on the function $a_1(Z_t; \theta^{(1)})$ in the drift term. 

We begin the study from the following Lemma, on which the consistency of $\hat\theta^{(1)}$ crucially relies:
\begin{lemma}\label{lemma:theta_1}
Under the assumptions (A1)-(A4), and assuming $\Delta_n \to 0$ and $n\Delta_n \to \infty$, the following holds:
\begin{multline*}
\lim_{n\to\infty, \Delta_n \to 0} \frac{\Delta_n}{n} \left[\mathcal{L}_{n,  {\Delta_n}}(\theta^{(1)}, \theta^{(2)}, \sigma^2; Z_{0:n}) - \mathcal{L}_{n,  {\Delta_n}}(\theta^{(1)}_0, \theta^{(2)}, \sigma^2; Z_{0:n}) \right]  \overset{\mathds{P}_0}{\longrightarrow} \\ 6 \int \frac{ (a_1(z; \theta^{(1)}_0) - a_1(z; \theta^{(1)}))^2}{b^2(z; \sigma)(\partial_ya_1)^2_\theta}  \nu_0(dz). 
\end{multline*}
\end{lemma}
Proof is postponed to Appendix. 
On the next step we obtain the consistency and the asymptotic normality of \eqref{eq:estimator_original} with respect to $\theta^{(1)}$:
\begin{theorem}\label{thm:consistency_and_asymptotics_theta1}
Under the assumptions (A1)-(A5), and assuming $\Delta_n \to 0$,  $n\Delta_n \to \infty$ and $n\Delta^2_n \to 0$, the following holds:
\[
\hat\theta^{(1)}_{n, \Delta_n} \overset{\mathds{P}_0}{\longrightarrow} \theta^{(1)}_0,
\]
\begin{multline*}
\sqrt{\frac{n}{\Delta_n}}(\hat\theta^{(1)}_{n, \Delta_n}  - \theta^{(1)}_0) \overset{\mathcal{D}}{\longrightarrow} \\ \mathcal{N}\left(0,\: 3\left(\int \frac{(\partial_{\theta^{(1)}}a_1)(\partial_{\theta^{(1)}}a_1)^T}{b^2(z;\sigma)(\partial_ya_1)} \nu_0(dz) \right)^{-1}\left(\int\frac{b^2(z;\sigma_0)}{b^4(z;\sigma)}(\partial_{\theta^{(1)}}a_1)(\partial_{\theta^{(1)}}a_1)^T\left(1+\frac{1}{(\partial_ya_1)^2}\right) \nu_0(dz)\right)  \right),
\end{multline*}
where $\partial_xa_1$ is a simplified notation for $\partial_xa_1(z; \theta^{(1)}_0)$
\end{theorem}
The asymptotic variance of the estimator slightly differs from the one obtained in \cite{Ditlevsen2017}. It is because the 2-dimensional estimator contains the cross-terms of type $(X_{i+1} - \bar{A}_1(Z_i; \theta^{(1)}, \theta^{(2)}, \sigma) )(Y_{i+1} -  \bar{A}_2(Z_i; \theta^{(1)}, \theta^{(2)}, \sigma)) $, not taken into account if the estimator is splitted in two separate contrasts for rough and smooth variables. The speed of convergence, however, stays the same. Notice also that the assumption (A5) is not used for Lemma \ref{lemma:theta_1}, on which proof of consistency relies. However, it is needed for the asymptotic normality. However, we do not need $\theta^{(2)}$ and $\sigma^2$ to be known, on the contrary to \cite{Ditlevsen2017}.

The idea of the proof of consistency for the diffusion and the rough term parameters follows \cite{Gloter2009}. Since we are working in a compact set, we can always find a sequence of estimators $\hat\theta^{(1)}_{n, \Delta_n} $ such that the sequence $(\hat\theta^{(1)}_{n,\Delta_n} - \theta^{(1)}_0 )$ is tight. Then we use the tightness in combination with the rate of convergence obtained in Theorem \ref{thm:consistency_and_asymptotics_theta1} and the continuous mapping theorem for proving the consistency of the remaining terms in a standard way. On this stage we need an additional assumption (A5). The reason for that is when the parameter $\theta^{(1)}$ is included in the derivative $\partial_ya_1$, this parameter is present both in the drift and the variance term, which substantially complicates the study.  
Also, assuming the linear shape of $a_1$ with respect to $\theta^{(1)}$, one can fully use the speed of convergence for $\hat\theta^{(1)}_{n,\Delta_n}$ obtained in Theorem \ref{thm:consistency_and_asymptotics_theta1}. It is rather restrictive, but the idea of the proof can be reused for a more general case. For example, consistency for the parameters of the rough variable can be obtained under the condition of Lipschitz continuity with respect to parameter $\theta^{(1)}$, at the cost of additional technicalities, which are omitted in this paper.
The consistency follows from the following Lemmas, on which Theorem \ref{thm:consistency_theta2_sigma} is based (proofs of both Lemmas and the Theorem are postponed to Appendix):
\begin{lemma}\label{lemma:theta_2}
Under assumptions (A1)-(A5), and assuming $\Delta_n \to 0$ and $n\Delta_n \to \infty$, the following holds:
\begin{multline*}
\lim_{n\to\infty, \Delta_n \to 0} \frac{1}{n\Delta_n} \left[\mathcal{L}_{n, {\Delta_n}}(\hat\theta^{(1)}_{n, \Delta_n} , \theta^{(2)}, \sigma^2; Z_{0:n}) - \mathcal{L}_{n,  {\Delta_n}}(\hat\theta^{(1)}_{n, \Delta_n} , \theta^{(2)}_0, \sigma^2; Z_{0:n})\right]  \overset{\mathds{P}_0}{\longrightarrow} \\ 2 \int \frac{ (a_2(z; \theta^{(2)}) - a_2(z; \theta^{(2)}_0))^2}{b^2(z; \sigma)}  \nu_0(dz)
\end{multline*}
\end{lemma}
\begin{lemma}\label{lemma:sigma}
Under assumptions (A1)-(A5), and assuming $\Delta_n \to 0$ while $n\Delta_n \to \infty$, the following holds:
\begin{equation*}
\frac{1}{n} \mathcal{L}_{n, {\Delta_n}}(\hat\theta^{(1)}_{n, \Delta_n}, \theta^{(2)}, \sigma^2; Z_{0:n}) \overset{\mathds{P}_0}{\longrightarrow} \\ \int \left( \frac{b^2(z;\sigma_0)}{b^2(z; \sigma)} + \log b^2(z; \sigma) \right) \nu_0(dz)
\end{equation*}
\end{lemma}
\begin{theorem}\label{thm:consistency_theta2_sigma}
Under assumptions (A1)-(A5), and assuming $\Delta_n\to 0$, $n\Delta_n \to \infty$ and $n\Delta^2_n \to 0$ the following holds: 
\begin{equation*}
\hat\theta^{(2)}_{n, \Delta_n} \overset{\mathds{P}_0}{\longrightarrow} \theta^{(2)}_0, \quad
\hat\sigma_{n,\Delta_n}  \overset{\mathds{P}_0}{\longrightarrow} \sigma_0
\end{equation*}
and
\begin{align*}
\sqrt{n\Delta_n}(\hat\theta^{(2)}_{n, \Delta_n} - \theta^{(2)}_0) & \overset{\mathcal{D}}{\longrightarrow} \mathcal{N}\left(0,\left(\int \frac{(\partial_{\theta^{(2)}}a_2(z;\theta^{(2)}_0))(\partial_{\theta^{(2)}}a_2(z;\theta^{(2)}_0))^T}{b^2(z,\sigma)} \nu_0(dz)\right)^{-1}  \right) \\ 
\sqrt{n}(\hat\sigma_{n,\Delta_n} - \sigma_0) & \overset{\mathcal{D}}{\longrightarrow} \mathcal{N}\left(0, 2 \left(\int \frac{(\partial_\sigma b(z,\sigma_0))(\partial_\sigma b(z,\sigma_0))^T }{b^2(z,\sigma_0)}\nu_0(dz) \right)^{-1} \right).
\end{align*}
\end{theorem}

The obtained rates coincide with the rates in \cite{Ditlevsen2017}, but with the advantage that we avoid fixing any of the parameters to their true value, instead we work with the estimated sequence  $\hat\theta^{(1)}_{n,\Delta_n}$. 

\subsection{Conditional least squares estimator}\label{subsection:lse}

For certain applications it is natural to split the estimation of the parameters in the diffusion coefficient and the drift term (see, for example, \cite{LeBreton1985}). First, it reduces the dimension of the optimization problem, and thus spares the computational cost. Second, it is easier to generalize the drift-based least square estimator to the high-dimensional hypoelliptic systems, when the approximation of the diffusion matrix is difficult to compute. 
The idea is to compute the least square estimator of the differences between the discrete observations of $(Z_t)_{t\geq 0}$ and the expectation of this process computed with the LL scheme. 
For system \eqref{eq:system_vector} however we should still be careful about the order of each difference. In order for the estimator to converge properly we need to renormalize the expression. We do that as follows: 
\begin{gather}\label{equation:drift_estimation_qv}
\hat\theta^{LSE}_{n, \Delta_n} := \left(\hat \theta^{(1), LSE}, \hat \theta^{(2), LSE}\right)^T
\end{gather}
where 
\begin{gather*}
\hat \theta^{ LSE,(1)}_{n, \Delta_n} := \underset{\theta^{(1)}}{\arg\min} \: \sum_{i=0}^{n-1}  \frac{ \left(X_{i+1} - \bar{A}_1(Z_i;\theta^{(1)}, \theta^{(2)}, \sigma)\right)^2}{\Delta^3_n}\\
\hat \theta^{LSE,(2)}_{n, \Delta_n} := \underset{\theta^{(2)}}{\arg\min} \: \sum_{i=0}^{n-1}  \frac{ \left(Y_{i+1} - \bar{A}_2(Z_i;\theta^{(1)}, \theta^{(2)}, \sigma)\right)^2}{\Delta_n},
\end{gather*}
where $\bar{A}_j(Z_i;\theta^{(1)}, \theta^{(2)}, \sigma), j = 1,2$ are defined in \eqref{eq:drift_approximation}. 
Using the same reasoning as for the LL contrast we prove the next Theorem (the proof is postponed to appendix): 
\begin{theorem}\label{thm:drift_contrast}
Under the assumptions (A1)-(A4) and the conditions $\Delta_n\to 0$, $n\Delta \to \infty$ and $n\Delta^2_n \to 0$ the following holds: 
\[
\hat\theta^{LSE}_{n, \Delta_n} \overset{\mathds{P}_0}{\longrightarrow} \theta_0,
\]
and
\begin{multline*}
\left(\begin{matrix} \sqrt{\frac{n}{\Delta_n}} (\hat\theta^{LSE, (1)}_{n, \Delta_n} - \theta_0) \\
\sqrt{n\Delta_n} (\hat\theta^{LSE, (2)}_{n, \Delta_n} - \theta_0) \end{matrix} \right)
\overset{\mathcal{D}}{\longrightarrow} 2 \mathcal{N}\left(0,I_{2}\cdot \left[\begin{matrix}  {1\over 3}\int b^2(z;\sigma_0) (\partial_ya_1(z; \theta^{(1)}_0))^2 (\partial_{\theta^{(1)}}a_1(z; \theta^{(1)}_0))^2  \nu_0(dz) \\  \int b^2(z;\sigma_0)  (\partial_{\theta^{(2)}}a_2(z; \theta^{(2)}_0))^2  \nu_0(dz)  \end{matrix} \right]  \right) ,
\end{multline*}
where  $I_2$ is a $2\times 2$ identity matrix, $\hat\theta^{LSE, (j)}_{n, \Delta_n}$ denote the $j-$th element of the vector $\mathcal{L}^{LSE}_{n,\Delta_n}(\theta;Z_{0:n})$ and
\[
C_i =  \int (\partial_{\theta^{(i)}}a_i(z; \theta^{(i)}_0))(\partial_{\theta^{(i)}}a_i(z; \theta^{(i)}_0))^T  \nu_0(dz)
\]
\end{theorem}
The advantage of this estimator over the LL contrast is that due to the absence of the cross-terms, the estimation of both parameters is independent. For instance, in Theorem \ref{thm:drift_contrast} we prove the consistency of the estimator with respect to $\theta^{(2)}$ without assumption (A5) and fixing $\theta^{(1)}$ to the estimated sequence $\hat\theta^{(1)}_{n,\Delta_n}$. Also, since the term $(\partial_ya_1(z; \theta^{(1)}_0))$ is not present in the variance, we do not need (A5) to obtain the asymptotic normality for the estimator of $\theta^{(1)}$. The asymptotic variance differs from that obtained in Theorems \ref{thm:consistency_and_asymptotics_theta1}-\ref{thm:consistency_theta2_sigma}. Since the terms $\partial_ya_1(z; \theta^{(1)}_0)$ and $b(x, y; \sigma)$ were not included in the normalization, they appear in the covariance matrix and influence the performance of the estimator. 
Thus, in comparison to LL estimator defined by \eqref{eq:estimator_original}, conditional least square estimator may perform worse and be prone to outliers when the diffusion coefficient and the value of $(\partial_ya_1(z; \theta^{(1)}_0))^2$ are large.

Note that the result of Theorem \ref{thm:drift_contrast} holds for any $\sigma$. However, the diffusion coefficient cannot be estimated from criteria \eqref{equation:drift_estimation_qv}. One possible way to estimate it would be to plug in the obtained drift parameters in the 2-dimensional criteria  \eqref{loglikelihood} or in the 1-dimensional criteria from \cite{Ditlevsen2017}, given by \eqref{eq:DitSam_cont2}.
Analogously, when the noise in SDE \eqref{eq:system} is additive (i.e., $b \equiv const$), or in a special case when $b(x, y; \sigma)\equiv\sigma f(x,y)$, the parameter $\sigma$ can be estimated explicitly with the help of the sample covariance matrix. The properties of this approach for the elliptic case are proven in \cite{Kessler1997, Jacod2011}. For hypoelliptic systems, this approach must be modified, as the discretization of order $\Delta_n$ does not allow to compute the terms of order $\Delta_n^3$, which represent the propagated noise. However, the value of $\sigma$ can still be inferred  from the observations of the rough coordinate by computing
\begin{equation}\label{estim_sigma_explicit}
\tilde{\sigma}^2_{n, \Delta_n} = \frac{1}{n\Delta_n} \sum_{i=0}^{n-1}\frac{(Y_{i+1} - Y_i)^2}{f^2(X_i,Y_i)}.
\end{equation}
It can be shown that this estimator is consistent and asymptotically normal. In fact, it is a straightforward consequence of point (iv) of Lemma \ref{lemma:convergence_in_dist} (see Appendix), but we do not aim to provide the details here as it only concerns the particular case of model \eqref{eq:system}, that is, when the diffusion term depends linearly on only one unknown parameter. However, we test the performance of the estimator \eqref{estim_sigma_explicit} in Section \ref{section_simulation}, devoted to the numerical experiments.

\section{Simulation study}\label{section_simulation}

\subsection{The model}\label{subsection:experiments_model}
The two estimators $(\hat\theta_{n, \Delta_n}, \hat{\sigma}^2_{n, \Delta_n})$ and $(\hat\theta^{LSE}_{n, \Delta_n}, \tilde{\sigma}^2_{n, \Delta_n})$ are evaluated on the simulation study with a hypoelliptic stochastic neuronal model called FitzHugh-Nagumo model \citep{Fitzhugh1961}. It is a simplified version of the Hodgkin-Huxley model \citep{Hodgkin1952}, which describes in a detailed manner activation and deactivation dynamics of a spiking neuron. First it was studied in the deterministic case, then in the stochastic elliptic setting with two sources of noise in both coordinates. However, it is often argued that only ion channels are perturbed by noise, while the membrane potential depends on them in a deterministic way. This idea leads to a 2-dimensional hypoelliptic diffusion. 
In this paper we consider a hypoelliptic SDE with noise only in the second coordinate as studied in \cite{leon2018samson}. 
More precisely, the behaviour of the neuron is defined through the solution of the system
\begin{equation}\label{FHN}
\begin{cases}
dX_t = \frac{1}{\varepsilon}(X_t - X_t^3-Y_t -s)dt \\
dY_t = (\gamma X_t - Y_t + \beta)dt+\sigma dW_t,
\end{cases}
\end{equation}
where the variable $X_t$ represents the membrane potential of the neuron at time $t$, and $Y_t$ is a recovery variable, which could represent the channel kinetic. The parameter $s$ is the magnitude of the stimulus current and is often known in experiments, $\varepsilon$ is a time scale parameter and is typically significantly smaller than $1$, since $X_t$ moves "faster" than $Y_t$.  Parameters to be estimated are $\theta = (\gamma, \beta, \varepsilon, \sigma)$. 
For system \eqref{FHN} we obtain the following expressions for $\bar{A}$ and $\Sigma_{\Delta_n}$, which we plug in \eqref{loglikelihood}:
\begin{equation*}
\bar{A}(Z_i;\theta) = \left(\begin{matrix}
X_i + {\Delta_n\over\varepsilon}(X_i-X_i^3-Y_i + s) + \frac{\Delta^2_n}{2\varepsilon}\left(\frac{(1-3X_i^2)}{\varepsilon}(X_i-X_i^3-Y_i + s) - (\gamma X_i-Y_i+\beta)\right) \\ 
Y_i + \Delta_n(\gamma X_i-Y_i+\beta) + \frac{\Delta^2_n}{2}\left(\frac{\gamma}{\varepsilon}(X_i-X_i^3-Y_i + s) - (\gamma X_i-Y_i+\beta)\right)
\end{matrix}\right)
\end{equation*}
\begin{equation*}
\Sigma_{\Delta_n} (Z_i;\theta, \sigma)  = \sigma^2 \left(\begin{matrix}  \frac{\Delta_n^3}{3\varepsilon^2} &  \frac{\Delta_n^2}{2\varepsilon}  \\  \frac{\Delta_n^2}{2\varepsilon}  & \Delta_n  \end{matrix} \right)
\end{equation*} 
Hypoellipticity and ergodicity of \eqref{FHN} are proven in \cite{leon2018samson}.  The same problem, but for the hypoelliptic setting is studied in \cite{Jensen2014, Ditlevsen2017}.

\subsection{Experimental design}\label{subsec:exp_design}

We consider two different settings: an excitatory and an oscillatory behaviour. For the first regime, the drift parameters are set to $ \gamma = 1.5, \: \beta = 0.3, \: \varepsilon = 0.1, \: s = 0.01$ and the diffusion coefficient $\sigma = 0.6$, and for the second  $\gamma = 1.2, \: \beta = 1.3, \: \varepsilon = 0.1, \: s = 0.01$ and $\sigma = 0.4$. The diffusion coefficient does not change the behaviour pattern, only the "noisiness" of the observations. The starting point is $(X_0, Y_0) = (0,0)$. Sample trajectories for both settings are shown on Figure \ref{eq:FitzHugh_traj}. 

We organize the trials as follows: first, we generate 100 trajectories using recursive formula \eqref{eq:process_approximation} for each set of parameters with $\Delta_n = 0.0001$ and $n = 500000$. The observed time interval is thus equal to $50$. Then we subsample the sequence so that we can vary the discretization step $\Delta_n$ and eventually truncate the observed time interval. We estimate the parameters by minimizing the contrast \eqref{loglikelihood}. We refer to this method as LL contrast. For the least square estimator (LSE) we do the following: we estimate the parameter $\sigma$ explicitly from the observations of the second variable by \eqref{estim_sigma_explicit}, and then compute the parameters of the drift by minimizing \eqref{equation:drift_estimation_qv}. 
In addition, we compare both methods to the 1.5 strong order scheme \citep{Ditlevsen2017}, based on two separate estimators for each coordinate, which are defined in \eqref{eq:DitSam_cont1} and \eqref{eq:DitSam_cont2}. 

The minimization of the criterions is conducted with the \texttt{optim} function in \textbf{R} with the Conjugate Gradient method. As the initial value of parameters we take $\theta_0 \pm U([0,1])$, where $U$ stays for the uniform probabilistic law. In Tables \ref{table:set1_delta_001}-\ref{table:set2_delta_001} we present the mean value of the estimated parameters and their standard deviation (in brackets), computed over 100 trajectories for each set of parameters. The reported value of $\sigma$ is obtained as $\sqrt{\sigma^2}$, since only $\sigma^2$ is identifiable. 
Figures \ref{FitzHugh_densities_1_Delta_001}-\ref{FitzHugh_densities_2_Delta_001} illustrate the estimation densities for $\Delta_n = 0.01$ and the interval of observations being fixed to $T=5$ or $T=50$. The LL contrast is depicted in blue, the least square estimator --- in red, the 1.5 scheme in green. 

The estimation of the diffusion coefficient $\sigma$ with the LL estimator is slightly biased in both sets of data. This bias does not appear in the one-dimensional criteria and when the value is directly computed from the observations as a mean empirical variance. The performance of the LL contrast improves when we reduce the step size and increase the observed time interval. However, when $\Delta_n$ becomes too small the performance of LL contrast with respect to $\sigma$ is worse than the one-dimensional estimators for $\sigma$ given by \eqref{eq:DitSam_cont2} and \eqref{estim_sigma_explicit}. It is slightly biased and its variance is bigger than that of LSE and 1.5 estimator. One possible explanation is that the estimation of $\sigma$ with the  LL contrast, as it is shown in Theorem \ref{thm:consistency_theta2_sigma}, depends heavily on the convergence of the parameters of the first coordinate. Minor inaccuracies in the estimation of the drift parameters lead to non-negligible errors in $\hat{\sigma}$. Note, for example, that the LL scheme scores better on interval $T=50$ for $\Delta_n = 0.01$ than for $\Delta_n = 0.001$ (see Table \ref{table:set1_delta_001}), while for the other schemes it is not the case. Thus, it is important to ensure that $n\to\infty$ faster than $\Delta_n\to 0$, as required by Theorem \ref{thm:consistency_theta2_sigma}. 

Parameters of the second coordinate $\gamma$ and $\beta$ are estimated accurately with all three methods once the time interval $T$ is big enough (see the bottom pictures on Figures \ref{FitzHugh_densities_1_Delta_001}-\ref{FitzHugh_densities_2_Delta_001} for $T=50$). However, when $T=5$, 1.5 scheme scores considerably worse than the LL and LSE estimator. Also when estimating $\varepsilon$, the 1-dimensional criteria \eqref{eq:DitSam_cont1} does not score better than the LL and LSE estimators. This parameter seems to be underestimated in the case of the 1.5 scheme, and this bias is bigger in the case of the inhibitory setting for $\Delta_n = 0.01$. The problems in the inhibitory setting are anticipated, since the trajectory is more erratic than in the excitatory case. Drift parameters are thus more difficult to estimate: the variance of the estimators is bigger in average.  Also, during the simulation study it is observed that $\varepsilon$ is the most sensitive to the initial value with which the \texttt{optim} function is initialized, since it directly regulates the amount of noise which is propagated to the first coordinate. However, as predicted by Theorems \ref{thm:consistency_and_asymptotics_theta1}-\ref{thm:consistency_theta2_sigma}, estimators for $\varepsilon$ converge indeed faster than for the rest of the parameters. 

\afterpage{%
\thispagestyle{empty}
\begin{figure}[p!]
\begin{center}
\subfloat[Excitatory set]{\includegraphics[height=0.45\textheight, width = 0.9\textwidth]{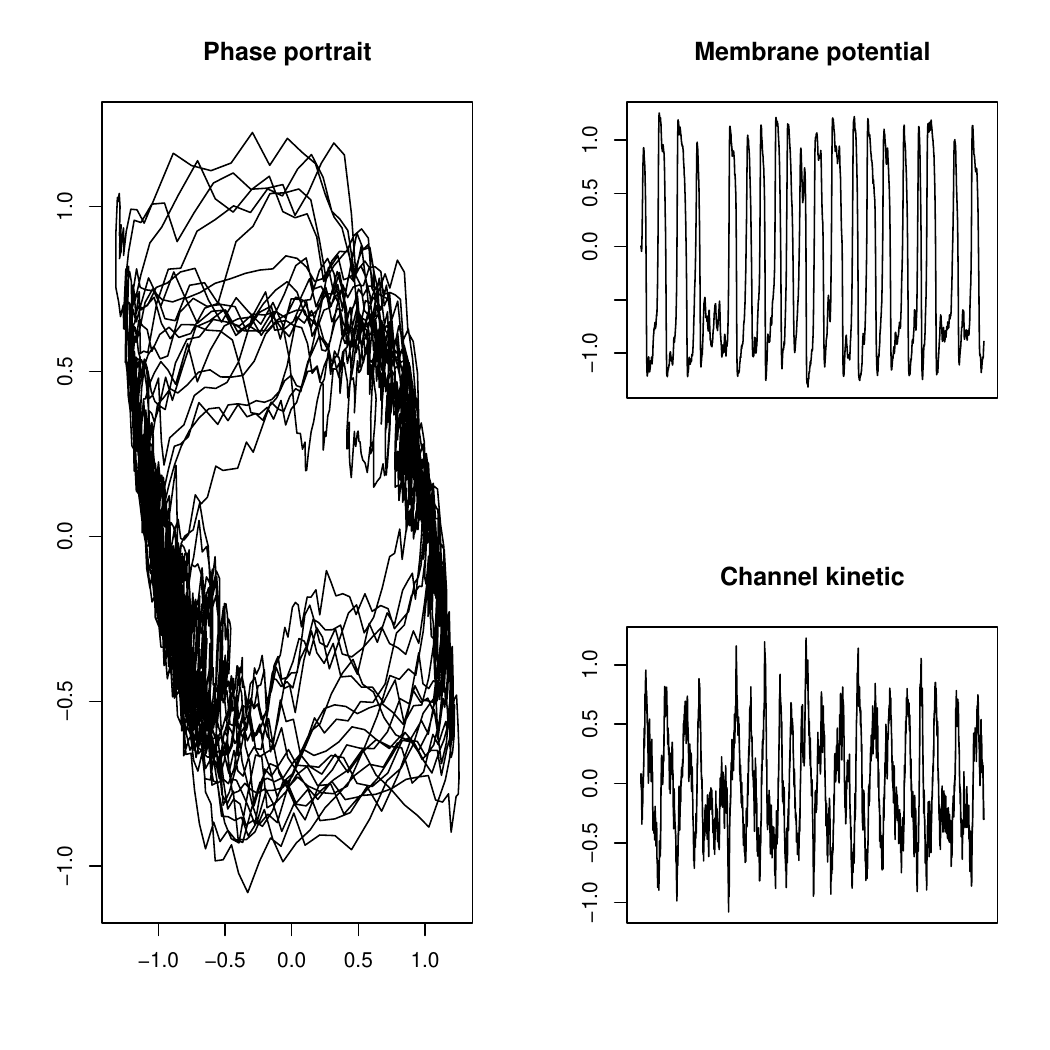}}

\subfloat[Oscillatory set]{\includegraphics[height=0.45\textheight, width = 0.9\textwidth]{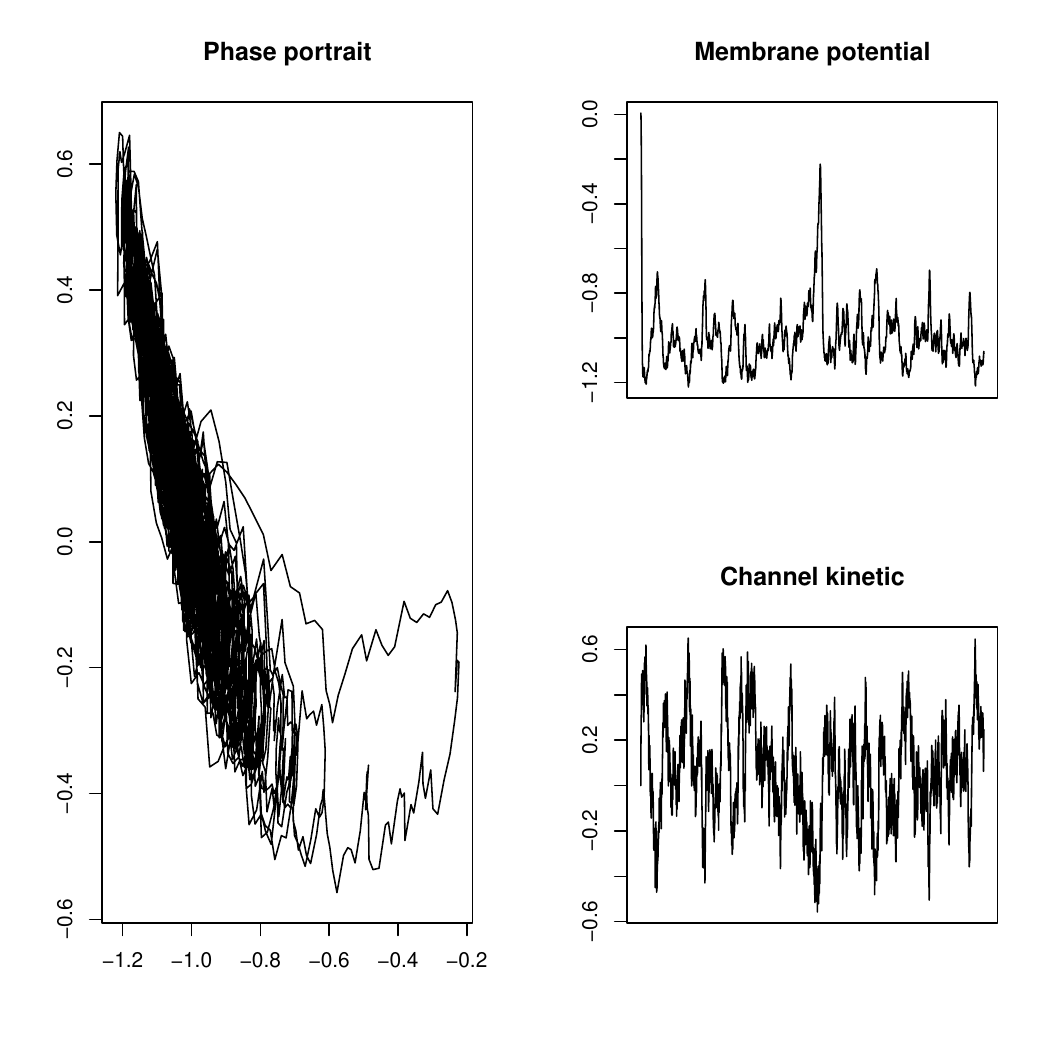}}
\end{center}
\caption{Trajectories for two sets of parameters}\label{eq:FitzHugh_traj}
\end{figure}
 \clearpage
}

\begin{table}[ht]
\centering
\begin{tabular}{rllll}
  \hline
$\Delta_n=0.01$, T = 5 & $\gamma$ & $\beta$ & $\varepsilon$ & $\sigma$ \\ 
  \hline
LC & 1.501 (0.053) & 0.302 (0.055) & 0.101 (0.001) & 0.592 (0.056) \\ 
  LSE &  1.488 (0.108) & 0.311 (0.149) & 0.100 (0.000) & 0.612 (0.020) \\ 
 1.5 scheme & 1.561 (0.362) & 0.324 (0.295) & 0.099 (0.000) & 0.598 (0.019) \\ 
   \hline
\end{tabular}
\begin{tabular}{rllll}
  \hline
$\Delta_n=0.01$, T = 10 & $\gamma$ & $\beta$ & $\varepsilon$ & $\sigma$ \\ 
  \hline
LC& 1.504 (0.055) & 0.306 (0.053) & 0.100 (0.001) & 0.562 (0.026) \\ 
  LSE & 1.503 (0.069) & 0.299 (0.176) & 0.100 (0.000) & 0.610 (0.014) \\ 
 1.5 scheme & 1.540 (0.237) & 0.301 (0.212) & 0.099 (0.000) & 0.596 (0.013) \\ 
   \hline
\end{tabular}
\begin{tabular}{rllll}
  \hline
$\Delta_n=0.01$, T = 50 & $\gamma$ & $\beta$ & $\varepsilon$ & $\sigma$ \\ 
  \hline
LC & 1.500 (0.050) & 0.297 (0.052) & 0.100 (0.000) & 0.560 (0.018) \\ 
  LSE & 1.513 (0.072) & 0.302 (0.068) & 0.100 (0.000) & 0.610 (0.007) \\ 
 1.5 scheme & 1.495 (0.095) & 0.301 (0.093) & 0.099 (0.000) & 0.596 (0.007) \\ 
   \hline
\end{tabular}

\vspace{1.5em}

\begin{tabular}{rllll}
  \hline
$\Delta_n=0.001$, T = 5 & $\gamma$ & $\beta$ & $\varepsilon$ & $\sigma$ \\ 
  \hline
LC & 1.505 (0.054) & 0.306 (0.051) & 0.100 (0.000) & 0.699 (0.090) \\ 
  LSE &  1.498 (0.062) & 0.290 (0.072) & -47.86 (477.2) & 0.599 (0.005) \\ 
 1.5 scheme & 1.497 (0.183) & 0.304 (0.169) & 0.100 (0.000) & 0.598 (0.005) \\ 
   \hline
\end{tabular}
\begin{tabular}{rllll}
  \hline
$\Delta_n=0.001$, T = 10 & $\gamma$ & $\beta$ & $\varepsilon$ & $\sigma$ \\ 
  \hline
LC& 1.513 (0.049) & 0.302 (0.054) & 0.100 (0.000) & 0.662 (0.096) \\ 
  LSE& 1.501 (0.051) & 0.299 (0.052) & 0.100 (0.000) & 0.600 (0.004) \\ 
 1.5 scheme & 1.513 (0.159) & 0.288 (0.161) & 0.100 (0.000) & 0.599 (0.004) \\ 
   \hline
\end{tabular}
\begin{tabular}{rllll}
  \hline
$\Delta_n=0.001$, T = 50 &$\gamma$ & $\beta$ & $\varepsilon$ & $\sigma$ \\ 
  \hline
LC & 1.487 (0.054) & 0.303 (0.050) & 0.100 (0.000) & 0.628 (0.098) \\ 
  LSE & 1.493 (0.056) & 0.303 (0.052) & 0.100 (0.000) & 0.601 (0.002) \\ 
 1.5 scheme  & 1.488 (0.066) & 0.302 (0.068) & 0.100 (0.000) & 0.600 (0.002) \\ 
   \hline
\end{tabular}
\caption{Set 1, $\gamma_0 = 1.5, \beta_0 = 0.3, \varepsilon_0 = 0.1, \sigma_0 = 0.6$.. Value without brackets: mean, value in parentheses: standard deviation.} \label{table:set1_delta_001}
\end{table}

\afterpage{%
\begin{table}[h!]
\centering
\begin{tabular}{rllll}
  \hline
$\Delta_n=0.01$, T = 5  & $\gamma$ & $\beta$ & $\varepsilon$ & $\sigma$ \\ 
  \hline
LC& 1.205 (0.046) & 1.311 (0.053) & 0.100 (0.001) & 0.357 (0.013) \\ 
LSE & 1.243 (0.771) & 1.592 (0.887) & 0.101 (0.002) & 0.400 (0.014) \\ 
1.5 scheme & 1.324 (0.357) & 1.415 (0.365) & 0.095 (0.002) & 0.397 (0.014) \\ 
\hline
\end{tabular}
\begin{tabular}{rllll}
  \hline
$\Delta_n=0.01$, T = 10  &$\gamma$ & $\beta$ & $\varepsilon$ & $\sigma$ \\ 
  \hline
LC & 1.201 (0.053) & 1.303 (0.053) & 0.100 (0.001) & 0.356 (0.008) \\ 
LSE & 1.251 (0.367) & 1.507 (0.521) & 0.100 (0.001) & 0.399 (0.009) \\ 
1.5 scheme & 1.260 (0.187) & 1.354 (0.188) & 0.091 (0.003) & 0.396 (0.009) \\ 
   \hline
\end{tabular}
\begin{tabular}{rllll}
  \hline
$\Delta_n=0.01$, T = 50 & $\gamma$ & $\beta$ & $\varepsilon$ & $\sigma$ \\ 
  \hline
LC & 1.200 (0.046) & 1.302 (0.048) & 0.101 (0.001) & 0.357 (0.004) \\ 
LSE &  1.207 (0.208) & 1.374 (0.288) & 0.100 (0.001) & 0.400 (0.004) \\ 
1.5 scheme & 1.217 (0.073) & 1.304 (0.075) & 0.083 (0.009) & 0.398 (0.004) \\ 
   \hline
\end{tabular}

\vspace{1.5em}

\begin{tabular}{rllll}
  \hline
$\Delta_n=0.001$, T = 5 & $\gamma$ & $\beta$ & $\varepsilon$ & $\sigma$ \\ 
  \hline
LC& 1.206 (0.052) & 1.302 (0.050) & 0.100 (0.000) & 0.370 (0.052) \\ 
  LSE & 1.183 (0.074) & 1.330 (0.126) & 0.100 (0.000) & 0.400 (0.004) \\
1.5 scheme & 1.239 (0.170) & 1.327 (0.177) & 0.100 (0.000) & 0.400 (0.004) \\ 
   \hline
\end{tabular}
\begin{tabular}{rllll}
  \hline
$\Delta_n=0.001$, T = 10 & $\gamma$ & $\beta$ & $\varepsilon$ & $\sigma$ \\ 
  \hline
LC & 1.193 (0.050) & 1.303 (0.050) & 0.100 (0.000) & 0.345 (0.013) \\ 
  LSE &  1.183 (0.069) & 1.328 (0.101) & 0.100 (0.000) & 0.400 (0.003) \\ 
1.5 scheme & 1.231 (0.126) & 1.328 (0.114) & 0.099 (0.000) & 0.400 (0.003) \\ 
   \hline
\end{tabular}
\begin{tabular}{rllll}
  \hline
$\Delta_n=0.001$, T = 50 &  $\gamma$ & $\beta$ & $\varepsilon$ & $\sigma$ \\ 
  \hline
LC& 1.201 (0.052) & 1.301 (0.053) & 0.100 (0.000) & 0.344 (0.009) \\ 
  LSE & 1.207 (0.208) & 1.374 (0.288) & 0.100 (0.001) & 0.400 (0.004) \\ 
1.5 scheme & 1.206 (0.088) & 1.295 (0.084) & 0.099 (0.000) & 0.400 (0.001) \\ 
   \hline
\end{tabular}
\caption{Set 2: $\gamma_0 = 1.2, \beta_0 = 1.3, \varepsilon_0 = 0.1, \sigma_0 = 0.4$. Value without brackets: mean, value in parentheses: standard deviation.} \label{table:set2_delta_001} 
\end{table}
}

\afterpage{%
\thispagestyle{empty}
\begin{figure}[p!]
\begin{center}
\subfloat[$T = 5$]{\includegraphics[height=0.47\textheight]{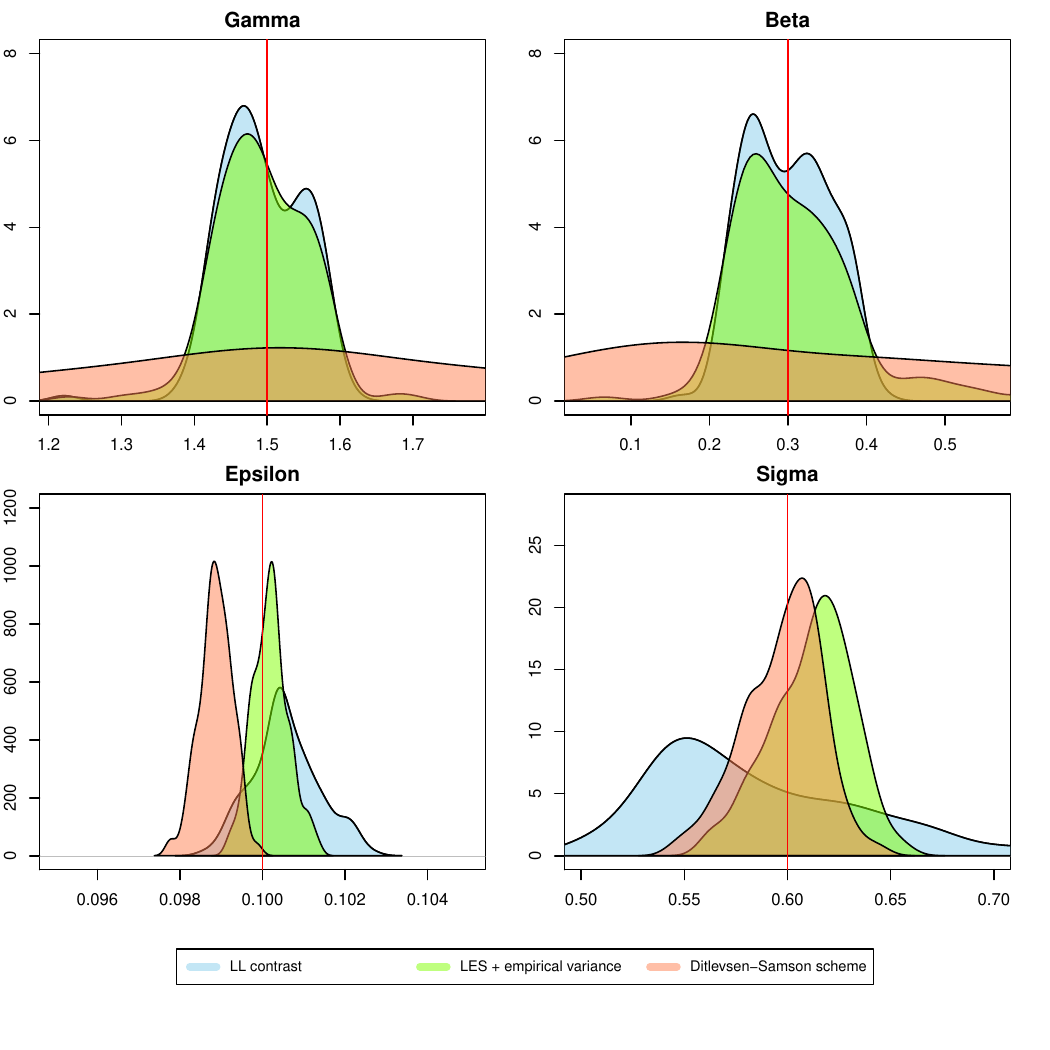}}

\subfloat[$T = 50$]{\includegraphics[height=0.47\textheight]{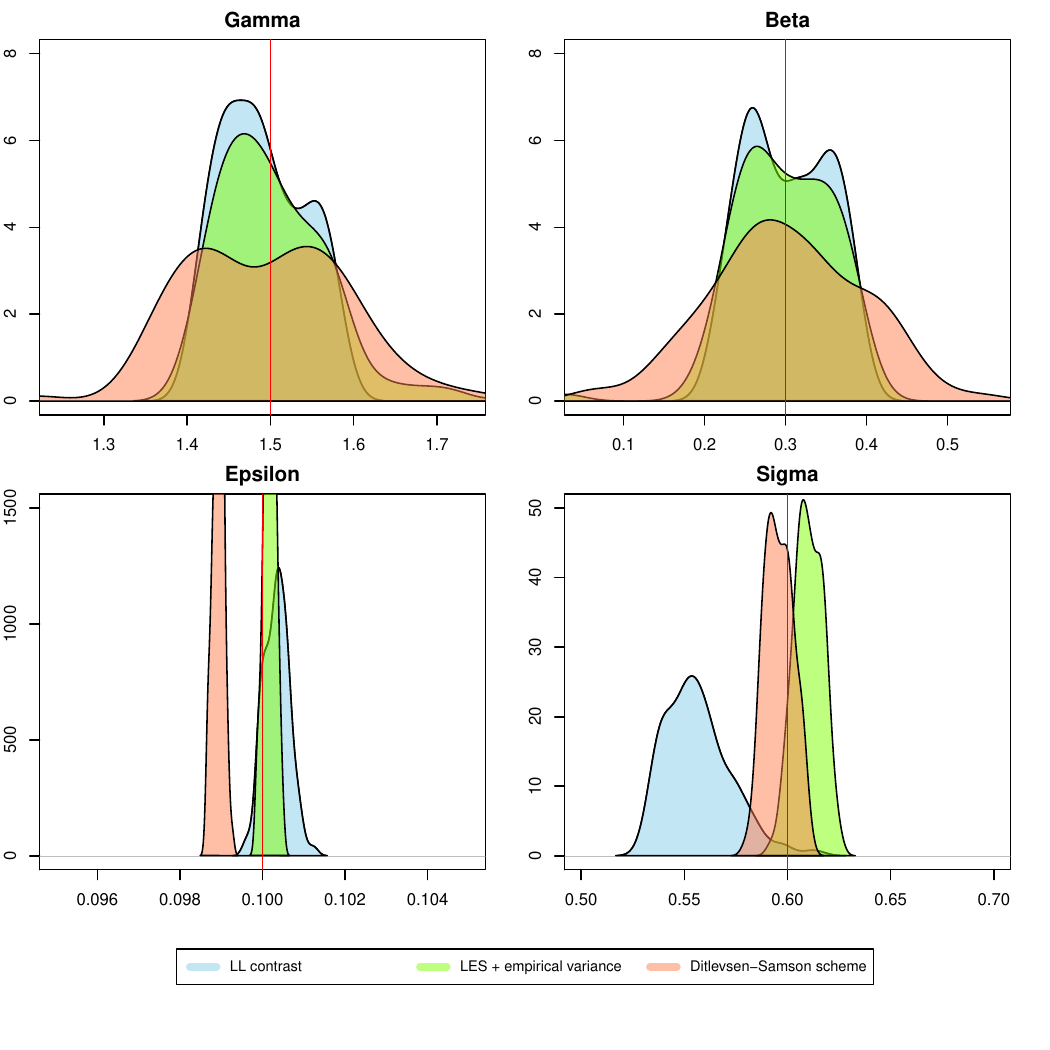}}
\end{center}
\caption{Estimation density for the LL contrast (blue), the LSE (red) and 1.5 scheme (green) estimators for the excitatory set. $\Delta_n = 0.01$}\label{FitzHugh_densities_1_Delta_001}
\end{figure}
\clearpage
}

\afterpage{%
\thispagestyle{empty}
\begin{figure}[p!]
\begin{center}
\subfloat[$T = 5$]{\includegraphics[height=0.47\textheight]{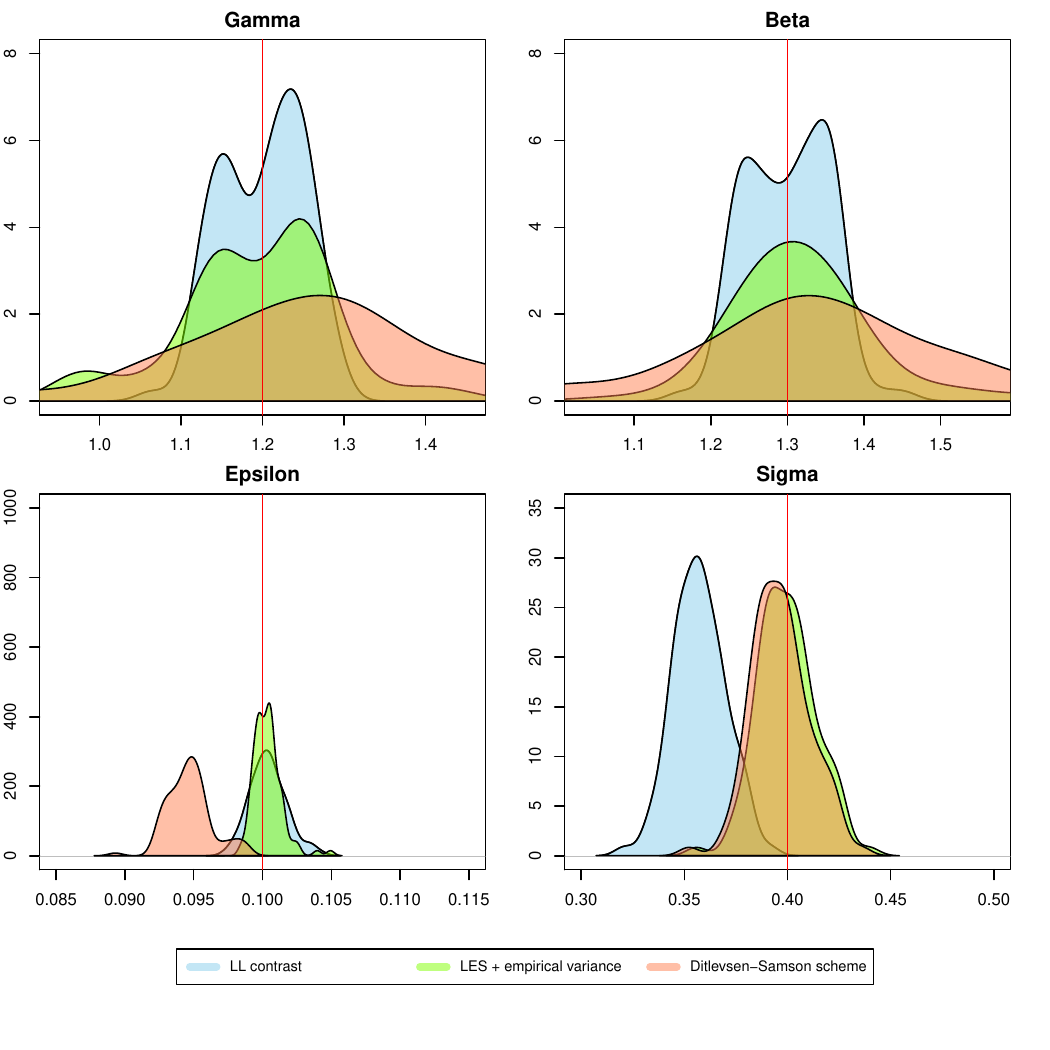}}

\subfloat[$T = 50$]{\includegraphics[height=0.47\textheight]{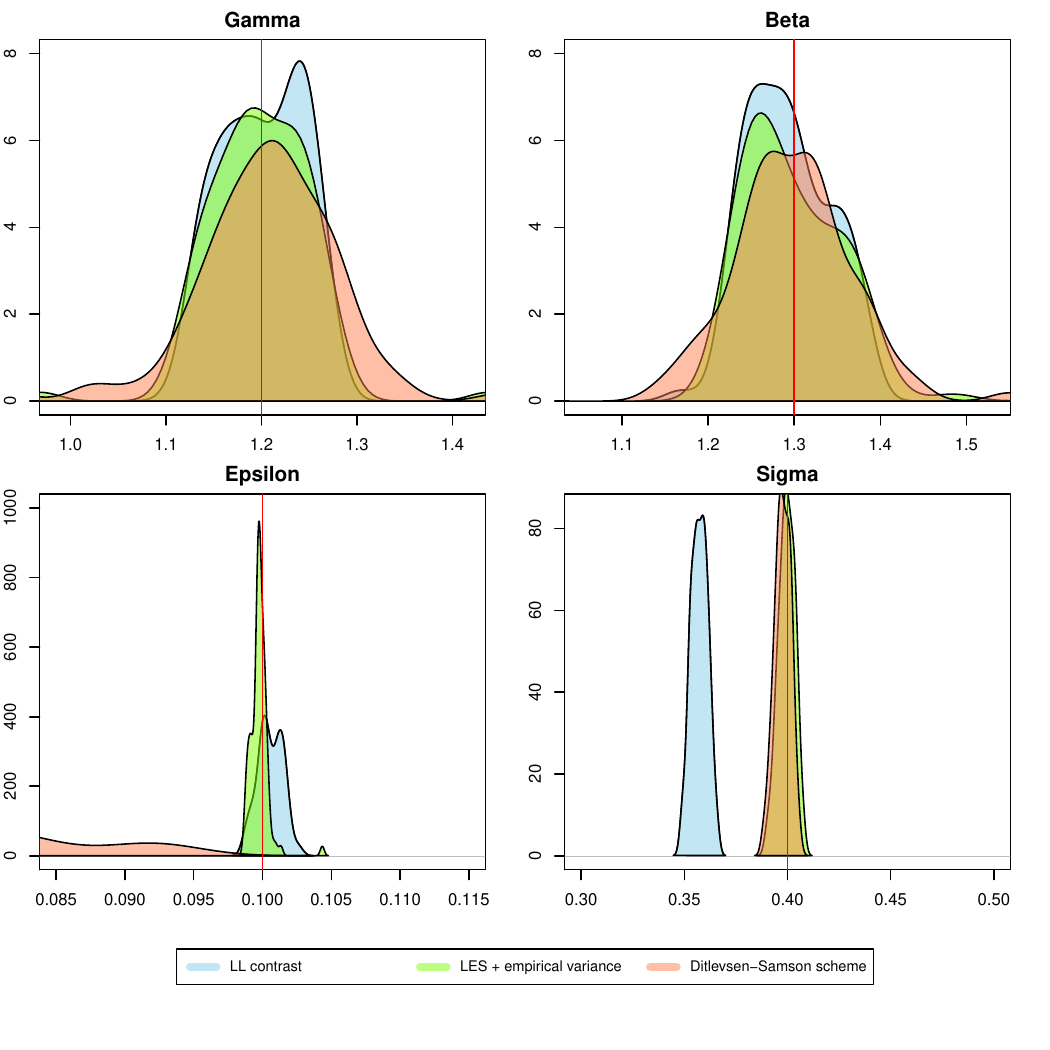}}
\end{center}
\caption{Estimation density for the LL contrast (blue), the LSE (red) and 1.5 scheme (green) estimators for the inhibitory set. $\Delta_n = 0.01$}\label{FitzHugh_densities_2_Delta_001}
\end{figure}
\clearpage
}

\section{Conclusions}\label{section:conclusions}

The proposed contrast estimator generalizes parametric inference methods developed for models of type \eqref{eq:hamiltonian_system} to more general class \eqref{eq:system}. Numerical study shows that it can be used with no prior knowledge of the parameters. It is the main advantage of our method over the analogous works, in particular \cite{Ditlevsen2017}, where the convergence of the estimator is proven with the parameters being partly fixed to their true values. 

From the theoretical point of view, our estimators reveal good properties. Both the contrast based on the local linearization scheme and the least square estimators are consistent. In the case of the contrast, the estimator of the rough coordinate asymptotically depends on the estimator of the smooth coordinate. Therefore its performance is sensitive to the form of the drift term.  The convergence of the smooth coordinate, however, does neither depend on the diffusion term, nor on the rough coordinate.  The question of the asymptotic normality is more intricated. We prove the asymptotic normality under rather restrictive assumptions of the drift term. Nevertheless, the method can be applied to more general models, which is confirmed by the numerical study.  The normality of the least squares estimator is studied under no additional assumptions on the drift term. It is noted that the estimation of parameters with LSE in the drift term is mutually independent, that gives an advantage over the classical contrast estimator. However, numerically LSE is rather sensitive to the experiment design and tends to produce outliers if the observation interval is not big enough. 

The most important direction of the prospective work is the adaptation of the estimation method to the case when only the observations of the first coordinate are available. Under proper conditions it must be possible to couple the contrast minimization with one of the existing filtering methods and estimate the parameters of the system (at least, partially). 

Another point is the generalization of the contrast to systems of higher dimension. In practice we often deal with high-dimensional systems with arbitrary number of rough and smooth variables. The general rule which describes gives the contrast function in that case is not yet established. The most important step here would be to establish the condition of hypoellipticity for such a system. Finally, it is crucial to pair the method with a robust optimization procedure, since the minimization of the contrast is sensitive to choice of the discretization step and initial conditions.

\section*{Acknowledgements}
Author's work was financially supported by LabEx PERSYVAL-Lab and LABEX MME-DII. Sincere gratitude is expressed to Adeline Leclercq-Samson and Eva L\"ocherbach for numerous discussions, helpful remarks and careful proof-reading of the paper draft, as well to the referee, whose comments helped to substantially improve this paper. 

This is a post-peer-review, pre-copyedit version of an article published in Statistical Inference for Stochastic Processes. The final authenticated version is available online at: \url{http://dx.doi.org/10.1007/s11203-020-09222-4}

\bibliography{references}
\section{Appendix}\label{section:appendix}

\subsection{Properties of the scheme}\label{subsec:app_scheme}
\begin{proof}[Proposition \ref{prop:drift}]
By integrating \eqref{eq:LAE_solution} by parts two times we get the following:
\begin{multline}\label{eq:LL_scheme}
\mathds{E}\left[\tilde{Z}_{i+1}|\tilde{Z}_{i} \right] = \tilde{Z}_i + J^{-1}(\tilde{Z}_i; \theta)\left(e^{J(\tilde{Z}_i; \theta)\Delta_n}- I \right)A(\tilde{Z}_i; \theta)+ \\
\frac{1}{2}J^{-2}(\tilde{Z}_i; \theta)\left(e^{J(\tilde{Z}_i; \theta)\Delta_n}- I-J(\tilde{Z}_i; \theta)\Delta_n \right)b^2(\tilde{Z}_i;\sigma)\partial^2_{yy}A(\tilde{Z}_i; \theta) 
\end{multline}
Recall that the matrix exponent for some square matrix $M$ is given by $e^M = \sum_{l=0}^\infty \frac{M^l}{l!}$. Then \eqref{eq:LL_scheme} can be simplified as: 
\begin{multline*}
\mathds{E}\left[\tilde{Z}_{i+1}|\tilde{Z}_{i} \right] =  \tilde{Z}_i + J^{-1}(\tilde{Z}_i; \theta)\left(I + \Delta_nJ(\tilde{Z}_i; \theta) + \frac{\Delta_n^2}{2}J^2(\tilde{Z}_i; \theta)  - I+ O(\Delta^3_n) \right)A(\tilde{Z}_i; \theta)+ \\  \frac{1}{2}J^{-2}(\tilde{Z}_i; \theta)\left(I + \Delta_nJ(\tilde{Z}_i; \theta) + \frac{\Delta_n^2}{2}J^2(\tilde{Z}_i; \theta) - I-\Delta_n J(\tilde{Z}_i; \theta)+ O(\Delta^3_n) \right)b^2(\tilde{Z}_i;\sigma)\partial^2_{yy}A(\tilde{Z}_i; \theta) = \\ 
 \tilde{Z}_i+ \Delta_n A(\tilde{Z}_i; \theta) + \frac{\Delta_n^2}{2}J(\tilde{Z}_i; \theta) A(\tilde{Z}_i; \theta) + \frac{\Delta_n^2}{4}b^2(\tilde{Z}_i;\sigma)\partial^2_{yy}A(\tilde{Z}_i; \theta) + O(\Delta^3_n)
\end{multline*}
Writing the above expression component-wise gives the proposition. 
\end{proof} 
\begin{proof}[Proposition \ref{prop:sigma}]
Let us consider each integral of \eqref{eq:cov_mat_cont} separately. Denote:
\[
\mathcal{W}_{(i+1)\Delta_n} = \int_{i\Delta_n}^{(i+1)\Delta_n} e^{J(\tilde{Z}_i; \theta)((i+1)\Delta_n - s) }B(\tilde{Z}_i; \sigma) d W_s.
\]
Recall that the Jacobian of system \eqref{eq:system_vector} is given by \eqref{eq:jacobian} and the definition of the matrix exponent, we have:
\begin{multline*}
\mathcal{W}_{(i+1)\Delta_n}  = \int_{i\Delta_n}^{(i+1)\Delta_n}(I + J(\tilde{Z}_i; \theta) ((i+1)\Delta_n - s) +  {O}(\Delta_n^2) ) B(\tilde{Z}_i; \sigma) d W_s= \\  
= \int_{i\Delta_n}^{(i+1)\Delta_n} \left[  \left(\begin{matrix} 1 + \partial_xa_1(\tilde{Z}_i; \theta^{(1)})((i+1)\Delta_n - s) & \partial_ya_1(\tilde{Z}_i; \theta^{(1)})((i+1)\Delta_n - s) \\ \partial_xa_2(\tilde{Z}_i; \theta^{(2)})((i+1)\Delta_n - s) & 1 + \partial_ya_2(\tilde{Z}_i; \theta^{(2)})((i+1)\Delta_n - s) \end{matrix}\right)+  {O}(\Delta_n^2)\right] \left( \begin{matrix} 0& 0 \\ 0& 1 \end{matrix} \right) b(\tilde{Z}_i; \sigma) dW_s \\
= b(\tilde{Z}_i; \sigma)\left[\begin{matrix}  0 & \partial_y a_1(\tilde{Z}_i; \theta^{(1)})  \int_{i\Delta_n}^{(i+1)\Delta_n}((i+1)\Delta_n - s)dW_s +  {O}(\Delta_n^2) \\  0 & \int_{i\Delta_n}^{(i+1)\Delta_n}dW_s +\partial_y a_2(\tilde{Z}_i; \theta^{(2)})  \int_{i\Delta_n}^{(i+1)\Delta_n}((i+1)\Delta_n - s)dW_s +  {O}(\Delta_n^2) 
\end{matrix}\right]
\end{multline*}
Then we can calculate $\mathds{E}\left[\mathcal{W}_{(i+1)\Delta_n} \mathcal{W}^\prime_{(i+1)\Delta_n}\right]$:
\begin{equation*}
\mathds{E}\left[\mathcal{W}_{(i+1)\Delta_n} \mathcal{W}^\prime_{(i+1)\Delta_n}\right] = b^2(\tilde{Z}_i; \sigma) \mathds{E} \left(\begin{matrix} 
\Sigma^{(1)}_{\Delta_n} & \Sigma^{(12)}_{\Delta_n} \\ 
\Sigma^{(12)}_{\Delta_n} & \Sigma^{(2)}_{\Delta_n}
\end{matrix}\right) +  {O}(\Delta_n^4),
\end{equation*}
where entries are given by: 
\begin{align*}
\Sigma^{(1)}_{\Delta_n} & = \left(\partial_y a_1(\tilde{Z}_i; \theta^{(1)})  \right)^2 \left[\int_{i\Delta_n}^{(i+1)\Delta_n}((i+1)\Delta_n - s)dW_s\right]^2 \\ 
\Sigma^{(12)}_{\Delta_n} & = \left(\partial_y a_1(\tilde{Z}_i; \theta^{(1)})  \int_{i\Delta_n}^{(i+1)\Delta_n}((i+1)\Delta_n - s)dW_s \right)\left(\int_{i\Delta_n}^{(i+1)\Delta_n}dW_s +\partial_y a_2 (\tilde{Z}_i; \theta^{(2)}) \int_{i\Delta_n}^{(i+1)\Delta_n}((i+1)\Delta_n - s)dW_s \right) \\ 
\Sigma^{(2)}_{\Delta_n} & = \left(\int_{i\Delta_n}^{(i+1)\Delta_n}dW_s +\partial_y a_2(\tilde{Z}_i; \theta^{(2)})  \int_{i\Delta_n}^{(i+1)\Delta_n}((i+1)\Delta_n - s)dW_s \right)^2
\end{align*}
The first entry can be easily calculated by the It\^{o} isometry: 
\begin{multline*}
\mathds{E}[ \Sigma^{(1)}_{\Delta_n} ] = \left(\partial_y a_1(\tilde{Z}_i; \theta^{(1)})  \right)^2 \mathds{E} \left[\int_{i\Delta_n}^{(i+1)\Delta_n}((i+1)\Delta_n - s)dW_s\right]^2 = \\  \left(\partial_y a_1(\tilde{Z}_i; \theta^{(1)})  \right)^2 \int_{i\Delta_n}^{(i+1)\Delta_n}((i+1)\Delta_n - s)^2ds =  \left(\partial_y a_1(\tilde{Z}_i; \theta^{(1)})  \right)^2 \frac{\Delta_n^3}{3}
\end{multline*}
Now consider the product of two stochastic integrals in the terms $\Sigma^{(12)}_{\Delta_n}$ and $\Sigma^{(2)}_{\Delta_n}$. Assume for simplicity that $t=0$. From the properties of the stochastic integrals \citep{karatzas}, it is straightforward to see that:
\begin{multline*}
\mathds{E}\left[ \lim_{n\to\infty}\sum_{t_i, t_{i-1}\in[0, \Delta_n]} (\Delta_n - s) (W_{t_i} - W_{t_{i-1}}) \sum_{t_i, t_{i-1}\in[0, \Delta_n]} (W_{t_i} - W_{t_{i-1}}) \right] = \\ =  \lim_{n\to\infty}\sum_{t_i, t_{i-1}\in[0, \Delta_n]} (\Delta_n - s) \mathds{E}\left[(W_{t_i} - W_{t_{i-1}})^2 \right] = \int_0^{\Delta_n} (\Delta_n - s)ds = \frac{\Delta_n^2}{2}
\end{multline*}
That gives the proposition. 
\end{proof}

\subsection{Auxiliary results}\label{subsec:app_aux}

We start with an important Lemma which links the sampling and the probabilistic law of the continuous process:
\begin{lemma}[\cite{Kessler1997}]\label{lemma_kessler} 
Let $\Delta_n\to 0$ and $n\Delta_n \to \infty$, let $f\in\mathds{R} \times \Theta \to \mathds{R}$ be such that $f$ is differentiable with respect to $z$ and $\theta$, with derivatives of polynomial growth in $z$ uniformly in $\theta$. Then:
\[
\frac{1}{n}\sum_{i=1}^n f(Z_i; \theta) \overset{\mathds{P}_0}{\longrightarrow} \int f(z; \theta) \nu_0(dz) \text{ as } n\to \infty \text{ uniformly in } \theta.
\]
\end{lemma}
Lemma is proven in \cite{Kessler1997} for the one-dimensional case. Its proof is based only on ergodicity of the process and the assumptions analogous to ours, and not on the discretization scheme or dimensionality. So it can be generalized to a multi-dimensional case.

Proposition \ref{prop:bounds_part} in combination with the continuous ergodic theorem and Lemma \ref{lemma_kessler} allow us to establish the following important result:
\begin{lemma}\label{lemma_bounds}
Let $f: \mathds{R}^2 \times \Theta \to \mathds{R} $ be a function with the derivatives of polynomial growth in $x$, uniformly in $\theta$. Assume $\Delta_n \to 0 \text{ and } n\Delta_n \to \infty$. Then:
\begin{enumerate}
\item[(i)] $ \frac{1}{n\Delta^3_n} \sum_{i = 0}^{n-1}\frac{ f(Z_i; \theta)}{(\partial_{y} a_1(Z_i; \theta^{(1)}_0))^2}\left(X_{i+1} - \bar{A}_1(Z_i;\theta^{(1)}_0, \theta^{(2)},\sigma)\right)^2 \overset{\mathds{P}_0}{\longrightarrow} \frac{1}{3} \int f(z; \theta) b^2(z; \sigma_0) \nu_0(dz) $
\item[(ii)] $
\frac{1}{n\Delta_n} \sum_{i = 0}^{n-1}f(Z_i; \theta)\left(Y_{i+1} - Y_i\right)^2 \overset{\mathds{P}_0}{\longrightarrow} \int f(z; \theta) b^2(z; \sigma_0) \nu_0(dz)$
\item[(iii)] $
\frac{1}{n\Delta^2_n} \sum_{i = 0}^{n-1}\frac{ f(Z_i; \theta)}{\partial_{y} a_1(Z_i; \theta^{(1)}_0)}\left(X_{i+1} - \bar{A}_1(Z_i;\theta^{(1)}_0, \theta^{(2)}, \sigma)\right)\left(Y_{i+1} - Y_i\right) \overset{\mathds{P}_0}{\longrightarrow} \frac{1}{2} \int f(z; \theta) b^2(z; \sigma_0) \nu_0(dz)$
\end{enumerate}
\end{lemma}
\begin{proof}
We consider only the cross-term (iii), since the results for the first and the second term are analogous to \cite{Ditlevsen2017} (upon replacing the bounds from Proposition \ref{prop:bounds} by \ref{prop:bounds_part}). 
Thanks to Proposition \ref{prop:bounds_part} we know that:
\begin{multline*}
\mathds{E} \left[ \frac{1}{n\Delta^2} \frac{f(Z_i; \theta)}{\partial_{y} a_1(Z_i; \theta^{(1)}_0)}\left(X_{i+1} - \bar{A}_1(Z_i;\theta^{(1)}_0, \theta^{(2)}, \sigma)\right)\left(Y_{i+1} - Y_i\right) |\mathcal{F}_i\right] = \\ \frac{1}{2n} f(Z_{i}; \theta)b^2(Z_i; \sigma_0) +  {O}(\Delta_n).
\end{multline*}
Then from Lemma \ref{lemma_kessler} it follows that for $n\to \infty$  uniformly in $\theta$: 
\begin{multline*}
\sum_{i=0}^{n-1}\mathds{E} \left[ \frac{1}{n\Delta^2} \frac{f(Z_i; \theta)}{\partial_{y} a_1(Z_i; \theta^{(1)}_0)}\left(X_{i+1} - \bar{A}_1(Z_i;\theta^{(1)}_0, \theta^{(2)}, \sigma)\right)\left(Y_{i+1} - Y_i\right) |\mathcal{F}_i\right] \overset{\mathds{P}_0}{\longrightarrow} \\ \frac{1}{2}\int f(z; \theta)b^2(z; \sigma_0)\nu_0(dz)
\end{multline*}
\end{proof}
Let us introduce an auxiliary Lemma which establishes the convergence in probability for the first moments:
\begin{lemma}\label{lemma_bounds_first}
Let $f: \mathds{R}^{2} \times \Theta \to \mathds{R}$ be a function with derivatives of polynomial growth in $x$, uniformly in $\theta$. Assume $\Delta_n \to 0$ and $n\Delta_n \to \infty$. Then the following convergence results hold: 
\begin{enumerate}
\item[(i)] $ \frac{1}{n\Delta_n} \sum_{i=0}^{n-1} f(Z_i; \theta) (X_{i+1} - \bar{A}_1(Z_i; \theta^{(1)}_0, \theta^{(2)}, \sigma)) \overset{\mathds{P}_0}{\longrightarrow} 0 $
\item[(ii)] $ \frac{1}{n\Delta_n} \sum_{i=0}^{n-1} f(Z_i; \theta) (Y_{i+1} - \bar{A}_2(Z_i; \theta^{(1)}, \theta^{(2)}_0, \sigma)) \overset{\mathds{P}_0}{\longrightarrow} 0 $
\end{enumerate}
uniformly in $\theta$.
\end{lemma}
\begin{proof}
Consider (ii). Expectation of the sum tends to zero for $\Delta_n \to 0$ and $n\Delta_n \to \infty$ due to Proposition \ref{prop:bounds_part}. 
Convergence for $\theta^{(1)}$ is due to Lemma 9 in \cite{Genon-Catalot1993} and uniformity in  $\theta^{(1)}$ follows the proof of Lemma 10 in \cite{Kessler1997}. The second assertion is proven in the same way. For (i) see Lemma 3 in \cite{Ditlevsen2017}.
\end{proof}

 We also need the following Lemma for proving the asymptotic normality of the estimators. 
\begin{lemma}\label{lemma:convergence_in_dist}
Assume (A1)-(A4) and $n\Delta_n\to\infty$ and $n\Delta_n^2\to 0$. Then for any bounded function $f(z;\theta)\in \mathds{R}^2\times\Theta \to \mathds{R}$ the following holds:
\begin{itemize}
\item[(i)] $\frac{1}{\sqrt{n\Delta_n^3}}\sum_{i=0}^{n-1}f(Z_i;\theta)(X_{i+1} - \bar{A}_1(Z_i;\theta^{(1)}_0, \theta^{(2)}, \sigma))\overset{\mathcal{D}}{\longrightarrow} \mathcal{N}\left(0, \frac{1}{3}\nu_0\left(b^2(z;\sigma_0)(\partial_y a_1(z; \theta^{(1)}_0) )^2 f^2(z;\theta)\right) \right)$
\item[(ii)] $\frac{1}{\sqrt n \Delta_n^3}\sum_{i=0}^{n-1}f(Z_i;\theta)(X_{i+1} - \bar{A}_1(Z_i;\theta^{(1)}_0, \theta^{(2)}, \sigma))^2 - \frac{1}{\sqrt n}\sum_{i=0}^{n-1} f(Z_i; \theta) \frac{1}{3}b^2(z;\sigma_0)(\partial_y a_1(z; \theta^{(1)}_0) )^2   \overset{\mathcal{D}}{\longrightarrow} \mathcal{N}\left(0, \frac{2}{9}\nu_0\left(b^4(z;\sigma_0)(\partial_y a_1(z; \theta^{(1)}_0) )^4 f^2(z;\theta)\right) \right)$
\item[(iii)] $\frac{1}{\sqrt{n\Delta_n}}\sum_{i=0}^{n-1}f(Z_i;\theta)(Y_{i+1} - Y_i)\overset{\mathcal{D}}{\longrightarrow} \mathcal{N}\left(0, \nu_0\left(b^2(z;\sigma_0) f^2(z;\theta)\right) \right)$
\item[(iv)] $\frac{1}{\sqrt n\Delta_n}\sum_{i=0}^{n-1}f(Z_i;\theta)(Y_{i+1} - Y_i)^2 - \frac{1}{\sqrt n}\sum_{i=0}^{n-1} f(Z_i; \theta) b^2(Z_i; \sigma_0) \overset{\mathcal{D}}{\longrightarrow} \mathcal{N}\left(0, 2\nu_0\left(b^4(z;\sigma_0) f^2(z;\theta)\right) \right)$
\item[(v)] $\frac{1}{\sqrt n\Delta_n^2}\sum_{i=0}^{n-1}f(Z_i;\theta)(X_{i+1} - \bar{A}_1(Z_i;\theta^{(1)}_0, \theta^{(2)}, \sigma))(Y_{i+1} - Y_i) - \frac{1}{\sqrt n}\sum_{i=0}^{n-1}f(Z_i;\theta)\frac{1}{2}b^2(Z_i; \sigma_0)\partial_ya_1(Z_i; \theta^{(1)}_0) \overset{\mathcal{D}}{\longrightarrow} \mathcal{N}\left(0, \frac{4}{3}\nu_0\left(f(z;\theta)b^4(z; \sigma_0)(\partial_ya_1(z; \theta^{(1)}_0))^2 \right) \right)$
\end{itemize}
\end{lemma}
\begin{proof}
We focus on the proof of (v), since (i)-(iv) closely follow Lemmas 4-5 in \cite{Ditlevsen2017}.
To simplify the proof for the cross-term, we recall that the representation \eqref{eq:process_approximation} can be transformed so that the two noise terms are independent. For example, we can use an analogue of such a decomposition proposed in \cite{Pokern2007}:
\begin{align*}
X_{i+1} - \bar{A}_1(Z_i; \theta^{(1)}_0, \theta^{(2)}, \sigma) &= b(Z_i; \sigma_0)\partial_ya_1(Z_i; \theta^{(1)}_0)\left(\frac{\Delta_n^{3\over 2}}{\sqrt 12}\eta^1_i +\frac{\Delta_n^{3\over 2}}{2}\eta^2_i  \right) + \delta^1_i \\
Y_{i+1} - Y_i &= \Delta_n a_2(Z_i; \theta^{(2)})+ b(Z_i; \sigma_0)\Delta_n^{1\over 2}\eta^2_i+ \delta^2_i,
\end{align*}
where $\delta^1_i$ and $\delta^2_i$ are error terms such that $\mathds{E}[\delta^k_i|\mathcal{F}_i]=O(\Delta_n^2)$ and $\mathds{E}[(\delta^k_i)^2|\mathcal{F}_i]=O(\Delta_n^4)$ (see Proposition \ref{prop:bounds_part}), and $\eta^1_i$ and $\eta^2_i$ are standard independent normal variables. 
 
Then Proposition \ref{prop:bounds_part} gives that $\mathds{E}\left[ \left( X_{i+1} - \bar{A}_1(Z_i; \theta^{(1)}_0, \theta^{(2)}, \sigma)\right) \left( Y_{i+1} - Y_i\right) |\mathcal{F}_i \right] = \frac{\Delta_n^2}{2} b(Z_i; \sigma_0)\partial_ya_1(Z_i; \theta^{(1)}_0) + O(\Delta_n^3)$, and then $ \mathds{E}\left[ f(Z_i; \theta) \left( \left( X_{i+1} - \bar{A}_1(Z_i; \theta^{(1)}_0, \theta^{(2)}, \sigma)\right) \left( Y_{i+1} - Y_i\right) - \frac{\Delta_n^2}{2}b^2(Z_i; \sigma_0)\partial_ya_1(Z_i; \theta^{(1)}_0) \right)|\mathcal{F}_i \right] = 0$. With slightly more tedious computations (which are omitted) we get also that 
\begin{multline*}
\mathds{E}\left[ \left(\left( X_{i+1} - \bar{A}_1(Z_i; \theta^{(1)}_0, \theta^{(2)}, \sigma)\right) \left( Y_{i+1} - Y_i\right)  -\frac{\Delta_n^2}{2}b^2(Z_i; \sigma_0)(\partial_ya_1(Z_i; \theta^{(1)}_0)) \right)^2 |\mathcal{F}_i \right] = \\ \frac{4\Delta^4_n}{3}b^4(Z_i; \sigma_0)(\partial_ya_1(Z_i; \theta^{(1)}_0))^2 + O(\Delta_n^5)
\end{multline*}
 Then we obtain:
\begin{multline*}
\frac{1}{\sqrt{n}\Delta_n^2}\sum_{i=0}^{n-1}f(Z_i;\theta)\left( X_{i+1} - \bar{A}_1(Z_i; \theta^{(1)}_0, \theta^{(2)}, \sigma)\right) \left( Y_{i+1} - Y_i\right)- \frac{1}{\sqrt{n}}\sum_{i=0}^{n-1}f(Z_i;\theta)\frac{1}{2}b^2(Z_i; \sigma_0)\partial_ya_1(Z_i; \theta^{(1)}_0) \\ 
= \frac{1}{\sqrt{n} \Delta_n^2}\sum_{i=0}^{n-1}f(Z_i;\theta) \left( b(Z_i; \sigma_0)\partial_ya_1(Z_i; \theta^{(1)}_0)\left(\frac{\Delta_n^{3\over 2}}{\sqrt{ 12}}\eta^1_i +\frac{\Delta_n^{3\over 2}}{2}\eta^2_i  \right) + \delta^1_i\right)\\\left( \Delta_n a_2(Z_i; \theta^{(2)})+ b(Z_i; \sigma_0)\Delta_n^{1\over 2}\eta^2_i + \delta^2_i\right) -  \frac{1}{\sqrt{n}}\sum_{i=0}^{n-1}f(Z_i;\theta)\frac{1}{2}b^2(Z_i; \sigma_0)\partial_ya_1(Z_i; \theta^{(1)}_0)
\end{multline*}
Since $\frac{\Delta_n}{n}\to 0$ by design we see that 
\begin{multline*}
\frac{1}{n\Delta^4_n} \mathds{E}\left[\sum_{i=0}^{n-1} f^2(Z_i; \theta)\left( \left( X_{i+1} - \bar{A}_1(Z_i; \theta^{(1)}_0, \theta^{(2)}, \sigma)\right) \left( Y_{i+1} - Y_i\right) -\right.\right. \\ \left.\left. \frac{\Delta^2_n}{2}b(Z_i; \sigma_0)(\partial_ya_1(Z_i; \theta^{(1)}_0)) \right)^2\right] \to  \frac{4}{3}\nu_0\left(f^2(z;\theta)b^4(z; \sigma_0)(\partial_ya_1(z; \theta^{(1)}_0)^2\right)
\end{multline*}
Further, since  $\mathds{E}\left[ f^4(Z_i; \theta)\left( \left( X_{i+1} - \bar{A}_1(Z_i; \theta^{(1)}_0, \theta^{(2)}, \sigma)\right) \left( Y_{i+1} - Y_i\right)  -\frac{\Delta_n^2}{2}b^2(Z_i; \sigma_0)(\partial_ya_1(Z_i; \theta^{(1)}_0)) \right)^4 |\mathcal{F}_i \right]$ is bounded by (A2), we have
$$\frac{1}{n^2\Delta^8_n}\mathds{E}\left[ \sum_{i=0}^{n-1}f^4(Z_i; \theta)\left( \left( X_{i+1} - \bar{A}_1(Z_i; \theta^{(1)}_0, \theta^{(2)}, \sigma)\right) \left( Y_{i+1} - Y_i\right)  -\frac{\Delta_n^2}{2}b^2(Z_i; \sigma_0)(\partial_ya_1(Z_i; \theta^{(1)}_0)) \right)^4 |\mathcal{F}_i \right] \to 0.$$ 
Therefore, we can apply again the Theorem 3.2 from \cite{hall1980} and obtain the statement (v). 
  \end{proof}
\begin{remark} Note that the results for the convergence in distribution for the increments of the second coordinate hold without any assumption on the parameters of the function $a_2(z; \theta^{(2)})$. It is due to the fact that the order of the noise dominates the order of the drift term (which is not the case in first coordinate, where the noise is propagated with the higher order). As a consequence, the convergence of a functional $\sum_{i=0}^{n-1}f(Z_i;\theta)(Y_{i+1}-\bar A_2(Z_i; \theta^{(1)}, \theta^{(2)}, \sigma))$ holds, with a proper scaling, for any value of $\theta$. 
\end{remark}

\subsection{Consistency and asymptotic normality of the LL contrast estimator}\label{subsec:app_consistency}

\begin{proof}[Lemma \ref{lemma:theta_1}]
Consider $$ \frac{{\Delta_n}}{n} \left[ \mathcal{L}_{n, \Delta_n}(\theta^{(1)}, \theta^{(2)}, \sigma^2; Z_{0:n}) - \mathcal{L}_{n, {\Delta_n}}(\theta^{(1)}_0, \theta^{(2)}, \sigma^2; Z_{0:n})  \right] = T_1 + T_2 + T_3 + T_4,$$ where the terms are given as follows:
\begin{align*}
T_1 &= \frac{6{\Delta_n}}{n{\Delta_n^3}}\sum_{i=0}^{n-1}\left[ \frac{ \left( X_{i+1} - \bar{A}_1(Z_i; \theta^{(1)}, \theta^{(2)}, \sigma) \right)^2}{b^2(Z_i; \sigma)\left(\partial_ya_1(Z_i; \theta^{(1)})\right)^2} -   \frac{ \left( X_{i+1} - \bar{A}_1(Z_i; \theta^{(1)}_0, \theta^{(2)}, \sigma) \right)^2}{b^2(Z_i; \sigma)\left(\partial_ya_1(Z_i; \theta^{(1)}_0)\right)^2} \right] \\
T_2 &= -\frac{6{\Delta_n}}{n{\Delta_n^2}}\sum_{i=0}^{n-1} \frac{1}{b^2(Z_i; \sigma) }\left[\frac{\left( X_{i+1} - \bar{A}_1(Z_i; \theta^{(1)}, \theta^{(2)}, \sigma) \right) \left(Y_{i+1} - \bar{A}_2(Z_i; \theta^{(1)}, \theta^{(2)}, \sigma)\right) }{\partial_ya_1(Z_i; \theta^{(1)})}  - \right. \\  & \qquad \left. \frac{\left( X_{i+1} - \bar{A}_1(Z_i; \theta^{(1)}_0, \theta^{(2)}, \sigma) \right)\left(Y_{i+1} - \bar{A}_2(Z_i; \theta^{(1)}_0, \theta^{(2)}, \sigma)\right) }{\partial_ya_1(Z_i; \theta^{(1)}_0)} \right] \\
T_3 & = \frac{2\Delta_n}{n\Delta_n} \sum_{i=0}^{n-1} \left[ \frac{ \left( Y_{i+1} - \bar{A}_2(Z_i; \theta^{(1)}, \theta^{(2)}, \sigma) \right)^2}{b^2(Z_i; \sigma)} -   \frac{ \left( Y_{i+1} - \bar{A}_2(Z_i; \theta^{(1)}_0, \theta^{(2)}, \sigma) \right)^2}{b^2(Z_i; \sigma)} \right] \\ 
T_4 &= \frac{{\Delta_n}}{n}\sum_{i=0}^{n-1} \log\left(\frac{\partial_ya_1(Z_i; \theta^{(1)})}{\partial_ya_1(Z_i; \theta^{(1)}_0)}\right) 
\end{align*}
Consider term $T_1$:
\begin{multline*}
T_1 =  \frac{6{\Delta_n}}{n{\Delta_n^3}}\sum_{i=0}^{n-1} \frac{1}{b^2(Z_i; \sigma)} \left[ \frac{ \left( X_{i+1} - \bar{A}_1(Z_i; \theta^{(1)}_0, \theta^{(2)}, \sigma) + \bar{A}_1(Z_i; \theta^{(1)}_0, \theta^{(2)}, \sigma) - \bar{A}_1(Z_i; \theta^{(1)}, \theta^{(2)}, \sigma) \right)^2}{\left(\partial_ya_1(Z_i; \theta^{(1)})\right)^2} - \right.\\ \left. \frac{ \left( X_{i+1} - \bar{A}_1(Z_i; \theta^{(1)}_0, \theta^{(2)}, \sigma) \right)^2}{\left(\partial_ya_1(Z_i; \theta^{(1)}_0)\right)^2} \right] = 
\frac{6{\Delta_n}}{n{\Delta_n^3}}\sum_{i=0}^{n-1} \frac{1}{b^2(Z_i; \sigma)}\left[   \left( X_{i+1} - \bar{A}_1(Z_i; \theta^{(1)}_0, \theta^{(2)}, \sigma) \right)^2\left[ \frac{1}{\left(\partial_ya_1(Z_i; \theta^{(1)})\right)^2} -  \right.\right.\\ \left. \left.   \frac{1}{\left(\partial_ya_1(Z_i; \theta^{(1)}_0)\right)^2}  \right] +  \frac{2\Delta_n}{\left(\partial_ya_1(Z_i; \theta^{(1)})\right)^2}\left( X_{i+1} - \bar{A}_1(Z_i; \theta^{(1)}_0, \theta^{(2)}, \sigma) \right) (a_1(Z_i; \theta^{(1)}_0) - a_1(Z_i; \theta^{(1)})) + \right.\\ \left.  \frac{\Delta_n^2}{(\partial_ya_1(Z_i; \theta^{(1)}))^2}\left(a_1(Z_i; \theta^{(1)}_0) - a_1(Z_i; \theta^{(1)})\right)^2\right].
\end{multline*}
Recalling Lemmas \ref{lemma_kessler}, \ref{lemma_bounds_first} and  \ref{lemma_bounds} we have that:
\begin{gather*}
\frac{6}{n{\Delta_n^2}}\sum_{i=0}^{n-1}  \frac{\left( X_{i+1} - \bar{A}_1(Z_i; \theta^{(1)}_0, \theta^{(2)}, \sigma) \right)^2}{b^2(Z_i; \sigma)} \left[ \frac{1}{\left(\partial_ya_1(Z_i; \theta^{(1)})\right)^2} -  \frac{1}{\left(\partial_ya_1(Z_i; \theta^{(1)}_0)\right)^2}  \right] \overset{\mathds{P}_0}{\longrightarrow} 0 \\ 
\frac{6}{n{\Delta_n} }\sum_{i=0}^{n-1}\frac{1}{b^2(Z_i; \sigma)(\partial_ya_1(Z_i; \theta^{(1)}))^2} \left( X_{i+1} - \bar{A}_1(Z_i; \theta^{(1)}_0, \theta^{(2)}, \sigma) \right) (a_1(Z_i; \theta^{(1)}_0) - a_1(Z_i; \theta^{(1)}))  \overset{\mathds{P}_0}{\longrightarrow} 0 \\
\frac{6}{n}\sum_{i=0}^{n-1}\frac{ (a_1(Z_i; \theta^{(1)}_0) - a_1(Z_i; \theta^{(1)}))^2}{b^2(Z_i; \sigma)(\partial_ya_1(Z_i; \theta^{(1)}))^2}  \overset{\mathds{P}_0}{\longrightarrow} 6 \int \frac{ (a_1(z; \theta^{(1)}_0) - a_1(z; \theta^{(1)}))^2}{b^2(z; \sigma)(\partial_ya_1(z; \theta^{(1)}))^2} \nu_0(dz).
\end{gather*}
Now consider $T_2$, which can be rewritten as:
\begin{multline*}
-\frac{6}{n{\Delta_n}}\sum_{i=0}^{n-1} \frac{\left(Y_{i+1} - Y_i + O(\Delta_n)\right) }{ b^2(Z_i; \sigma)} \left[\frac{\left( X_{i+1} - \bar{A}_1(Z_i; \theta^{(1)}_0, \theta^{(2)}, \sigma) +  \bar{A}_1(Z_i; \theta^{(1)}_0, \theta^{(2)}, \sigma) - \bar{A}_1(Z_i; \theta^{(1)}, \theta^{(2)}, \sigma) \right) }{\partial_ya_1(Z_i; \theta^{(1)})}  - \right. \\ \left. \frac{\left( X_{i+1} - \bar{A}_1(Z_i; \theta^{(1)}_0, \theta^{(2)}, \sigma) \right) }{\partial_ya_1(Z_i; \theta^{(1)}_0)} \right] = 
-\frac{6}{n {\Delta_n}} \sum_{i=0}^{n-1} \frac{\left(Y_{i+1} - Y_i + O(\Delta_n)\right) }{b^2(Z_i; \sigma)} \left[\left( X_{i+1} - \bar{A}_1(Z_i; \theta^{(1)}_0, \theta^{(2)}, \sigma)\right)\right. \\ \left.\left[\frac{1}{\partial_ya_1(Z_i; \theta^{(1)})}  - \frac{1}{\partial_ya_1(Z_i; \theta^{(1)}_0)}\right] +  \frac{{\Delta_n}}{(\partial_ya_1(Z_i; \theta^{(1)}))} (a_1(Z_i; \theta^{(1)}_0) - a_1(Z_i; \theta^{(1)}))\right].
\end{multline*}
Then we use the fact that the expectation of $\left( X_{i+1} - \bar{A}_1(Z_i; \theta^{(1)}_0, \theta^{(2)}, \sigma)\right)$ is of order $\Delta_n^2$ and of increments $Y_{i+1} - Y_i$ is of $\Delta_n$, and by Lemma \ref{lemma_bounds} we obtain:
\begin{gather*}
-\frac{6}{n{\Delta_n} } \sum_{i=0}^{n-1} \frac{\left( X_{i+1} - \bar{A}_1(Z_i; \theta^{(1)}_0, \theta^{(2)}, \sigma) \right) \left(Y_{i+1} - Y_i + O(\Delta_n)\right)}{b^2(Z_i; \sigma)(\partial_ya_1(Z_i; \theta^{(1)}_0))} \left[ \frac{(\partial_ya_1(Z_i; \theta^{(1)}_0))}{(\partial_ya_1(Z_i; \theta^{(1)}))} -  1  \right] \overset{\mathds{P}_0}{\longrightarrow} 0.
\end{gather*}
The same holds for $T_4$. Consider then $T_3$: 
\begin{multline*}
 \frac{2\Delta_n}{n\Delta_n} \sum_{i=0}^{n-1} \left[ 2\frac{ \left( Y_{i+1} - \bar{A}_2(Z_i; \theta^{(1)}, \theta^{(2)}, \sigma) \right)^2}{b^2(Z_i; \sigma)} -   \frac{ \left( Y_{i+1} - \bar{A}_2(Z_i; \theta^{(1)}_0, \theta^{(2)}, \sigma) \right)^2}{b^2(Z_i; \sigma)} \right] = \\ 
 \frac{2}{n} \sum_{i=0}^{n-1} \frac{1}{b^2(Z_i; \sigma)} \left[\left( Y_{i+1} - \bar{A}_2(Z_i; \theta^{(1)}_0, \theta^{(2)}, \sigma) \right)\left( \bar{A}_2(Z_i; \theta^{(1)}_0, \theta^{(2)}, \sigma) - \bar{A}_2(Z_i; \theta^{(1)}, \theta^{(2)}, \sigma) \right) -  \right. \\ \left.  \left( \bar{A}_2(Z_i; \theta^{(1)}_0, \theta^{(2)}, \sigma) - \bar{A}_2(Z_i; \theta^{(1)}, \theta^{(2)}, \sigma) \right)^2 \right]
\end{multline*}
This term is of order $O(\Delta^3_n)$ (since $\theta^{(1)}$ is contained only in terms of order $\Delta^2_n$), thus it converges to zero as $\Delta_n\to 0$. 
 Thus, we indeed have 
\begin{multline}\label{statement_first}
\lim_{n\to\infty, {\Delta_n} \to 0} \frac{{\Delta_n}}{n} \left[ \mathcal{L}_{n, \Delta_n}(\theta^{(1)}, \theta^{(2)}, \sigma^2; Z_{0:n}) - \mathcal{L}_{n, {\Delta_n}}(\theta^{(1)}_0, \theta^{(2)}, \sigma^2; Z_{0:n})  \right] \overset{\mathds{P}_0}{\longrightarrow} \\ 6 \int \frac{ (a_1(z; \theta^{(1)}_0) - a_1(z; \theta^{(1)}))^2}{b^2(z; \sigma)(\partial_ya_1(z; \theta^{(1)}))^2} \nu_0(dz).
\end{multline}
 \end{proof}

\begin{proof}[Theorem \ref{thm:consistency_and_asymptotics_theta1} (consistency and asymptotic normality of $\theta^{(1)}$)]

Throughout the proof we assume that $\theta^{(1)}\in \mathds{R}$ in order to simplify the notations. 

\textbf{Consistency.}  It follows essentially from Lemma \ref{lemma:theta_1}. Indeed, the result of the Lemma (and the fact that the parameter space is compact) implies that we can find a subsequence $\hat\theta^{(1)}_{n, \Delta_n} $ which converges to some value $\theta^{(1)}_\infty$. However, the minimum of the expression in Lemma \ref{lemma:theta_1} is attained for $\theta^{(1)}_0$. Then by identifiability of the drift function we have the consistency, that is $\hat\theta^{(1)}_{n, \Delta_n}  \to \theta^{(1)}_0$. 

\textbf{Asymptotic normality.}  The proof follows the standard pattern (see \cite{Kessler1997}, \cite{genon-catalot1999}, \cite{Ditlevsen2017}). First, we write the Taylor expansion of the function \eqref{loglikelihood}.  Then we have:
\begin{multline*}
\int \frac{\Delta_n}{n}\frac{\partial^2}{\partial \theta^{(1)}\partial \theta^{(1)}}\mathcal{L}_{n, \Delta_n}\left(\theta^{(1)}_0 + u(\hat\theta^{(1)}_{n,\Delta_n} - \theta_0), \theta^{(2)}, \sigma; z \right)du \cdot \sqrt\frac{n}{\Delta_n}(\hat\theta^{(1)}_{n, \Delta_n}  - \theta^{(1)}_0) =\\ - \sqrt\frac{\Delta_n}{n}\frac{\partial}{\partial \theta^{(1)}}\mathcal{L}_{n, \Delta_n}(\theta^{(1)}_0, \theta^{(2)}, \sigma; z)
\end{multline*}
Note that the values of $\theta^{(2)}$ and $\sigma$ may be taken arbitrary. Now we have to compute the first and the second order derivatives of \eqref{loglikelihood}. We omit the dependency on parameters in the expression for partial derivatives to make it readable and study the convergence of the first order derivative:
\begin{multline}\label{eq:proofs_derivative}
\frac{\partial}{\partial \theta^{(1)}}\mathcal{L}_{n, \Delta_n}(\theta^{(1)}_0,\theta^{(2)}, \sigma; z) = \sum_{i=1}^{n-1}\left[ \frac{2\partial^2_{y,\theta^{(1)}}a_1}{\partial_ya_1} - \frac{6}{b^2(Z_i; \sigma) \partial_ya_1}\right. \\ \left.\left[\frac{2(X_{i+1} - \bar{A}_1(Z_i; \theta^{(1)}_0, \theta^{(2)}, \sigma))^2}{\Delta_n^3 (\partial_ya_1)^2}\partial^2_{y,\theta^{(1)}}a_1 + \frac{2(X_{i+1} - \bar{A}_1(Z_i; \theta^{(1)}_0, \theta^{(2)}, \sigma))(\partial_{\theta^{(1)}}a_1) }{\Delta_n^2(\partial_ya_1)} - \right. \right. \\ \left.\left.\frac{(Y_{i+1}-\bar A_2(Z_i;\theta^{(1)}_0, \theta^{(2)}, \sigma))(\partial_{\theta^{(1)}}a_1) }{\Delta_n} - \right. \right. \\ \left. \left.\frac{(X_{i+1} - \bar{A}_1(Z_i; \theta^{(1)}_0, \theta^{(2)}, \sigma))(Y_{i+1}-\bar A_2(Z_i;\theta^{(1)}_0, \theta^{(2)}, \sigma))(\partial^2_{y,\theta^{(1)}}a_1)}{\Delta_n^2 (\partial_ya_1)} \right]\right]
\end{multline}
Under assumption (A5) the only non-zero terms are the following:
\begin{multline*}
\frac{\partial}{\partial \theta^{(1)}}\mathcal{L}_{n, \Delta_n}(\theta^{(1)}_0,\theta^{(2)}, \sigma; z) = \sum_{i=1}^{n-1} - \frac{6}{b^2(Z_i; \sigma) \partial_ya_1}\left[ \frac{2(X_{i+1} - \bar{A}_1(Z_i; \theta^{(1)}_0, \theta^{(2)}, \sigma))(\partial_{\theta^{(1)}}a_1) }{\Delta_n^2(\partial_ya_1)} - \right. \\ \left.\frac{(Y_{i+1}-\bar A_2(Z_i;\theta^{(1)}_0, \theta^{(2)}, \sigma))(\partial_{\theta^{(1)}}a_1) }{\Delta_n}  \right]
\end{multline*} 
Applying Lemma \ref{lemma:convergence_in_dist}, we get:
\begin{gather*}
\frac{1}{\sqrt{n\Delta_n^3}} \sum_{i=1}^{n-1} \left[\frac{12 (\partial_{\theta^{(1)}}a_1)}{b^2(Z_i; \sigma) (\partial_ya_1)}(X_{i+1} - \bar{A}_1(Z_i; \theta^{(1)}_0, \theta^{(2)}, \sigma))  \right]  \overset{\mathcal{D}}{\longrightarrow} \mathcal{N}\left(0, 36\nu_0\left(\frac{b^2(z;\sigma_0)}{b^4(z;\sigma)}(\partial_{\theta^{(1)}}a_1)^2\right)\right) \\ 
\frac{1}{\sqrt{n\Delta_n}} \sum_{i=1}^{n-1} \left[ \frac{6(\partial_{\theta^{(1)}}a_1)}{b^2(Z_i; \sigma) }\frac{(Y_{i+1}-\bar A_2(Z_i;\theta^{(1)}_0, \theta^{(2)}, \sigma)) }{(\partial_ya_1)} \right] \overset{\mathcal{D}}{\longrightarrow}  \mathcal{N}\left(0, 36\nu_0\left(\frac{b^2(z;\sigma_0)}{b^4(z;\sigma)}\frac{(\partial_{\theta^{(1)}}a_1)^2}{(\partial_ya_1)^2}\right)\right)
\end{gather*}
Thus, we have the following convergence in law: 
\begin{equation*}
\sqrt\frac{\Delta_n}{n}\frac{\partial}{\partial \theta^{(1)}}\mathcal{L}_{n, \Delta_n}(\theta^{(1)}_0,\theta^{(2)}, \sigma; z)\overset{\mathcal{D}}{\longrightarrow} \mathcal{N}\left(0, 36\nu_0\left(\frac{b^2(z;\sigma_0)}{b^4(z;\sigma)}(\partial_{\theta^{(1)}}a_1)^2\left(1+\frac{1}{(\partial_ya_1)^2}\right)\right)\right)
\end{equation*}
For the second order derivative we split again the expression \eqref{eq:proofs_derivative} in several parts and study their convergence:  
\begin{multline*}
T_1 := \frac{\Delta_n}{n} \sum_{i=1}^{n-1}-\frac{12\Delta_n}{\Delta_n^2 b^2(Z_i; \sigma) (\partial_ya_1)^2}\left[ (\partial_{\theta^{(1)}} a_1)^2 + (X_{i+1} - \bar{A}_1(Z_i; \theta^{(1)}_0, \theta^{(2)}, \sigma))\frac{(\partial^2_{\theta^{(1)}\theta^{(1)}}a_1)(\partial_ya_1)}{(\partial_ya_1)^2} \right]
\end{multline*}
\begin{equation*}
T_2 := \frac{\Delta_n}{n} \sum_{i=1}^{n-1} \frac{6(Y_{i+1} -\bar A_2(Z_i; \theta^{(1)}, \theta^{(2)}, \sigma))}{\Delta_n b^2(Z_i;\sigma) } \frac{(\partial_ya_1)^2\partial^2_{ \theta^{(1)} \theta^{(1)}}a_1}{(\partial_ya_1)^4}
\end{equation*}
It is easy to see that the terms $T_2$ converges to $0$ by Lemmas \ref{lemma_bounds_first} and \ref{lemma_bounds}. $T_1$, according to the Lemma \ref{lemma_kessler} and Lemma \ref{lemma_bounds_first}, converges to $12\int\frac{(\partial_{\theta^{(1)}}a_1)^2}{b^2(z; \sigma) (\partial_ya_1)}\nu_0(dz) $. That gives the result.
  \end{proof}

\begin{proof}[Lemma \ref{lemma:theta_2}]
Note that we cannot infer the value of $\theta^{(2)}$ with the same scaling as the parameter of the smooth coordinate because the estimator for each variable converges with different speed. Thus, we fix the parameter $\theta^{(1)}$ to its estimated value $\hat\theta^{(1)}_{n, \Delta_n} $ and consider the same sum, but with a different scaling, namely :
\[
\lim_{n\to\infty, {\Delta_n} \to 0} \frac{1}{n{\Delta_n}} \left[\mathcal{L}_{n, {\Delta_n}}(\hat\theta^{(1)}_{n, \Delta_n} ,\theta^{(2)}, \sigma^2; Z_{0:n}) - \mathcal{L}_{n, {\Delta_n}}(\hat\theta^{(1)}_{n, \Delta_n} , \theta^{(2)}_0, \sigma^2; Z_{0:n}) \right] = T_1 + T_2  + T_3
\] 
where the terms are given as follows:
\begin{multline*}
T_1 = \frac{6}{n{\Delta_n^4}}\sum_{i=0}^{n-1} \frac{1}{b^2(Z_i; \sigma)\left(\partial_ya_1(Z_i; \hat\theta^{(1)}_{n, \Delta_n} )\right)^2}\left[\left( X_{i+1} - \bar{A}_1(Z_i; \hat\theta^{(1)}_{n, \Delta_n}, \theta^{(2)},\sigma )\right)^2  - \right. \\ \left. \left( X_{i+1} - \bar{A}_1(Z_i; \hat\theta^{(1)}_{n, \Delta_n}, \theta^{(2)}_0,\sigma )\right)^2  \right] 
\end{multline*}
\begin{multline*}
T_2 = -\frac{6}{n{\Delta_n^3}}\sum_{i=0}^{n-1} \frac{1}{b^2(Z_i; \sigma)}\left[\frac{\left( X_{i+1} - \bar{A}_1(Z_i; \hat\theta^{(1)}_{n, \Delta_n}, \theta^{(2)},\sigma ) \right) \left(Y_{i+1} - \bar{A}_2(Z_i;  \hat\theta^{(1)}_{n, \Delta_n}, \theta^{(2)}, \sigma)\right)}{\partial_ya_1(Z_i; \hat\theta^{(1)}_{n, \Delta_n} )}  - \right. \\ \left. \frac{\left( X_{i+1} - \bar{A}_1(Z_i; \hat\theta^{(1)}_{n, \Delta_n}, \theta^{(2)}_0,\sigma ) \right) \left(Y_{i+1} - \bar{A}_2(Z_i;  \hat\theta^{(1)}_{n, \Delta_n}, \theta^{(2)}_0, \sigma)\right)}{\partial_ya_1(Z_i;\hat\theta^{(1)}_{n, \Delta_n} )} \right] 
\end{multline*}
\begin{align*}
T_3 &= \frac{2}{n{\Delta_n^2}}\sum_{i=0}^{n-1} \frac{ \left[ (Y_{i+1} - \bar{A}_2(Z_i;  \hat\theta^{(1)}_{n, \Delta_n}, \theta^{(2)}, \sigma))^2 - (Y_{i+1} - \bar{A}_2(Z_i;  \hat\theta^{(1)}_{n, \Delta_n}, \theta^{(2)}_0, \sigma))^2 \right] }{b^2(Z_i; \sigma)} 
\end{align*}
We start with $T_1$:
\begin{multline*}
T_1 = \frac{6}{n{\Delta_n^4}}\sum_{i=0}^{n-1} \frac{1}{b^2(Z_i; \sigma)\left(\partial_ya_1(Z_i; \hat\theta^{(1)}_{n, \Delta_n} )\right)^2}\left[\left( \bar{A}_1(Z_i; \hat\theta^{(1)}_{n, \Delta_n}, \theta^{(2)}_0,\sigma ) - \bar{A}_1(Z_i; \hat\theta^{(1)}_{n, \Delta_n}, \theta^{(2)},\sigma )\right)\right. \\ \left.\left( X_{i+1} - \bar{A}_1(Z_i; \hat\theta^{(1)}_{n, \Delta_n}, \theta^{(2)}_0,\sigma )\right)  -  \left( \bar{A}_1(Z_i; \hat\theta^{(1)}_{n, \Delta_n}, \theta^{(2)}_0,\sigma ) - \bar{A}_1(Z_i; \hat\theta^{(1)}_{n, \Delta_n}, \theta^{(2)},\sigma )\right)^2  \right]  = \\ 
 \frac{6}{n{\Delta_n^4}}\sum_{i=0}^{n-1} \frac{1}{b^2(Z_i; \sigma)\left(\partial_ya_1(Z_i; \hat\theta^{(1)}_{n, \Delta_n} )\right)^2}\left[ \frac{\Delta^2_n}{2}\left(\partial_ya_1(Z_i; \hat\theta^{(1)}_{n, \Delta_n} )\right)\left(a_2(Z_i; \theta^{(2)}_0) - a_2(Z_i; \theta^{(2)}) \right) \right. \\ \left.   \left( X_{i+1} - \bar{A}_1(Z_i; \theta^{(1)}_0, \theta^{(2)}_0,\sigma ) \right) + \frac{\Delta^3_n}{2}\left(\partial_ya_1(Z_i; \hat\theta^{(1)}_{n, \Delta_n} )\right)\left(a_2(Z_i; \theta^{(2)}_0) - a_2(Z_i; \theta^{(2)}) \right)\right. \\ \left. \left(a_1(Z_i; \theta^{(1)}_0)-a_1(Z_i; \hat\theta^{(1)}_{n, \Delta_n})\right) +  \frac{\Delta^4_n}{4}\left(\partial_ya_1(Z_i; \hat\theta^{(1)}_{n, \Delta_n} )\right)^2\left(a_2(Z_i; \theta^{(2)}_0) - a_2(Z_i; \theta^{(2)}) \right)^2 - \right. \\ \left. \frac{\Delta^4_n}{4}\left(\partial_ya_1(Z_i; \hat\theta^{(1)}_{n, \Delta_n} )\right)^2\left(a_2(Z_i; \theta^{(2)}_0) - a_2(Z_i; \theta^{(2)}) \right)^2 \right]
\end{multline*}
Recall that  $\left( X_{i+1} - \bar{A}_1(Z_i; \theta^{(1)}_0, \theta^{(2)}_0,\sigma )\right) $ is of $O(\Delta^3_n)$ by Proposition \ref{prop:bounds}. Thus, the first summand of the $T_1$ is of order $\Delta^5_n$ and converges to $0$. The second summand, under assumption $(A5)$, can be rewritten as
\begin{multline*}
 \frac{3}{n{\Delta_n^4}}\sum_{i=0}^{n-1} \frac{\Delta^3_n\left(\theta^{(1)}_0 - \hat\theta^{(1)}_{n, \Delta_n}\right)^T g(X_i)\left(a_2(Z_i; \theta^{(2)}_0) - a_2(Z_i; \theta^{(2)}) \right) }{b^2(Z_i; \sigma)\left(\partial_ya_1(Z_i; \hat\theta^{(1)}_{n, \Delta_n} )\right)}  = \\ 
 \frac{3}{\sqrt{n\Delta_n} n}\sum_{i=0}^{n-1} \frac{\sqrt\frac{n}{\Delta_n}\left(\theta^{(1)}_0 - \hat\theta^{(1)}_{n, \Delta_n}\right)^T g(X_i)\left(a_2(Z_i; \theta^{(2)}_0) - a_2(Z_i; \theta^{(2)}) \right) }{b^2(Z_i; \sigma)\left(\partial_ya_1(Z_i; \hat\theta^{(1)}_{n, \Delta_n} )\right)} . 
\end{multline*}
$\sqrt\frac{n}{\Delta_n}\left(\theta^{(1)}_0 - \hat\theta^{(1)}_{n, \Delta_n}\right)^T$ converges to a normal variable with zero mean due to Theorem \ref{thm:consistency_and_asymptotics_theta1}, and the whole expression converges to $0$, because $n\Delta_n\to\infty$. 
 Thus $T_1$ converges to $0$.
  Consider $T_2$:
 \begin{multline*}
 T_2 = -\frac{6}{n{\Delta_n^3}}\sum_{i=0}^{n-1} \left[\frac{\Delta_n\left( X_{i+1} - \bar{A}_1(Z_i; \theta^{(1)}_0, \theta^{(2)}, \sigma) \right)(a_2(Z_i; \theta^{(2)}) - a_2(Z_i; \theta^{(2)}_0))}{\partial_ya_1(Z_i;\hat\theta^{(1)}_{n, \Delta_n} )b^2(Z_i; \sigma)} + \right. \\ \left.  \frac{\Delta_n^2 (a_1(Z_i; \theta^{(1)}_0) -a_1(Z_i;\hat \theta^{(1)}_{n,\Delta_n})) (a_2(Z_i; \theta^{(2)}) - a_2(Z_i; \theta^{(2)}_0))}{\partial_ya_1(Z_i;\hat\theta^{(1)}_{n, \Delta_n} )b^2(Z_i; \sigma)} \right] 
 \end{multline*} 
 Then, the first part of the sum converges to zero in probability after applying Lemma \ref{lemma_bounds_first}. The second part of the sum also converges to zero because $n\Delta_n\to\infty$ by design, and $\hat\theta^{(1)}_{n, \Delta_n}  \overset{\mathds{P}_0}{\longrightarrow}\theta^{(1)}_0$. So that, recalling (A5),  and applying the arguments used above for $T_1$ to $a_1(Z_i; \theta^{(1)}_0) - a_1(Z_i;\hat \theta^{(1)}_{n,\Delta_n}) =  \left(\theta^{(1)}_0 - \hat \theta^{(1)}_{n,\Delta_n} \right) g(X_i) $, we prove that $T_2$ also converges to $0$. 
 So we just have to consider the remaining term $T_3$:
\begin{multline*}
T_3 = \frac{2}{n\Delta_n^2} \sum_{i=0}^{n-1} \frac{1}{b^2(Z_i; \sigma)} \left[ (Y_{i+1} - \bar{A}_2(Z_i;  \hat\theta^{(1)}_{n, \Delta_n}, \theta^{(2)}_0, \sigma))^2 + \right.\\ \left. (Y_{i+1} - \bar{A}_2(Z_i;  \hat\theta^{(1)}_{n, \Delta_n}, \theta^{(2)}_0, \sigma))(\bar{A}_2(Z_i;  \hat\theta^{(1)}_{n, \Delta_n}, \theta^{(2)}_0, \sigma) - \bar{A}_2(Z_i;  \hat\theta^{(1)}_{n, \Delta_n}, \theta^{(2)}, \sigma)) + \right. \\ \left. (\bar{A}_2(Z_i;  \hat\theta^{(1)}_{n, \Delta_n}, \theta^{(2)}_0, \sigma) - \bar{A}_2(Z_i;  \hat\theta^{(1)}_{n, \Delta_n}, \theta^{(2)}, \sigma))^2 - (Y_{i+1} - \bar{A}_2(Z_i;  \hat\theta^{(1)}_{n, \Delta_n}, \theta^{(2)}_0, \sigma))^2 \right] = \\ 
\frac{2}{n\Delta_n^2} \sum_{i=0}^{n-1} \frac{1}{b^2(Z_i; \sigma)} \left[{\Delta_n}(Y_{i+1} - \bar{A}_2(Z_i;  \hat\theta^{(1)}_{n, \Delta_n}, \theta^{(2)}_0, \sigma))(a_2(Z_i; \theta^{(2)}_0) - a_2(Z_i; \theta^{(2)})) + \right. \\ \left.  {\Delta_n^2}(a_2(Z_i; \theta^{(2)}_0) - a_2(Z_i; \theta^{(2)}))^2  \right] 
\end{multline*}
The first part of the sum converges to $0$ due to Lemma \ref{lemma_bounds_first}. Then we apply Lemma \ref{lemma_bounds} and get the convergence: 
\begin{multline*}
\lim_{n\to\infty, {\Delta_n} \to 0} \frac{1}{n{\Delta_n}} \left[\mathcal{L}_{n, {\Delta_n}}(\hat\theta^{(1)}_{n, \Delta_n} ,\theta^{(2)}, \sigma^2; Z_{0:n}) - \mathcal{L}_{n, {\Delta_n}}(\hat\theta^{(1)}_{n, \Delta_n} , \theta^{(2)}_0, \sigma^2_0; Z_{0:n}) \right]  \overset{\mathds{P}_0}{\longrightarrow} \\ 2 \int \frac{(a_2(z; \theta^{(2)}_0) - a_2(z; \theta^{(2)}))^2}{b^2(z; \sigma)} \nu_0(dz)
\end{multline*}
\end{proof}

\begin{proof}[Lemma \ref{lemma:sigma}]
We can split the contrast in the following sum:
\[
\lim_{n\to\infty, {\Delta_n} \to 0} \frac{1}{2n} \mathcal{L}_{n, {\Delta_n}}(\hat\theta^{(1)}_{n, \Delta_n} , \theta^{(2)}, \sigma^2; Z_{0:n})  = \lim_{n\to\infty, {\Delta_n} \to 0}\left[ 3T_1 - 3T_2 + T_3 + T_4 \right]
\]
where terms are given by follows:
\begin{align*}
T_1 & = \frac{1}{n}\sum_{i=0}^{n-1}\frac{(X_{i+1} - \bar{A}_1(Z_i; \hat\theta^{(1)}_{n, \Delta_n}, \theta^{(2)},\sigma ))^2}{{\Delta_n^3} b^2(Z_i; \sigma)(\partial_{y} a_1(Z_i; \hat\theta^{(1)}_{n, \Delta_n} ))^2} \\
T_2 & = \frac{1}{n}\sum_{i=0}^{n-1}\frac{(X_{i+1} - \bar{A}_1(Z_i; \hat\theta^{(1)}_{n, \Delta_n}, \theta^{(2)},\sigma ))(Y_{i+1} - \bar{A}_2(Z_i;  \hat\theta^{(1)}_{n, \Delta_n}, \theta^{(2)}, \sigma))}{{\Delta_n^2} b^2(Z_i; \sigma)(\partial_{y} a_1(Z_i; \hat\theta^{(1)}_{n, \Delta_n} ))}\\
T_3 & = \frac{1}{n}\sum_{i=0}^{n-1}\frac{(Y_{i+1} - \bar{A}_2(Z_i;  \hat\theta^{(1)}_{n, \Delta_n}, \theta^{(2)}, \sigma))^2}{{\Delta_n} b^2(Z_i; \sigma) } \\ 
T_4 & = \frac{1}{n}\sum_{i=0}^{n-1} \log b^2(Z_i; \sigma)
\end{align*}
 {For the term $T_1$ we have:
\begin{multline*}
T_1 = \frac{1}{n{\Delta_n^3}}\sum_{i=0}^{n-1} \frac{1}{b^2(Z_i; \sigma)} \frac{ \left( X_{i+1} - \bar{A}_1(Z_i; \hat\theta^{(1)}_{n, \Delta_n}, \theta^{(2)},\sigma ) \right)^2}{(\partial_ya_1(Z_i; \hat\theta^{(1)}_{n, \Delta_n} ))^2} = \\   \frac{1}{n}\sum_{i=0}^{n-1} \frac{1}{b^2(Z_i; \sigma)} \frac{ \left( X_{i+1}  - \bar{A}_1(Z_i; \theta^{(1)}_0, \theta^{(2)}, \sigma) +\bar{A}_1(Z_i; \theta^{(1)}_0, \theta^{(2)}, \sigma)   - \bar{A}_1(Z_i; \hat\theta^{(1)}_{n, \Delta_n}, \theta^{(2)},\sigma ) \right)^2}{{\Delta_n^3}(\partial_ya_1(Z_i; \hat\theta^{(1)}_{n, \Delta_n} ))^2} =  \\ 
= \frac{1}{n} \sum_{i=0}^{n-1} \frac{1}{b^2(Z_i; \sigma)}  \left[ \frac{ \left( X_{i+1} - \bar{A}_1(Z_i; \theta^{(1)}_0, \theta^{(2)}, \sigma) \right)^2}{{\Delta_n^3}(\partial_ya_1(Z_i; \hat\theta^{(1)}_{n, \Delta_n} ))^2} \frac{(\partial_ya_1(Z_i; \theta^{(1)}_0))^2}{(\partial_ya_1(Z_i; \theta^{(1)}_0))^2} + \right. \\ \left.\frac{2\Delta_n\left(X_{i+1} - \bar{A}_1(Z_i; \theta^{(1)}_0, \theta^{(2)}, \sigma)\right)\left(a_1(Z_i;\theta^{(1)}_0)-a_1(Z_i;\hat\theta^{(1)}_{n, \Delta_n} ) \right)}{\Delta_n^3 (\partial_{y} a_1(Z_i; \hat\theta^{(1)}_{n, \Delta_n} ))^2 }  +  \frac{\Delta_n^2}{\Delta_n^3}\frac{\left(a_1(Z_i;\theta^{(1)}_0)-a_1(Z_i;\hat\theta^{(1)}_{n, \Delta_n} ) \right)^2}{b^2(Z_i; \sigma)(\partial_{y} a_1(Z_i; \hat\theta^{(1)}_{n, \Delta_n} ))^2} \right]
\end{multline*}
Thanks to the Lemmas \ref{lemma_bounds} and \ref{lemma_bounds_first}, we know that the second term of the sum converges to $0$ in probability, and for the first one we have:
\begin{multline*}
 \frac{1}{n} \sum_{i=0}^{n-1} \frac{1}{b^2(Z_i; \sigma)}  \frac{ \left( X_{i+1} - \bar{A}_1(Z_i; \theta^{(1)}_0, \theta^{(2)}, \sigma) \right)^2}{{\Delta_n^3}(\partial_ya_1(Z_i; \hat\theta^{(1)}_{n, \Delta_n} ))^2} \frac{(\partial_ya_1(Z_i; \theta^{(1)}_0))^2}{(\partial_ya_1(Z_i; \theta^{(1)}_0))^2}  \overset{\mathds{P}_0}{\longrightarrow}\\ \int \frac{b^2(z; \sigma_0)}{b^2(z; \sigma)}\frac{(\partial_ya_1(z; \theta^{(1)}_0))^2}{(\partial_ya_1(z; \hat\theta^{(1)}_{n, \Delta_n} ))^2} \nu_0(dz)
\end{multline*}
For the third term, we use the assumption (A5), and then obtain the convergence to $0$ in probability thanks to Theorem \ref{thm:consistency_and_asymptotics_theta1}, the continuous mapping theorem and Lemma \ref{lemma_kessler}: 
\begin{equation*}
\frac{2}{n} \sum_{i=0}^{n-1} \frac{\Delta_n^2}{\Delta_n^3}\frac{\left(a_1(Z_i;\theta^{(1)}_0)-a_1(Z_i;\hat\theta^{(1)}_{n, \Delta_n} ) \right)^2}{b^2(Z_i; \sigma) (\partial_ya_1(Z_i; \hat\theta^{(1)}_{n, \Delta_n} ))^2} = \frac{2}{n^2} \sum_{i=0}^{n-1} \frac{\left(\sqrt \frac{n}{\Delta_n}(\theta^{(1)}_0-\hat\theta^{(1)}_{n, \Delta_n}  )\right)^2 g^2(X_i)}{b^2(Z_i; \sigma) (\partial_ya_1(Z_i; \hat\theta^{(1)}_{n, \Delta_n} ))^2} \overset{\mathds{P}_0}{\longrightarrow} 0.
\end{equation*}
Then, $T_2$ decomposes as:
\begin{multline*}
\frac{1}{n}\sum_{i=0}^{n-1}\frac{(X_{i+1} - \bar{A}_1(Z_i;  \hat\theta^{(1)}_{n, \Delta_n}, \theta^{(2)}, \sigma))(Y_{i+1} - \bar{A}_2(Z_i;  \hat\theta^{(1)}_{n, \Delta_n}, \theta^{(2)}, \sigma))}{{\Delta_n^2} b^2(Z_i; \sigma)(\partial_{y} a_1(Z_i; \hat\theta^{(1)}_{n, \Delta_n} ))} = \\ \frac{1}{n}\sum_{i=0}^{n-1}\frac{(\partial_{y} a_1(Z_i; \theta^{(1)}_0))}{b^2(Z_i; \sigma)(\partial_{y} a_1(Z_i; \hat\theta^{(1)}_{n, \Delta_n} ))} \left[ \frac{(X_{i+1} - \bar{A}_1(Z_i; \theta^{(1)}_0, \theta^{(2)}, \sigma))(Y_{i+1} - \bar{A}_2(Z_i;  \hat\theta^{(1)}_{n, \Delta_n}, \theta^{(2)}_0, \sigma))}{{\Delta_n^2} (\partial_{y} a_1(Z_i; \theta^{(1)}_0))}  + \right. \\ \left. \frac{\Delta_n(X_{i+1} - \bar{A}_1(Z_i; \theta^{(1)}_0, \theta^{(2)}, \sigma))(a_2(Z_i; \theta^{(2)}_0)-a_2(Z_i; \theta^{(2)}))}{{\Delta_n^2} (\partial_{y} a_1(Z_i; \theta^{(1)}_0))} + \right. \\ \left.\frac{\Delta_n(Y_{i+1} - \bar{A}_2(Z_i;  \hat\theta^{(1)}_{n, \Delta_n}, \theta^{(2)}_0, \sigma))(a_1(Z_i; \theta^{(1)}_0)-a_1(Z_i; \hat\theta^{(1)}_{n, \Delta_n} ))}{{\Delta_n^2} (\partial_{y} a_1(Z_i; \theta^{(1)}_0))} + \right. \\ \left. \frac{\Delta_n^2(a_1(Z_i; \theta^{(1)}_0)-a_1(Z_i; \hat\theta^{(1)}_{n, \Delta_n} ))(a_2(Z_i; \theta^{(2)}_0)-a_2(Z_i; \theta^{(2)}))}{{\Delta_n^2} (\partial_{y} a_1(Z_i; \theta^{(1)}_0))}  \right]
\end{multline*}
Again, using Lemma  \ref{lemma_bounds_first}, we know that the second and the third terms are converging to $0$ in probability. For the first term, thanks to Lemma \ref{lemma_bounds} we have the following convergence: 
\begin{multline*}
\frac{1}{n}\sum_{i=0}^{n-1} \frac{(X_{i+1} - \bar{A}_1(Z_i; \theta^{(1)}_0, \theta^{(2)}, \sigma))(Y_{i+1} - \bar{A}_2(Z_i;  \hat\theta^{(1)}_{n, \Delta_n}, \theta^{(2)}_0, \sigma))(\partial_{y} a_1(Z_i; \theta^{(1)}_0))}{{\Delta_n^2} b^2(Z_i; \sigma)(\partial_{y} a_1(Z_i; \theta^{(1)}_0))(\partial_{y} a_1(Z_i; \hat\theta^{(1)}_{n, \Delta_n} ))}  \overset{\mathds{P}_0}{\longrightarrow} \\ \int \frac{b^2(z; \sigma_0)}{b^2(z; \sigma)}\frac{\partial_ya_1(z; \theta^{(1)}_0) }{\partial_ya_1(z; \hat\theta^{(1)}_{n, \Delta_n} )} \nu_0(dz)
\end{multline*}
Finally, we treat the last term:
\begin{multline*}
\frac{1}{n}\sum_{i=0}^{n-1} \frac{\Delta_n^2(a_1(Z_i; \theta^{(1)}_0)-a_1(Z_i; \hat\theta^{(1)}_{n, \Delta_n} ))(a_2(Z_i; \theta^{(2)}_0)-a_2(Z_i; \theta^{(2)}))}{\Delta_n^2b^2(Z_i; \sigma)(\partial_{y} a_1(Z_i; \hat\theta^{(1)}_{n, \Delta_n} ))}  
\end{multline*}
Using again the Lipschitz continuity of $a_1$, Theorem \ref{thm:consistency_and_asymptotics_theta1} and the Slutsky's theorem, we obtain a convergence to zero in probability for this term. }
$T_4$ converges in probability to $\int \log b^2(z; \sigma) \nu_0(dz)$ due to Lemma \ref{lemma_kessler}. 
Consider $T_3$:
\begin{multline*}
T_3 = \frac{1}{n \Delta_n}\sum_{i=0}^{n-1}\frac{1}{b^2(Z_i; \sigma)}\left[ (Y_{i+1} - \bar{A}_2(Z_i;  \hat\theta^{(1)}_{n, \Delta_n}, \theta^{(2)}_0, \sigma))^2 + \right. \\ \left. 2(Y_{i+1} - \bar{A}_2(Z_i;  \hat\theta^{(1)}_{n, \Delta_n}, \theta^{(2)}_0, \sigma))(\bar{A}_2(Z_i;  \hat\theta^{(1)}_{n, \Delta_n}, \theta^{(2)}_0, \sigma) - \bar{A}_2(Z_i;  \hat\theta^{(1)}_{n, \Delta_n}, \theta^{(2)}, \sigma)) + \right. \\ \left. (\bar{A}_2(Z_i;  \hat\theta^{(1)}_{n, \Delta_n}, \theta^{(2)}_0, \sigma) - \bar{A}_2(Z_i;  \hat\theta^{(1)}_{n, \Delta_n}, \theta^{(2)}, \sigma))^2\right] =  \frac{1}{n \Delta_n  }\sum_{i=0}^{n-1} \frac{(Y_{i+1} - \bar{A}_2(Z_i;  \hat\theta^{(1)}_{n, \Delta_n}, \theta^{(2)}_0, \sigma))^2}{b^2(Z_i; \sigma)} + \\ 2\frac{{\Delta_n}}{n {\Delta_n}  }\sum_{i=0}^{n-1} \frac{(Y_{i+1} - \bar{A}_2(Z_i;  \hat\theta^{(1)}_{n, \Delta_n}, \theta^{(2)}_0, \sigma))(a_2(Z_i; \theta^{(2)}_0) - a_2(Z_i; \theta^{(2)}))}{b^2(Z_i; \sigma)} + \\ \frac{{\Delta_n^2}}{n {\Delta_n}}\sum_{i=0}^{n-1} \frac{(a_2(Z_i; \theta^{(2)}_0) - a_2(Z_i; \theta^{(2)}))^2}{b^2(Z_i; \sigma)}
\end{multline*}
Thanks to Lemma \ref{lemma_bounds} and \ref{lemma_bounds_first} we conclude that 
\begin{equation*}
T_3 \overset{\mathds{P}_0}{\longrightarrow} \int \frac{b^2(z; \sigma_0)}{b^2(z; \sigma)} \nu_0(dz) +0 +0
\end{equation*}
Finally, we obtain
\begin{multline*}
\frac{1}{n} \mathcal{L}_{n, \Delta_n}(\theta, \sigma^2; Z_{0:n}) \overset{\mathds{P}_0}{\longrightarrow} \\ \int \left( \frac{b^2(z;\sigma_0)}{b^2(z; \sigma)}\left[ 3\left(\frac{\partial_ya_1(z; \theta^{(1)}_0)}{\partial_ya_1(z; \hat\theta^{(1)}_{n, \Delta_n} )}\right)^2  - 3\frac{\partial_ya_1(z; \theta^{(1)}_0) }{\partial_ya_1(z; \hat\theta^{(1)}_{n, \Delta_n} )} + 1 \right] + \log b^2(z; \sigma) \right) \nu_0(dz)
\end{multline*}
By assumption (A5) $\partial_ya_1(\cdot)$ does not depend  on $\theta^{(1)}$, thus $\frac{\partial_ya_1(z; \theta^{(1)}_0) }{\partial_ya_1(z; \hat\theta^{(1)}_{n, \Delta_n} )} = 1 $. 
It gives the Lemma.
\end{proof}

\begin{proof}[Theorem \ref{thm:consistency_theta2_sigma}]
\textbf{Consistency.} The consistency of the estimator for the parameter $\theta^{(2)}$ is based on Lemma \ref{lemma:theta_2}, with the arguments analogous to the proof of Theorem \ref{thm:consistency_and_asymptotics_theta1}. 
For the diffusion parameter $\sigma$, the result follows from Lemma \ref{lemma:sigma}. Denote $\mathcal{I}(\sigma, \sigma_0): =\frac{b^2(z;\sigma_0)}{b^2(z; \sigma)} + \log b^2(z; \sigma)$.
We can choose some subsequence $n_k$ such that $\hat\sigma_{n, \Delta_{n}}$ converges to some $\sigma_{\infty}$. By the definition of the estimator we know that $\mathcal{I}(\sigma_{\infty}, \sigma_0) \leq \mathcal{I}(\sigma_{0}, \sigma_0)$. But we also know that $\frac{b^2(z;\sigma_0)}{b^2(z; \sigma)} + \log b^2(z; \sigma) \geq 1 + \log b^2(z;\sigma_0)$ and thus $\mathcal{I}(\sigma_{\infty}, \sigma_0) \geq \mathcal{I}(\sigma_{0}, \sigma_0)$, and by the identifiability assumption $\sigma_{\infty} \equiv \sigma_0$. It proves the consistency of $\hat{\sigma}$. 

\textbf{Asymptotic normality.} The proof follows the standard pattern. Throughout the proof we assume that $\theta^{(2)} \text{ and } \sigma \in \mathds{R}$ in order to simplify the notations.  We write the Taylor expansion of the contrast function defined in \eqref{loglikelihood} and apply an appropriate scaling 
\begin{equation*}
\int C_{n, \Delta_n}\left(\varphi_0 + u(\hat\varphi_{n, \Delta_n} - \varphi_0); z \right)du \: E_{n, \Delta_n} =\\ - D_{n, \Delta_n}(\varphi_0),
\end{equation*}
where by $\varphi$ we now denote $(\theta^{(2)}, \sigma)$ and the parameter $\theta^{(1)}$ is fixed to its estimate $\hat\theta^{(1)}_{n,\Delta_n}$ throughout the proof, and 
\begin{align*}
C_{n, \Delta_n}(\theta) & := \left[\begin{matrix} \frac{1}{n\Delta_n} \frac{\partial^2}{\partial  \theta^{(2)} \partial  \theta^{(2)}} \mathcal{L}_{n, \Delta_n}(\hat\theta^{(1)}_{n,\Delta_n},\theta^{(2)}, \sigma; Z_{0:n}) & \frac{1}{n\sqrt{\Delta_n}}\frac{\partial^2}{\partial  \sigma \partial  \theta^{(2)}} \mathcal{L}(\hat\theta^{(1)}_{n,\Delta_n},\theta^{(2)}, \sigma; Z_{0:n}) \\
\frac{1}{n\sqrt{\Delta_n}}\frac{\partial^2}{\partial  \theta^{(2)}\partial  \sigma } \mathcal{L}_{n, \Delta_n}(\hat\theta^{(1)}_{n,\Delta_n},\theta^{(2)}, \sigma; Z_{0:n}) &  \frac{1}{n}
\frac{\partial^2}{\partial  \sigma \partial  \sigma }\mathcal{L}_{n, \Delta_n}(\hat\theta^{(1)}_{n,\Delta_n},\theta^{(2)}, \sigma; Z_{0:n})\end{matrix}\right], \\
E_{n, \Delta_n}:= & \left[\begin{matrix} \sqrt{n\Delta_n} (\hat\theta^{(2)}_n - \theta^{(2)}_0) \\
 \sqrt{n} (\hat\sigma_n - \sigma_0)  \end{matrix} \right], 
\quad D_{n, \Delta_n} = \left[\begin{matrix} 
\frac{1}{\sqrt{n\Delta_n}} \frac{\partial}{\partial  \theta^{(2)} } \mathcal{L}_{n, \Delta_n}(\hat\theta^{(1)}_{n,\Delta_n},\theta^{(2)}, \sigma; Z_{0:n}) \\ 
\frac{1}{\sqrt{n}} \frac{\partial}{\partial  \sigma} \mathcal{L}_{n, \Delta_n}(\hat\theta^{(1)}_{n,\Delta_n},\theta^{(2)}, \sigma; Z_{0:n}) \end{matrix}\right].
\end{align*}
First, we compute the higher-order terms of the partial derivatives of first and second order with respect to $\theta^{(2)}$ and $\sigma$:
\begin{multline*}
\frac{\partial}{\partial \theta^{(2)}}\mathcal{L}_{n, \Delta_n}(\cdot) = \sum_{i=1}^{n-1} \left[ - 6 \frac{\Delta_n(\partial_{\theta^{(2)}}a_2) (X_{i+1} - \bar{A}_1(Z_i; \hat\theta^{(1)}_{n,\Delta_n}, \theta^{(2)}, \sigma))}{\Delta_n^2b^2(Z_i; \sigma)(\partial_{\hat\theta^{(1)}_{n,\Delta_n}}a_1)}  +\right. \\ \left. 2\frac{\Delta_n(\partial_{\theta^{(2)}}a_2) (Y_{i+1} - \bar{A}_2(Z_i; \hat\theta^{(1)}_{n,\Delta_n}, \theta^{(2)}, \sigma))}{\Delta_n b^2(Z_i; \sigma)} \right] =: D^{1}_{n, \Delta_n} 
\end{multline*}
\begin{gather*}
\frac{\partial}{\partial \sigma}\mathcal{L}_{n, \Delta_n}(\cdot) = - \sum_{i=1}^{n-1}\frac{\partial_\sigma b}{b^3(Z_i; \sigma)} \left[ 6 \frac{ (X_{i+1} - \bar{A}_1(Z_i; \hat\theta^{(1)}_{n,\Delta_n}, \theta^{(2)}, \sigma))^2}{\Delta_n^3 (\partial_{\hat\theta^{(1)}_{n,\Delta_n}}a_1)^2} - \right. \\ \left.  6\frac{ (X_{i+1} - \bar{A}_1(Z_i; \hat\theta^{(1)}_{n,\Delta_n}, \theta^{(2)}, \sigma)( Y_{i+1} - \bar{A}_2(Z_i; \hat\theta^{(1)}_{n,\Delta_n}, \theta^{(2)}, \sigma))}{\Delta_n^2 (\partial_{\hat\theta^{(1)}_{n,\Delta_n}}a_1)} + \right. \\ \left.  2\frac{ (Y_{i+1} - \bar{A}_2(Z_i; \hat\theta^{(1)}_{n,\Delta_n}, \theta^{(2)}, \sigma))^2}{\Delta_n}  \right] + \frac{\partial_\sigma b}{b(Z_i; \sigma)} =: D^2_{n, \Delta_n} \\ 
\frac{\partial^2}{\partial \theta^{(2)}\partial \theta^{(2)}}\mathcal{L}_{n, \Delta_n}(\cdot) = \sum_{i=1}^{n-1} \left[ - 6 \frac{\Delta_n(\partial^2_{\theta^{(2)}\theta^{(2)}}a_2) (X_{i+1} - \bar{A}_1(Z_i; \hat\theta^{(1)}_{n,\Delta_n}, \theta^{(2)}, \sigma))}{\Delta_n^2b^2(Z_i; \sigma)(\partial_{\theta^{(1)}}a_1)}  + \right. \\  \left.  2\frac{\Delta_n(\partial^2_{\theta^{(2)}\theta^{(2)}}a_2) (Y_{i+1} - \bar{A}_2(Z_i; \hat\theta^{(1)}_{n,\Delta_n}, \theta^{(2)}, \sigma))}{\Delta_n b^2(Z_i; \sigma)} + \frac{\Delta_n^2(\partial_{\theta^{(2)}}a_2)^2}{\Delta_n b^2(Z_i; \sigma)} \right] =: C^{11}_{n, \Delta_n}   \\ 
\frac{\partial^2}{\partial \theta^{(2)}\partial\sigma}\mathcal{L}_{n, \Delta_n}(\cdot) = \sum_{i=1}^{n-1}\frac{\partial_\sigma b}{b^2(Z_i; \sigma)} \left[  12 \frac{\Delta_n(\partial_{\theta^{(2)}}a_2) (X_{i+1} - \bar{A}_1(Z_i; \hat\theta^{(1)}_{n,\Delta_n}, \theta^{(2)}, \sigma))}{\Delta_n^2b(Z_i; \sigma)(\partial_{\theta^{(1)}}a_1)}  +  \right. \\ \left. 4\frac{\Delta_n(\partial_{\theta^{(2)}}a_2)  (Y_{i+1} - \bar{A}_2(Z_i; \hat\theta^{(1)}_{n,\Delta_n}, \theta^{(2)}, \sigma))}{\Delta_n b(Z_i; \sigma)} \right] =: C^{12}_{n, \Delta_n}=C^{21}_{n, \Delta_n}
 \end{gather*}
\begin{gather*}
\frac{\partial^2}{\partial \sigma^2}\mathcal{L}_{n, \Delta_n}(\cdot) = - \sum_{i=1}^{n-1}\frac{6(\partial_\sigma b)^2 - 2b(Z_i; \sigma)(\partial^2_{\sigma\sigma} b) }{b^4(Z_i; \sigma)} \left[ 6 \frac{ (X_{i+1} - \bar{A}_1(Z_i; \hat\theta^{(1)}_{n,\Delta_n}, \theta^{(2)}, \sigma))^2}{\Delta_n^3 (\partial_{\theta^{(1)}}a_1)^2} -\right. \\  \left. 6\frac{ (X_{i+1} - \bar{A}_1(Z_i; \hat\theta^{(1)}_{n,\Delta_n}, \theta^{(2)}, \sigma))( Y_{i+1} - \bar{A}_2(Z_i; \hat\theta^{(1)}_{n,\Delta_n}, \theta^{(2)}, \sigma))}{\Delta_n^2 (\partial_{\theta^{(1)}}a_1)} +  2\frac{ (Y_{i+1} - \bar{A}_2(Z_i; \hat\theta^{(1)}_{n,\Delta_n}, \theta^{(2)}, \sigma))^2}{\Delta_n}  \right] + \\ 2\frac{b(Z_i;\sigma)(\partial^2_{\sigma\sigma} b)-(\partial_\sigma b)^2}{b^2(Z_i; \sigma)} =: C^{22}_{n, \Delta_n}
\end{gather*}
We start with proving the convergence for the terms $C_{n, \Delta_n}$. Then we can obtain a convergence in probability after few technical steps. We start with $C^{11}_{n, \Delta_n}$:
\begin{multline*}
\frac{1}{n\Delta_n}C^{11}_{n, \Delta_n} = \frac{1}{n\Delta_n}\sum_{i=1}^{n-1} \left[ - 6 \frac{\Delta_n(\partial^2_{\theta^{(2)}\theta^{(2)}}a_2) (X_{i+1} - \bar A_1(Z_i; \hat\theta^{(1)}_{n,\Delta_n}, \theta^{(2)}, \sigma))}{\Delta_n^2b^2(Z_i; \sigma)(\partial_{\theta^{(1)}}a_1)}  + \right. \\ \left.  2\frac{\Delta_n(\partial^2_{\theta^{(2)}\theta^{(2)}_0}a_2) (Y_{i+1} - \bar{A}_2(Z_i; \hat\theta^{(1)}_{n,\Delta_n}, \theta^{(2)}, \sigma))}{\Delta_n b^2(Z_i; \sigma)} + \frac{\Delta_n^2(\partial_{\theta^{(2)}}a_2)^2}{\Delta_n b^2(Z_i; \sigma)} \right] = \\ 
\frac{1}{n\Delta_n}\sum_{i=1}^{n-1} \left[ - 6 \frac{\Delta_n(\partial^2_{\theta^{(2)}\theta^{(2)}}a_2) (X_{i+1} - \bar{A}_1(Z_i; \theta^{(1)}_0, \theta^{(2)}, \sigma))}{\Delta_n^2b^2(Z_i; \sigma)(\partial_{\theta^{(1)}}a_1)} - 6 \frac{\Delta_n^2(\partial^2_{\theta^{(2)}\theta^{(2)}}a_2) (a_1(Z_i;\theta^{(1)}_0)  - a_1(Z_i; \hat\theta^{(1)}_{n,\Delta_n}))}{\Delta_n^2b^2(Z_i; \sigma)(\partial_{\theta^{(1)}}a_1)} \right. \\ \left.  2\frac{\Delta_n(\partial^2_{\theta^{(2)}\theta^{(2)}}a_2) (Y_{i+1} - \bar{A}_2(Z_i; \hat\theta^{(1)}_{n,\Delta_n}, \theta^{(2)}, \sigma))}{\Delta_n b^2(Z_i; \sigma)} + \frac{\Delta_n^2(\partial_{\theta^{(2)}}a_2)^2}{\Delta_n b^2(Z_i; \sigma)} \right]
\end{multline*}
Note that thanks to Lemma \ref{lemma_bounds_first} we know that 
\begin{multline*}
\frac{1}{n\Delta_n}\sum_{i=1}^{n-1} \left[ - 6 \frac{\Delta_n(\partial^2_{\theta^{(2)}\theta^{(2)}}a_2) (X_{i+1} - \bar{A}_1(Z_i;\theta^{(1)}_0, \theta^{(2)}, \sigma))}{\Delta_n^2b^2(Z_i; \sigma)(\partial_{\theta^{(1)}}a_1)} +\right. \\ \left. 2\frac{\Delta_n(\partial^2_{\theta^{(2)}\theta^{(2)}}a_2) (Y_{i+1} - \bar{A}_2(Z_i; \hat\theta^{(1)}_{n,\Delta_n}, \theta^{(2)}, \sigma))}{\Delta_n b^2(Z_i; \sigma)} + \right. \\ \left. \frac{\Delta_n^2(\partial_{\theta^{(2)}}a_2)^2}{\Delta_n b^2(Z_i; \sigma)} \right] \overset{\mathds{P}_0}{\longrightarrow} \int \frac{(\partial_{\theta^{(2)}}a_2)^2}{b^2(z; \sigma)}\nu_0(dz)
\end{multline*}
What about the remaining term, thanks to the assumption (A5) we have:
\begin{multline*}
-\frac{6}{n\Delta_n}\sum_{i=1}^{n-1}\frac{(\partial^2_{\theta^{(2)}\theta^{(2)}}a_2) (a_1(Z_i;\theta^{(1)}_0)  - a_1(Z_i; \hat\theta^{(1)}_{n,\Delta_n}))}{b^2(Z_i; \sigma)(\partial_{\theta^{(1)}}a_1)} =  \\ - \frac{6 }{\sqrt{n\Delta_n}}\frac{1}{n}\sum_{i=1}^{n-1}\frac{(\partial^2_{\theta^{(2)}\theta^{(2)}}a_2)a_1(Z_i;\sqrt{\frac{n}{\Delta_n}}\left(\theta^{(1)}_0  - \hat\theta^{(1)}_{n,\Delta_n}\right) ) }{b^2(Z_i; \sigma)(\partial_{\theta^{(1)}}a_1)} 
\end{multline*}
We know that $\left(\theta^{(1)}_0  - \hat\theta^{(1)}_{n,\Delta_n}\right)\sqrt{\frac{n}{\Delta_n}}$ is normally distributed by Theorem \ref{thm:consistency_and_asymptotics_theta1}, and $ \frac{1}{n}\sum_{i=1}^{n-1}\frac{(\partial^2_{\theta^{(2)}\theta^{(2)}}a_2) }{b^2(Z_i; \sigma)(\partial_{\theta^{(1)}}a_1)} $ converges to its invariant density by Lemma \ref{lemma_kessler}. Then by Slutsky's and the continuous mapping theorem the product also converges in distribution to a normal variable, which is, divided by $\sqrt{n\Delta_n}$ converges to zero since $n\Delta_n\to\infty$ by design. However, as $n\Delta_n\to\infty$, this term converges to $0$ in probability. As a result, 
\begin{equation*}
\frac{1}{n\Delta_n}C^{11}_{n, \Delta_n} \overset{\mathds{P}_0}{\longrightarrow} \int \frac{(\partial_{\theta^{(2)}}a_2)^2}{b^2(z; \sigma)}\nu_0(dz)
\end{equation*}
With the same arguments we prove that $\frac{1}{n\sqrt{\Delta_n}}C^{12}_{n, \Delta_n}=\frac{1}{n\sqrt{\Delta_n}}C^{21}_{n, \Delta_n} \overset{\mathds{P}_0}{\longrightarrow} 0$ and that
\begin{equation*}
\frac{1}{n}C^{22}_{n, \Delta_n}\overset{\mathds{P}_0}{\longrightarrow} - 4 \int \frac{(\partial_\sigma b)^2 }{b^2(z; \sigma_0)} \nu_0(dz)
\end{equation*} 
Then we consider the remaining term, recalling the assumption (A5):
We start with the term 
\begin{multline*}
\frac{1}{\sqrt{n\Delta_n}} D^{1}_{n, \Delta_n} = \frac{1}{\sqrt{n\Delta_n}}\sum_{i=1}^{n-1} \left[ - 6 \frac{\Delta_n(\partial_{\theta^{(2)}}a_2) (X_{i+1} - \bar A_1(Z_i; \hat\theta^{(1)}_{n,\Delta_n}, \theta^{(2)}, \sigma))}{\Delta_n^2b^2(Z_i; \sigma)(\partial_{\theta^{(1)}}a_1)}  + \right. \\ \left. 2\frac{\Delta_n(\partial_{\theta^{(2)}}a_2) (Y_{i+1} - \bar{A}_2(Z_i; \hat\theta^{(1)}_{n,\Delta_n}, \theta^{(2)}, \sigma))}{\Delta_n b^2(Z_i; \sigma)} \right] = 
\frac{1}{\sqrt{n\Delta_n}}\sum_{i=1}^{n-1} \left[ - 6 \frac{(\partial_{\theta^{(2)}}a_2) (X_{i+1} - \bar{A}_1(Z_i; \theta^{(1)}_0, \theta^{(2)}, \sigma))}{\Delta_n b^2(Z_i; \sigma)(\partial_{\theta^{(1)}}a_1)} -  \right. \\ \left.6 \frac{(\partial_{\theta^{(2)}}a_2) a_1(Z_i; \hat\theta^{(1)}_{n,\Delta_n}- \theta^{(1)}_0) }{b^2(Z_i; \sigma)(\partial_{\theta^{(1)}}a_1)}  + 2\frac{(\partial_{\theta^{(2)}}a_2) (Y_{i+1} - \bar{A}_2(Z_i; \hat\theta^{(1)}_{n,\Delta_n}, \theta^{(2)}, \sigma))}{b^2(Z_i; \sigma)} \right] 
\end{multline*}
For the first and the third term we simply apply Lemma \ref{lemma:convergence_in_dist} and obtain convergence in distribution to $\mathcal{N}\left(0, \nu_0\left(\frac{(\partial_{\theta^{(2)}}a_2)^2}{b^2(z; \sigma_0)} \right) \right)$.
For the second term we apply the result of Theorem \ref{thm:consistency_and_asymptotics_theta1}, as well as the continuous mapping and Slutsky's theorem we may state that:
\[
-6\sum_{i=1}^{n-1} \frac{(\partial_{\theta^{(2)}}a_2) \: a_1(Z_i; \sqrt\frac{n}{\Delta_n}(\hat\theta^{(1)}_{n,\Delta_n}- \theta^{(1)}_0)) }{b^2(Z_i; \sigma)(\partial_{\theta^{(1)}}a_1)} \overset{\mathcal{D}}{\longrightarrow} -6 \int \frac{(\partial_{\theta^{(2)}}a_2) }{b^2(z; \sigma)(\partial_{\theta^{(1)}}a_1)}a_1(z; \tilde\eta)\nu_0(dz), 
\]
where $\tilde\eta$ is distributed as stated in Theorem \ref{thm:consistency_and_asymptotics_theta1}. Then as $n\to 0$, 
\[
-\frac{6}{n}\sum_{i=1}^{n-1} \frac{(\partial_{\theta^{(2)}}a_2) \: a_1(Z_i; \sqrt\frac{n}{\Delta_n}(\hat\theta^{(1)}_{n,\Delta_n}- \theta^{(1)}_0)) }{b^2(Z_i; \sigma)(\partial_{\hat\theta^{(1)}_{n,\Delta_n}}a_1)} \overset{\mathds{P}_0}{\longrightarrow} 0
\]
By analogy, we prove the convergence for the term $D^2_{n, \Delta_n}$, obtaining: 
\[
\frac{1}{\sqrt n}D^2_{n, \Delta_n} \overset{\mathcal{D}}{\longrightarrow} \mathcal{N}\left(0, 32\nu_0\left(\frac{(\partial_\sigma b)^2}{b^2(z; \sigma_0)} \right) \right)
\]
That gives the result. 
\end{proof}

\subsection{Consistency and asymptotic normality of the least squares contrast}\label{subsec:consistency_lse}

\begin{proof}[Theorem \ref{thm:drift_contrast}]
 The proof will follow the one of the classical contrast. First, we define the following quantities:
\begin{gather*}
\mathcal{L}^{(1),LSE}_{n,\Delta_n}(\theta^{(1)}, \theta^{(2)}, \sigma;Z_{0:n}) = \frac{1}{n} \sum_{i=0}^{n-1} \frac{(X_{i+1} - \bar{A}_1(Z_i; \theta^{(1)}, \theta^{(2)}, \sigma))^2}{\Delta^3_n} \\ 
\mathcal{L}^{(2),LSE}_{n,\Delta_n}(\theta^{(1)}, \theta^{(2)}, \sigma;Z_{0:n}) = \frac{1}{n} \sum_{i=0}^{n-1} \frac{(Y_{i+1} - \bar{A}_2(Z_i; \theta^{(1)}, \theta^{(2)}, \sigma))^2}{\Delta_n}
\end{gather*}
\textbf{Consistency of $\hat \theta^{(1)}$}.
First, consider:
\begin{multline*}
\Delta_n\left[\mathcal{L}^{LSE}_{n,\Delta_n}(\theta^{(1)}, \theta^{(2)}; Z_{0:n})-\mathcal{L}^{LSE}_{n,\Delta_n}(\theta^{(1)}_0, \theta^{(2)}; Z_{0:n})\right]  =  \\ 
\frac{\Delta_n}{n\Delta^3_n}\sum_{i=0}^{n-1} \left[
(X_{i+1} - \bar{A}_1(Z_i; \theta^{(1)}_0, \theta^{(2)}, \sigma) + \bar{A}_1(Z_i; \theta^{(1)}_0, \theta^{(2)}, \sigma) - \right. \\ \left. \bar{A}_1(Z_i; \theta^{(1)}, \theta^{(2)}, \sigma))^2  - (X_{i+1} - \bar{A}_1(Z_i; \theta^{(1)}_0, \theta^{(2)}, \sigma))^2 \right] = \\
\frac{\Delta^2_n}{n\Delta^3_n}\sum_{i=0}^{n-1} \left[ 2(X_{i+1} - \bar{A}_1(Z_i; \theta^{(1)}_0, \theta^{(2)}, \sigma))(a_1(Z_i; \theta^{(1)}_0) -a_1(Z_i; \theta^{(1)})) + \right. \\ \left.\Delta_n (a_1(Z_i; \theta^{(1)}_0) - a_1(Z_i; \theta^{(1)}))^2 + O(\Delta^2_n) \right]
\end{multline*}
Then we have from Lemmas \ref{lemma_bounds}, \ref{lemma_bounds_first}:
\begin{gather*}
 \frac{2 }{n\Delta_n} \sum_{i=0}^{n-1} (X_{i+1} - \bar{A}_1(Z_i; \theta^{(1)}_0, \theta^{(2)}, \sigma))(a_1(Z_i; \theta^{(1)}_0) - a_1(Z_i; \theta^{(1)})) \overset{\mathds{P}_0}{\longrightarrow} 0 \\
 \frac{1}{n} \sum_{i=0}^{n-1} (a_1(Z_i; \theta^{(1)}_0) - a_1(Z_i; \theta^{(1)}))^2 \overset{\mathds{P}_0}{\longrightarrow}\int (a_1(z; \theta^{(1)}_0) - a_1(z; \theta^{(1)}))^2 \nu_0(dz)
 \end{gather*}
 We conclude that there exists a subsequence $\hat\theta^{(1)}_{n,{\Delta_n}} = \underset{\theta}{\arg\min} \: \mathcal{L}^{LSE}_{n,\Delta_n}(\theta;Z_{0:n}) $ that tends to $\theta_\infty$. Since the minimum is attained at the point $\theta_0$ and from (A4), we conclude that $\theta_\infty = \theta_0$. Hence the estimator is consistent.
\\
\textbf{Consistency of $\hat \theta^{(2)}$.} Consider:
\begin{multline*}
\frac{1}{\Delta_n}\left[\mathcal{L}^{LSE}_{n,\Delta_n}(\theta^{(1)}, \theta^{(2)}; Z_{0:n})-\mathcal{L}^{LSE}_{n,\Delta_n}(\theta^{(1)}, \theta^{(2)}_0; Z_{0:n})\right]  =  \\ 
\left[
\frac{1}{n\Delta^2_n}\sum_{i=0}^{n-1} (Y_{i+1} - \bar{A}_2(Z_i; \theta^{(1)}, \theta^{(2)}_0, \sigma) + \bar{A}_2(Z_i; \theta^{(1)}, \theta^{(2)}_0, \sigma) - \bar{A}_2(Z_i; \theta^{(1)}, \theta^{(2)}, \sigma))^2  - \right. \\ \left (Y_{i+1} - \bar{A}_2(Z_i; \theta^{(1)}, \theta^{(2)}_0, \sigma))^2 \right]= 
\frac{\Delta_n}{n\Delta^2_n}\sum_{i=0}^{n-1} \left[ 2(Y_{i+1} - \bar{A}_2(Z_i; \theta^{(1)}, \theta^{(2)}_0, \sigma))(a_2(Z_i; \theta^{(2)}_0) - a_2(Z_i; \theta^{(2)})) +\right. \\ \left.  \Delta_n(a_2(Z_i; \theta^{(2)}_0) - a_2(Z_i; \theta^{(2)}))^2 + O(\Delta^2_n) \right]
\end{multline*}
Thanks to Lemmas \ref{lemma_kessler}, \ref{lemma_bounds}:
 \begin{gather*}
 \frac{2 {\Delta_n}}{n{\Delta_n^2}} \sum_{i=0}^{n-1} (Y_{i+1} - \bar{A}_2(Z_i; \theta^{(1)}, \theta^{(2)}_0, \sigma))(a_2(Z_i; \theta^{(2)}_0) - a_2(Z_i; \theta^{(2)})) \overset{\mathds{P}_0}{\longrightarrow} 0\\
 \frac{{\Delta_n^2}}{n{\Delta_n^2}} \sum_{i=0}^{n-1} (a_2(Z_i; \theta^{(2)}_0) - a_2(Z_i; \theta^{(2)}))^2 \overset{\mathds{P}_0}{\longrightarrow} \int (a_2(z; \theta^{(2)}_0) - a_2(z; \theta^{(2)}))^2 \nu_0(dz)
\end{gather*} 
The consistency is concluded following the same arguments as in the case of $\theta^{(1)}$.

\textbf{Asymptotic normality.}
We apply again a Taylor formula for a function \eqref{equation:drift_estimation_qv}:
\begin{equation*}
\int C_n\left(\theta_0 + u(\hat\theta_n - \theta_0 \right))du\: E_n = D_n(\theta_0),
\end{equation*}
where we define 
\begin{align*}
C_n(\theta) & := \left[\begin{matrix} \frac{\Delta_n}{n}\frac{\partial^2}{\partial  \theta^{(1)} \partial  \theta^{(1)}} \mathcal{L}^{LSE}_{n,\Delta_n}(\theta;Z_{0:n})& 
\frac{1}{n}\frac{\partial^2}{\partial  \theta^{(1)} \partial  \theta^{(2)}} \mathcal{L}^{LSE}_{n,\Delta_n}(\theta;Z_{0:n}) \\
\frac{1}{n }\frac{\partial^2}{\partial  \theta^{(1)} \partial  \theta^{(2)}} \mathcal{L}^{LSE}_{n,\Delta_n}(\theta;Z_{0:n}) &  
\frac{1}{n \Delta_n}\frac{\partial^2}{\partial  \theta^{(2)} \partial  \theta^{(2)}} \mathcal{L}^{LSE}_{n,\Delta_n}(\theta;Z_{0:n})\end{matrix}\right], \\
E_n:= & \left[\begin{matrix}\sqrt{\frac{n}{\Delta_n}} (\hat\theta^{(1)}_{n, \Delta_n}  - \theta^{(1)}_0) \\
\sqrt{n\Delta_n} (\hat\theta^{(2)}_n - \theta^{(2)}_0)  \end{matrix} \right], 
\quad D_n (\theta) = \left[\begin{matrix} 
\frac{\sqrt{\Delta_n}}{n}\frac{\partial}{\partial  \theta^{(1)} } \mathcal{L}^{LSE}_{n,\Delta_n}(\theta;Z_{0:n}) \\ 
\frac{1}{n\sqrt\Delta_n}\frac{\partial}{\partial  \theta^{(2)} } \mathcal{L}^{LSE}_{n,\Delta_n}(\theta;Z_{0:n}) \end{matrix}\right].
\end{align*}
Using Lemma \ref{lemma:convergence_in_dist} we get:
\begin{equation*}
D_n(\theta_0)  \overset{\mathcal{D}}{\longrightarrow} -2 \mathcal{N}\left(0,I_{2}\cdot \left[\begin{matrix}  {1\over 3}\int b^2(z;\sigma_0) (\partial_ya_1(z; \theta^{(1)}_0))^2 (\partial_{\theta^{(1)}}a_1(z; \theta^{(1)}_0))^2  \nu_0(dz) \\  \int b^2(z;\sigma_0)  (\partial_{\theta^{(2)}}a_2(z; \theta^{(2)}_0))^2  \nu_0(dz)  \end{matrix} \right]  \right) ,
\end{equation*}
where $I_{2}$ is $2\times 2$ identity matrix. 
And by Lemmas \ref{lemma_bounds}, \ref{lemma_kessler} we have the result for $C_n(\theta)$:
\begin{equation*}
C_n(\theta_0)  \overset{\mathds{P}_0}{\longrightarrow} -2 \left[\begin{matrix} 
\int (\partial_{\theta^{(1)}}a_1(z; \theta^{(1)}_0))^2 \nu_0(dz) & 0  \\ 
0 & \int (\partial_{\theta^{(2)}}a_2(z; \theta^{(2)}_0))^2 \nu_0(dz) 
\end{matrix} \right].
\end{equation*}
That, in the combination with the consistency result, gives the theorem.
 \end{proof}

\end{document}